\documentclass[11pt]{amsart}
\usepackage{amsmath}
\allowdisplaybreaks
\usepackage{amssymb,latexsym}
\usepackage[matrix,arrow]{xy}
\usepackage[usenames]{color}
\usepackage{slashed}
\usepackage{parskip}
\usepackage{amsmath,amssymb,graphics}
\usepackage{mathrsfs}
\usepackage{graphicx}

\newcommand{\mathsym}[1]{{}}

\newtheorem{thm}{Theorem}[section]
\newtheorem{cor}[thm]{Corollary}
\newtheorem{lem}[thm]{Lemma}
\newtheorem{prop}[thm]{Proposition}
\newtheorem{question}[thm]{Question}

\theoremstyle{definition}
\newtheorem{defn}{Definition}[section]
\numberwithin{equation}{section}
\theoremstyle{remark}

\theoremstyle{example}

\newcommand{\md }{\mathrm{d}}

\newcommand{\be}{\begin{equation}}
\newcommand{\ee}{\end{equation}}
\newcommand{\bag}{\begin{eqnarray}}
\newcommand{\eag}{\end{eqnarray}}
\newcommand{\ban}{\begin{eqnarray*}}
\newcommand{\ean}{\end{eqnarray*}}
\newcommand{\ba}{\begin{aligned}}
\newcommand{\ea}{\end{aligned}}

\newcommand{\bpf}{\begin{proof} }
\newcommand{\epf}{\end{proof} }

\newcommand{\ric}{\mathrm{Ric}}
\newcommand{\tr}{\mathrm{tr}}

\newcommand{\mn}{\sqrt{-1}}

\begin{document}

\title{ Transverse Fully Nonlinear Equations on Sasakian Manifolds and Applications}
\author{Ke Feng$^*$}
\address{School of Mathematical Sciences, Peking University, Beijing 100871, China}
\email{kefeng@math.pku.edu.cn}
\author{Tao Zheng$^{**}$}
\address{School of Mathematics and Statistics, Beijing Institute of Technology, Beijing 100081, China}
\email{zhengtao08@amss.ac.cn}
\thanks {$^{*}$Supported by National Natural Science Foundation of China grant No. 11901009.}
\thanks {$^{**}$Supported by Beijing Institute of Technology Research Fund Program for Young Scholars.}
\maketitle
\begin{abstract}
We prove a priori estimates for a class of transverse fully nonlinear equations on Sasakian manifolds and give some geometric applications such as the transversion Calabi-Yau theorem for transverse balanced and (strongly) Gauduchon metrics. We also explain that similar results hold on compact oriented, taut, transverse Hermitian foliated manifold of complex co-dimension $n$, and give some geometric applications such as the transverse Calabi-Yau theorems for transverse Hermitian and (strongly) Gauduchon metrics.
\end{abstract}

\section{Introduction}\label{secintroduction}
Sasakian manifold was first introduced by Sasaki  \cite{sasaki1960} in 1960s. It is the odd dimension counterpart of K\"ahler manifold, just as the contact manifolds is  substitutes for symplectic manifolds.  Sasakian manifold has received a lot of attention because it is the natural intersection of CR, contact and Riemannian geometry, and plays a very important role in Riemannian \& algebraic geometry and in Physics.  Sasakian manifolds first appeared in String theory in \cite{maldacena1998}.  Sasaki-Einstein metric is useful in Ads/CFT correspondence (see the detailed survey paper \cite{sparkssasaki} the references therein).
Boyer-Galicki  \cite{sasakigeo} includes a series of papers and references about various differential geometric aspects of Sasakian manifolds.  We can find transverse counterparts on Sasakian manifolds of the famous results in K\"ahler manifolds, such as the transverse Calabi-Yau theorem \cite{elk90}  (see also \cite{bg01,swz10}), the existence of canonical metrics on Sasakian manifolds \cite{fow09}, Sasaki-Einstein metrics and K-(semi-)stability on Sasakian manifolds \cite{cs15,cs18},   the Frankel conjecture on Sasakian manifolds \cite{hs15,hs16}, the Uhlenbeck-Yau theorem \cite{uhlenbeckyau} about the existence of Hermitian-Einstein structure \cite{bs2010}, foliated Hitchin-Kobayashi correspondence \cite{baragliahekmati2018}, and  the geometric pluripotential theory \cite{hl18} on Sasakian manifolds. There are also many other results about Sasakian manifolds in \cite{bgk05,gmsw04,msy08,chl2018,futakizhang2018} and references therein.

Motivated by El Kacimi-Alaoui \cite{elk90},  we consider a class of transverse fully nonlinear equations on a compact Sasaki manifold
 $(M,\phi,\xi,\eta,g)$ with $\dim_{\mathbb{R}}M=2n+1$ $($$n\geq 2$$)$ and $\omega_\dag=\frac{1}{2}\mathrm{d}\eta=g(\phi\cdot,\cdot)$ as its transverse K\"ahler form. Note that $\omega_\dag$ determines uniquely a transverse K\"aher metric $g_\dag:=\omega_\dag(\cdot,\phi\cdot)$ and hence we will not distinguish the two terms in the following. All the terms  in this section can be found in Section \ref{secsasaki}.

Let us fix a basic real $(1,1)$ form $\beta_\dag$ (see \eqref{basicformdefn} for the definition). Then for any basic function  $u\in C_{\mathrm{B}}^2(M, \mathbb{R})$, we define a new basic real $(1,1)$ form
\begin{equation}
\label{defnhdag}
h_\dag=\beta_\dag+\partial_{\mathrm{B}}\overline{\partial}_{\mathrm{B}}u,
\end{equation}
and get a transverse Hermitian endomorphism  of $(\nu(\mathcal{F}_\xi),\mathbf{I})$
with respect to $\sigma^*g$ denoted, here and hereafter, by $A_{\dag,u},$ where 
$\nu(\mathcal{F}_\xi)$ is the normal bundle of the foliation $\mathcal{F}_\xi,$ and $\sigma$, $C_{\mathrm{B}}^k(M, \mathbb{R})$ with $k\in \mathbb{N}^*\cup\{\infty\}$,  $\partial_{\mathrm{B}}$(and $\bar\partial_{\mathrm{B}}$) \& $\mathbf{I}$ are given by  \eqref{sigma}, \eqref{ckbfunction}, \eqref{defndb} and \eqref{ti} respectively. That is, there holds
\begin{equation*}
g_\dag\left(\phi\left(A_{\dag,u}(\sigma(W))\right), \sigma(V)\right)=h_\dag(\sigma(W),\sigma(V)),\quad \forall \;W,\,V\in\Gamma(\nu(\mathcal{F}_\xi)),
\end{equation*}
where $\Gamma(\bullet)$ is the set of smooth sections of the vector bundle $\bullet.$

 We consider the equations for the basic function $u\in C_{\mathrm{B}}^2(M,\mathbb{R})$
defined by
\begin{equation}
\label{equ1}
F(A_{\dag,u})=\psi
\end{equation}
for a given basic function $\psi\in C_{\mathrm{B}}^\infty(M,\mathbb{R})$, where $F(A_{\dag,u})$ is a smooth symmetric function of the eigenvalues $(\lambda_1,\cdots,\lambda_n)$ of the map $A_{\dag,u}$. We denote it by
\begin{equation}
\label{equ2}
F(A_{\dag,u})=f(\lambda_1,\cdots,\lambda_n).
\end{equation}
We suppose that $f$ is defined in an open symmetric cone $\Gamma\subsetneqq\mathbb{R}^n$, with vertex at the origin. We also assume that $\Gamma\supset\Gamma_n:=\left\{(x_1,\cdots,x_n)\in\mathbb{R}^n:\;x_i>0,\;1\leq i\leq n\right\}.$  For example, we can take (see \cite{spruck}) $\Gamma$ as the standard $k$-positive cone $\Gamma_k\subset\mathbb{R}^n$ which are defined by
\begin{equation*}
\Gamma_k:=\{\mathbf{x}\in\mathbb{R}^n:\;\sigma_i(\mathbf{x})>0,\;i=1,\cdots,k\},\quad 1\leq k\leq n,
\end{equation*}
where $\sigma_i$ is the $i^{\mathrm{th}}$ elementary symmetric polynomial defined on $\mathbb{R}^n$  given by
\begin{equation*}
\sigma_i(\mathbf{x})=\sum_{1\leq j_1<\cdots<j_i\leq n} x_{j_1}\cdots x_{j_i},\quad
\forall\;\mathbf{x}=(x_1,\cdots,x_n)\in\mathbb{R}^n,\quad 1\leq i\leq n.
\end{equation*}
In addition, $f$ satisfies
\begin{enumerate}
\item \label{assum1}$f$ is a concave function and $f_i:=\partial f/\partial \lambda_i>0$ for any $i=1,\cdots,n;$
\item \label{assum2}there holds $\sup_{\partial\Gamma}f<\inf_M \psi,$ where
\begin{equation*}
\sup_{\partial\Gamma}f:=\sup_{\lambda'\in\partial\Gamma}\limsup_{\Gamma\ni\lambda\to\lambda'}f(\lambda);
    \end{equation*}
\item \label{assum3}for any $\sigma$ with $\sigma<\sup_\Gamma f$ and $\lambda\in\Gamma$, we have
\begin{equation*}
    \lim_{t\to+\infty}f(t\lambda)>\sigma.
\end{equation*}
\end{enumerate}
 Assumption \eqref{assum3}, together with the concavity of $f$, yields that (see for example \cite{cnsacta})
\begin{equation}
 \label{filambdai}
\sum_{i=1}^nf_i\lambda_i\geq 0.
\end{equation}
\begin{defn}
  \label{tcsubsol}
Let $(M,\phi,\xi,\eta,g)$ be a Sasakian manifold with $\dim_{\mathbb{R}}M=2n+1$ $($$n\geq 2$$)$ and $\omega_\dag=\frac{1}{2}\mathrm{d}\eta=g(\phi\cdot,\cdot)$ as its transverse K\"ahler form. Then a  basic function $\underline{u}\in C_{\mathrm{B}}^\infty(M,\mathbb{R})$ is called a transverse $\mathcal{C}$-subsolution of \eqref{equ1} if at each point $\mathbf{p}$, the set
\begin{equation}
\left(\lambda\left(A_{\dag,\underline{u}}\right)+\Gamma_n\right)\cap \partial \Gamma^{\psi(\mathbf{p})}
\end{equation}
is bounded. Here and hereafter, $\lambda(A)$ denotes the $n$-tuple of eigenvalues of $A,$ and $\Gamma^\sigma$ is a convex set given by $$\Gamma^\sigma:=\{\lambda\in\Gamma:\;f(\lambda)>\sigma\}.$$
A  basic function $\underline{u}\in C_{\mathrm{B}}^\infty(M,\mathbb{R})$ is called transverse admissible if $\lambda\left(A_{\dag,\underline{u}}\right)\in\Gamma.$
\end{defn}
Note that  a transverse $\mathcal{C}$-subsolution need not to be transverse admissible.
\begin{thm}
\label{mainthm}
Suppose the basic function $u\in C_{\mathrm{B}}^\infty(M,\mathbb{R})$ is a solution to \eqref{equ1} with $\sup_Mu=0,$ and that the basic function $\underline{u}\in C_{\mathrm{B}}^\infty(M,\mathbb{R})$ is a transverse $\mathcal{C}$-subsolution to \eqref{equ1}. There exists a uniform constant $C$ depending only on $\eta,\,g,\,\omega_\dag$ and $\underline{u}$ such that
\begin{equation}
\label{mainequ}
\| u\|_{C^{2,\alpha}(M,g)}\leq C,\quad 0<\alpha<1.
\end{equation}
\end{thm}
Let us consider some applications.
\begin{cor}
\label{corjia1}
Let $(M,\phi,\xi,\eta,g)$ be a compact Sasakian manifold with $\dim_{\mathbb{R}}M=2n+1$ $($$n\geq 2$$)$ and $\omega_\dag=\frac{1}{2}\mathrm{d}\eta=g(\phi\cdot,\cdot)$ as its strictly transverse K\"ahler form, and let $\omega_h$ be a  strictly  transverse $k$-positive basic real $(1,1)$ form, i.e., the $n$-tuple $\lambda(\omega_h)$ of eigenvalues  of $\omega_h$ with respect to $\omega_\dag$ satisfies $\lambda\left( \omega_h  \right)\in \Gamma_k.$  Then given a basic function $G\in C_{\mathrm{B}}^\infty(M, \mathbb{R})$, there exists a unique basic function  $u\in C_{\mathrm{B}}^\infty(M,\mathbb{R})$ and unique constant $b\in\mathbb{R}$ such that the pair $(u,b)$ solves the equation
\begin{equation}
\label{twhesstcma}
\left(\omega_h+\frac{1}{n-1}\left[\left(\Delta_{\mathrm{B}}u\right)\omega_\dag-\sqrt{-1}\partial_{\mathrm{B}}\overline{\partial}_{\mathrm{B}}u \right] \right)^k\wedge\omega_\dag^{n-k}\wedge\eta=e^{G+b}
\omega_\dag^n\wedge\eta,
\end{equation}
where  $\sup_Mu=0$ and
\begin{equation*}
\omega_h+\frac{1}{n-1}\left[\left(\Delta_{\mathrm{B}}u\right)\omega_\dag-\sqrt{-1}\partial_{\mathrm{B}}\overline{\partial}_{\mathrm{B}}u \right]
\end{equation*}
is a   strictly transverse $k$-positive basic real $(1,1)$ form.
\end{cor}
In the K\"ahler setup, this equation was first introduced by \cite{houmawu}  with $k=n$, and solved by Tosatti and Weinkove \cite{twjams} for $k=n$ (see  \cite{houmawu}  on the K\"ahler manifolds with non-negative bisectional curvature and \cite{twcrelle} for Hermitian case) and by  Sz\'ekelyhidi \cite{gaborjdg} for $1\leq k<n$ which answers a question in \cite{twwycvpde}.

We say that a transverse positive basic real $(1,1)$ form $\varpi_\dag $ is transverse balanced if $\mathrm{d}_{\mathrm{B}}\varpi_\dag^{n-1}=0$,  transverse Gauduchon if $\partial_{\mathrm{B}}\overline{\partial}_{\mathrm{B}}\varpi_\dag^{n-1}=0$ and transverse strongly Gauduchon if $\overline{\partial}_{\mathrm{B}}\varpi_\dag^{n-1}$ is $\partial_{\mathrm{B}}$-exact.
We give a geometric description of Corollary \ref{corjia1} when $k=n$ which is a transverse counterpart of \cite{twjams}.
\begin{cor}
\label{corjia2}
Let $(M,\phi,\xi,\eta,g)$ be a compact Sasakian manifold with $\dim_{\mathbb{R}}M=2n+1$ $($$n\geq 2$$)$ and $\omega_\dag=\frac{1}{2}\mathrm{d}\eta=g(\phi\cdot,\cdot)$ as its transverse K\"ahler form, and let $\omega_0$ be a    transverse  balance $($resp. transverse Gauduchon, resp. transverse strongly Gauduchon$)$  metric and $F'\in C_{\mathrm{B}}^\infty(M,\mathbb{R})$ . Then  there exists a unique constant $b'\in\mathbb{R}$ and a unique transverse  balance $($resp. transverse Gauduchon, resp. transverse strongly Gauduchon$)$ metric
\begin{equation}
\label{twu}
\tilde\omega_{\dag,u}^{n-1}:=\omega_0^{n-1}+\sqrt{-1}\partial_{\mathrm{B}}\overline{\partial}_{\mathrm{B}}u \wedge\omega_{\dag}^{n-2}
\end{equation}
for some smooth basic function $u\in C_{\mathrm{B}}^\infty(M,\mathbb{R})$, solving the transverse Calabi-Yau equation
\begin{equation}
\frac{\tilde\omega_{\dag,u}^n\wedge\eta}{\omega_\dag^n\wedge\eta}=e^{F'+b'}.
\end{equation}
 This yields that given a representative $\Psi_\dag\in c_1^{\mathrm{BC,b}}(\nu(\mathcal{F}_\xi))$ $($basic first Chern class$)$, there exists a unique  transverse  balance $($resp. transverse Gauduchon, resp. transverse strongly Gauduchon$)$ metric defined as in \eqref{twu} such that
 \begin{equation}
 \ric(\tilde\omega_{\dag,u})=\Psi_\dag.
 \end{equation}
\end{cor}
Next, let us consider the transverse Hessian and Hessian quotient   equations
\begin{cor}
\label{cor1}
Let $(M,\phi,\xi,\eta,g)$ be a compact Sasakian manifold with $\dim_{\mathbb{R}}M=2n+1$ $($$n\geq 2$$)$ and $\omega_\dag=\frac{1}{2}\mathrm{d}\eta=g(\phi\cdot,\cdot)$ as its transverse K\"ahler form, and let $\omega_h$ be a strictly transverse $k$-positive basic $(1,1)$ form. Then given a basic function $G\in C_{\mathrm{B}}^\infty(M, \mathbb{R})$, there exists a unique basic function  $u\in C_{\mathrm{B}}^\infty(M,\mathbb{R})$ and unique constant $b\in\mathbb{R}$ such that the pair $(u,b)$ solves the transverse Hessian equation
\begin{equation}
\label{thesscma}
\left(\omega_h+ \sqrt{-1}\partial_{\mathrm{B}}\overline{\partial}_{\mathrm{B}}u \right)^k\wedge\omega_\dag^{n-k}\wedge\eta=e^{G+b}
\omega_\dag^n\wedge\eta,
\end{equation}
where  $\sup_Mu=0$ and
\begin{equation*}
\omega_h+ \sqrt{-1}\partial_{\mathrm{B}}\overline{\partial}_{\mathrm{B}}u
\end{equation*}
is a  strictly transverse $k$-positive basic $(1,1)$ form.
\end{cor}
When $k=n$ and $\omega_h=\omega_\dag$, this is solved by \cite{elk90} (see also \cite{bg01,swz10}), which is a counterpart to Yau \cite{yau1978} and means that given a Sasakian structure $(\phi,\xi,\eta,g)$ on $M$ and a representation $\frac{1}{2\pi}\rho$ (a real basic $(1,1)$ form) of the first basic Chern class $c_1(\nu(\mathcal{F}_\xi))$, there exists a unique Sasakian structure $(\tilde\phi,\xi,\tilde\eta,\tilde g)$ defined by \eqref{tildeeta}, \eqref{tildephi} and \eqref{tildeg}  and homologous to $(\phi,\xi,\eta,g)$ such that $\ric(\tilde g_\dag)=\rho-\mathrm{d}\tilde\eta,$ where $\ric(\tilde g_\dag)$ is given by \eqref{transversericciform} (\cite[Theorem 7.5.20]{sasakigeo} for more details and \cite{twjams10} for Hermitian case).

If $1<k<n$, then Corollary \ref{cor1} is a odd dimensional version of Dinew and Ko{\l}odziej \cite{dkajm} in the K\"ahler case.
\begin{cor}
\label{cor2}
Let $(M,\phi,\xi,\eta,g)$ be a compact Sasakian manifold with $\dim_{\mathbb{R}}M=2n+1$ $($$n\geq 2$$)$ and $\omega_\dag=\frac{1}{2}\mathrm{d}\eta=g(\phi\cdot,\cdot)$ as its transverse K\"ahler form, and let $\omega_h$ be a   closed strictly transverse $k$-positive basic $(1,1)$ form. Then  there exists a basic function  $u\in C^\infty_{\mathrm{B}}(M,\mathbb{R})$ solving the general transverse Hessian quotient equation
\begin{equation}
\label{thessquotient}
\left(\omega_h+ \sqrt{-1}\partial_{\mathrm{B}}\overline{\partial}_{\mathrm{B}}u \right)^\ell\wedge\omega_\dag^{n-\ell}\wedge\eta
=c \left(\omega_h+ \sqrt{-1}\partial_{\mathrm{B}}\overline{\partial}_{\mathrm{B}}u \right)^k\wedge\omega_\dag^{n-k}\wedge\eta,\quad 1\leq \ell<k\leq n,
\end{equation}
if
\begin{equation}
\label{thessquotientcon}
kc\omega_h^{k-1}\wedge\omega_\dag^{n-k} -\ell\omega_h^{\ell-1}\wedge\omega_\dag^{n-\ell}
\end{equation}
is a strictly transverse positive basic $(n-1,n-1)$ form, where
$$
c=\frac{\int_M\omega_h^\ell\wedge\omega_\dag^{n-\ell}\wedge\eta}{\int_M\omega_h^k
\wedge\omega_\dag^{n-k}\wedge\eta}.
$$
\end{cor}
This is a counterpart to Sz\'ekelyhidi \cite{gaborjdg} (see also Song and Weinkove \cite{songweinkove2008} for the case of $\ell=n-1,\,k=n$ whose solution is the critical point of the $J$-flow introduce by Donaldson \cite{donald1999asian} from the point of view of moment amps, as well as Chen \cite{chen2000imrn,chen2004cag} in his study of the Mabuchi energy, and Fang, Lai and Ma \cite{fanglaima2011} for the case of general $\ell$ and $k=n$) in the K\"ahler case. Similarly, the solution to \eqref{thessquotient} for $\ell=n-1,k=n$ is the critical point of the transverse  $J$-flow
\begin{equation*}
\frac{\partial}{\partial t}\omega_\dag(t)=-\sqrt{-1}\partial_{\mathrm{B}}\overline{\partial}_{\mathrm{B}}
\left(\frac{\omega_h\wedge\omega_\dag(t)^{n-1}\wedge\eta}{\omega_\dag^n(t)\wedge\eta}\right),\quad
\omega_\dag(0)=\omega_h.
\end{equation*}
Furthermore, we can also propose a transverse parabolic flow for basic function $u\in C_{\mathrm{B}}^\infty(M,\mathbb{R})$
\begin{equation*}
\partial_tu=F(A_{\dag,u})-\psi,\quad u(0)=0,
\end{equation*}
such that the solution to  \eqref{equ1} is its critical point. This is a transverse version of Phong and T\^{o} \cite{phongto2017} in Hermitian manifolds (cf.\cite{collinszekelyhidijdg2017,fanglaitransaction2013,fanglaipacific2012} in K\"ahler manifolds).

If the Sasakian manifold $M$ is regular, then its base space $B$ of
the Boothby-Wang foliation \cite{bw58} is a Hodge manifold (see \cite[Theorem 4]{ha63}), in which case the transverse fully nonlinear equations in this paper  can be reduced to the fully nonlinear equations on the base space $B$. However, generally the base space $B$ is very wild and has no any manifold structure, and hence it is meaningful to consider transverse fully non-linear equations on Sasakian manifolds from this viewpoint.

The paper is organized as follows. In section \ref{secsasaki}, we collect some basic concepts about (almost) contact and Sasakian manifolds.  In section \ref{sec0order}, we give the uniform bounds of  the solution $u$ to \eqref{equ1}. In section \ref{sec2order} and Section \ref{sec1st}, we give the second and first order priori estimates of the solution $u$ to \eqref{equ1} respectively, and complete the proof of Theorem \ref{mainthm}. In Section \ref{secapp}, as applications, we prove Corollary \ref{cor1} and Corollary \ref{cor2}. In Section \ref{secfoliation}, we will explain that our method works on a compact oriented, taut, transverse foliated manifold with complex codimension $n.$ In particular, we give a Calabi-Yau type theorem for the transverse Gauduchon metric on such foliated manifolds.

The method in this paper is modified from the complex case \cite{twjams,twcrelle,gaborjdg,stw1503,collinszekelyhidijdg2017}. For readers' convenience, we give all the details in Sasakian setup and sketch in the compact oriented, taut, transverse foliated manifold case.

\noindent {\bf Acknowledgements}
The first-named author thanks Professor Gang Tian and Xiaohua Zhu for
their concern and help in the research.
The second-named author thanks Professor Jean-Pierre Demailly, Valentino Tosatti and Ben Weinkove for their invaluable directions, and Professor Chao Qian and Luigi Vezzoni  for the helpful discussions about foliation and Sasakian manifolds. Part of the work was carried out when the second-named author was a post-doc in Institute Fourier  supported by the European Research Council (ERC) grant No. 670846 (ALKAGE) and hence he thanks the institute for the hospitality. The authors are also grateful to the anonymous referees and the editor for their careful reading and helpful suggestions which greatly improved the paper.
\section{Preliminaries}\label{secsasaki}
In this section, we collect some preliminaries about contact manifolds which will be used in the following (see for example \cite{sasakigeo,kn}).
\subsection{Almost Contact Manifolds}
Let $M$ be a   manifold with $\dim_{\mathbb{R}}M=2n+1$, and $\phi,\,\xi,\,\eta$ be a tensor of type $(1,1)$, a vector field and a $1$-form respectively. Then if $\phi,\,\xi$ and $\eta$ satisfy
\begin{align}
\label{sasaki1}
\eta(\xi)=&1,\\
 \label{sasaki2}
 \phi^2Z=&-Z+\eta(Z)\xi,\quad \forall \,Z\in\mathfrak{X}(M),
\end{align}
then $M$ is said to have an \emph{almost contact structure} $(\phi,\,\xi,\,\eta)$, and is called an \emph{almost contact manifold}.
For the almost contact structure $(\phi,\,\xi,\,\eta)$, it follows from \cite[Proposition V-1.1 \& V-1.2]{kn} that
\begin{align*}
\phi(\xi)= 0,\quad \text{rank}\phi=2n,\quad
\eta(\phi (Z))= 0,\quad \forall \,Z\in\mathfrak{X}(M),
\end{align*}
and that $M$ admits a Riemannian metric $g$ such that
$$g(Z,\xi)=\eta(Z),\quad g_\dag(Y,Z):=g(\phi Y,\phi Z)=g(Y,Z)-\eta(Y)\eta(Z),\quad\forall \,Y,\,Z\in\mathfrak{X}(M).$$
We also define
\begin{equation*}
\omega_\dag(Y,Z):=g_\dag(\phi Y,Z),\quad\forall \,Y,\,Z\in\mathfrak{X}(M).
\end{equation*}
Since the vector field $\xi$ is nowhere vanishing, it generates a $1$ dimensional  subbundle $L_\xi$ of the tangent bundle $TM$. Hence, an almost contact manifold $(M,\,\phi,\,\xi,\,\eta)$  has a $1$ dimensional  foliation $\mathcal{F}_\xi$ which is called the \emph{characteristic foliation} associated to $L_\xi.$ The $1$ form $\eta$ is called the \emph{characteristic} $1$ form, and it
defines a $2n$ dimensional vector bundle $\mathcal{D}$, called \emph{horizontal subbundle} of $TM,$ on $M$, where the fiber $\mathcal{D}_p$ of $\mathcal{D}$ is defined by
$$
\mathcal{D}_p:=\mathrm{Ker}\eta_p,\quad\forall\,p\in M.
$$
Hence we get a decomposition of the tangent bundle $TM$ given by
$$
TM=\mathcal{D}\oplus L_\xi,
$$
and an exact sequence of vector bundles
$$
0\longrightarrow  L_\xi\longrightarrow TM \stackrel{\pi}{\longrightarrow} \nu(\mathcal{F}_\xi)\longrightarrow0,
$$
where $\nu(\mathcal{F}_\xi):=TM/L_\xi$ which is called the \emph{normal bundle} of the foliation $\mathcal{F}_\xi.$ There is a smooth vector bundle isomorphism $\sigma$ given by
\begin{equation}
\label{sigma}
\sigma:\,\nu(\mathcal{F}_\xi)\longrightarrow\mathcal{D}
\end{equation}
such that $\pi\circ\sigma=\mathrm{Id}_{\nu(\mathcal{F}_\xi)}.$
It follows that $\phi$ induces a splitting
$$
\mathcal{D}\otimes_{\mathbb{R}}\mathbb{C}=\mathcal{D}^{1,0}\oplus \mathcal{D}^{0,1},
$$
where $\mathcal{D}^{1,0}$ and $\mathcal{D}^{0,1}$ are eigenspaces of $\phi$ with eigenvalues $\sqrt{-1}$ and $-\sqrt{-1}$, respectively. We call $(\mathcal{D}, \phi_{\upharpoonright\mathcal{D}})$    an \emph{almost CR  structure}.

A $p$ form $\varpi$ on $M$ is called \emph{basic} if
\begin{equation}
\label{basicformdefn}
\iota_Y\varpi=0\quad\text{and}\quad\mathcal{L}_Y\varpi=\iota_{Y}\mathrm{d}\varpi=0,\quad \forall\, Y\in \Gamma(L_\xi),
\end{equation}
where $\iota_Y$ is the inner product defined by
$$(\iota_Y\varpi)(X_1,\cdots,X_{p-1})=\varpi(Y,X_1,\cdots,X_{p-1}),\quad \forall\;X_1,\cdots,X_p\in\mathfrak{X}(M),$$
and $\mathcal{L}_Y$ is the Lie derivative given  by $\mathcal{L}_Y=\iota_{Y}\circ\mathrm{d}+\mathrm{d}\circ\iota_{Y}$.
Here we recall that $\mathrm{d}\varpi$ is defined by
\begin{align}
\label{waiweifen}
(\md\varpi)(X_1,\cdots,X_{p+1})
=&\sum\limits_{\lambda=1}^{p+1}(-1)^{\lambda+1}X_{\lambda}(\varpi(X_1,\cdots,\widehat{X_{\lambda}},\cdots,X_{p+1}))\\
&+\sum\limits_{\lambda<\mu}(-1)^{\lambda+\mu}\varpi([X_{\lambda},X_{\mu}],X_1,\cdots,\widehat{X_{\lambda}},\cdots,\widehat{X_{\mu}},\cdots,X_{p+1})\nonumber
\end{align}
for any fields $X_1,\cdots,X_{p+1}\in \mathfrak{X}(M)$.

Note that a basic $0$ form (i.e., basic function) means that a function  $u\in C^1(M, \mathbb{R})$ satisfying $\xi (u)\equiv0$.
We use the notation that
 \begin{equation}
 \label{ckbfunction}
 C^k_{\mathrm{B}}(M,\mathbb{R}):=\left\{u\in C^k(M,\mathbb{R}):\, \xi(u)=0 \right\},\quad k\in \mathbb{N}^*\cup\{\infty\}.
 \end{equation}
Let $\bigwedge^p_{\mathrm{B}}$  denote the sheaf of germs of basic $p$ forms,   $\Omega_{\mathrm{B}}^p:=\Gamma(M, \bigwedge^p_{\mathrm{B}})$  the set of global sections of $\bigwedge^p_{\mathrm{B}}.$
  Since the exterior differential preserves basic forms, we set $\mathrm{d}_{\mathrm{B}}:=\mathrm{d}_{\upharpoonright\Omega_{\mathrm{B}}^p}$ with $\mathrm{d}_{\mathrm{B}}^2=0$. Note that $\Omega_{\mathrm{B}}^p$ is $C^\infty_{\mathrm{B}}(M,\mathbb{R})$ module.

The manifold $M$ with $\dim_{\mathbb{R}}M=2n+1$ is said to have a \emph{contact structure} and is called a \emph{contact manifold} if it carries a $1$-form $\eta$ with
\begin{equation}
  \eta\wedge(\mathrm{d}\eta)^n\not=0
\end{equation}
everywhere on $M$, and $\eta$ is called a \emph{contact form} on $M$. We know that (see for example \cite[Lemma 6.1.24]{sasakigeo}) there exists a unique vector field $\xi$ such that
\begin{equation*}
\eta(\xi)=1,\quad \mathrm{d}\eta(\xi,Z)=0,\quad \forall\,Z\in\mathfrak{X}(M).
\end{equation*}
The vector field $\xi$ is called the \emph{characteristic vector field} or the \emph{Reeb vector field}.
For this contact manifold $M$, there exists (see for example \cite[Theorem V-2.1]{kn}) an almost contact metric structure $(\phi,\,\xi,\,\eta,\,g)$
such that
\begin{equation}
\label{gdag}
\omega_\dag(Y,Z)= g_\dag(\phi Y, Z)=g(\phi Y, Z)=\frac{1}{2}\mathrm{d}\eta(Y,Z), \quad Y,\,Z\in \mathfrak{X}(M).
\end{equation}
This structure $(\phi,\,\xi,\,\eta,\,g)$ is called \emph{contact metric structure} associated to the contact form $\eta$, and a manifold with such a structure is called \emph{contact metric manifold}.

For a contact metric manifold $(M,\,\phi,\,\xi,\,\eta,\,g)$, we take
\begin{equation}
\label{defnvol}
\mathrm{dvol}_g:=\frac{\eta\wedge\omega_\dag^n}{n!}=\frac{1}{2^nn!}\eta\wedge(\mathrm{d}\eta)^n
\end{equation}
as the Riemannian volume form.  We define the transverse Hodge star operator $\ast_{\dag}$ in term of the usual Hodge star operator $\ast$ by (see for example \cite[Formula (7.2.2)]{sasakigeo})
\begin{equation}
\label{defnbardag}
\ast_{\dag}:\,\Omega_{\mathrm{B}}^p\to \Omega_{\mathrm{B}}^{2n-p},\quad   \ast_{\dag} \varphi\mapsto \ast (\eta\wedge \varphi)=(-1)^p\iota_\xi(\ast\varphi).
\end{equation}
The  adjoint $\delta_{\mathrm{B}}:\,\Omega_{\mathrm{B}}^p\to \omega_{\mathrm{B}}^{p-1}$ of $\mathrm{d}_{\mathrm{B}}$  is defined by
\begin{equation}
  \delta_{\mathrm{B}}:=- \ast_{\dag}\circ\mathrm{d}_{\mathrm{B}}\circ \ast_{\dag}.
\end{equation}
The basic Laplacian $\Delta_{\mathrm{d}_{\mathrm{B}}}$ is defined in terms of $\mathrm{d}_{\mathrm{B}}$ and its its adjoint $\delta_{\mathrm{B}}$ by
\begin{equation}
\Delta_{\mathrm{d}_{\mathrm{B}}}:=\mathrm{d}_{\mathrm{B}}\circ\delta_{\mathrm{B}}+\delta_{\mathrm{B}}\circ\mathrm{d}_{\mathrm{B}}.
\end{equation}
For a contact metric manifold $(M,\,\phi,\,\xi,\,\eta,\,g)$, if $\xi$ is a Killing vector field with respect to $g$, then $(\phi,\,\xi,\,\eta,\,g)$ is called a \emph{K-contact structure} and $M$ is called a \emph{K-contact manifold}.
\subsection{Normality of Almost Contact Manifolds} There are two slightly different methods to introduce the notation of normality of the almost contact structure (see \cite[Section V-3]{kn} and \cite[Section 6.5]{sasakigeo}).

Let $(M,\,\phi,\,\xi,\,\eta)$ be an almost contact manifold with $\dim_{\mathbb{R}}M=2n+1$.  Then we define
\begin{align*}
4\mathcal{N}_\phi(Y,\,Z):=&\phi^2[Y,\,Z]+[\phi Y,\,\phi Z]-\phi[\phi Y,\,Z]-\phi[Y,\,\phi Z],\\
\mathcal{N}^{(1)}(Y,Z):=&4\mathcal{N}_\phi(Y,Z)+\mathrm{d}\eta(Y,Z)\xi,\quad
\mathcal{N}^{(2)}(Y,Z):= (L_{\phi Y}\eta)(Z)-(L_{\phi Z}\eta)(Y),\\
\mathcal{N}^{(3)}(Z):=&(L_{\xi}\phi)(Z),\quad
 \mathcal{N}^{(4)}(Z):= (L_{\xi}\eta)(Z),\quad \forall\,Y,\,Z\in \mathfrak{X}(M).
\end{align*}
If  $\mathcal{N}^{(1)}$ vanishes so does $\mathcal{N}^{(i)}$ for $i=2,3,4$ (see for example \cite[Lemma 6.5.10]{sasakigeo}).

If  $(M,\phi,\xi,\eta,\,g)$ is a contact metric manifold, then we have $\mathcal{N}^{(2)}=\mathcal{N}^{(4)}=0$. Furthermore, $(M,\phi,\xi,\eta,\,g)$ is K-contact if and only if $\mathcal{N}^{(3)}=0$ (see for example \cite[Proposition 6.5.12]{sasakigeo}).

For a smooth manifold $M$, we denote by $X:=\mathbb{R}_+\times M$ the \emph{cone} on $M$, where $\mathbb{R}_+$ is the set of positive of real numbers with coordinate $r$. We shall identify $M$ with $\{1\}\times M$.

Let $(M,\,\phi,\,\xi,\,\eta)$ be an almost contact manifold with $\dim_{\mathbb{R}}M=2n+1$. Then we define an almost complex structure $J$ on the tangent bundle $TX$ of the cone by
\begin{equation*}
J Z=\phi Z-\eta(Z) (r\partial_r) , \quad J (r\partial_r)=\xi,\quad Z\in\mathfrak{X}(M),
\end{equation*}
where $r\partial_r:=r(\partial/\partial r)$ is the Liouville (or Euler) vector field, and a Hermitian metric $g_X$ given by
$$
g_X=\mathrm{d}r\otimes \mathrm{d}r+r^2g.
$$
Indeed, a direct calculation yields that
\begin{align*}
J^2=&-\mathrm{Id}_{TX},\quad g_X(JY,JZ)=g_X(Y,Z),\quad\forall\, Y,\,Z\in \mathfrak{X}(X).
\end{align*}
The Nijenhuis tensor $\mathcal{N}_J$  of this almost complex structure is defined by
\begin{equation}
\label{torsion}
4\mathcal{N}_J(Y,\,Z)=J^2[Y,\,Z]+[JY,\,JZ]-J[JY,\,Z]-J[Y,\,JZ], \quad \forall\,Y,\,Z\in \mathfrak{X}(X).
\end{equation}
If the Nijenhuis tensor is integrable, i.e., $\mathcal{N}_J\equiv0$, then the almost contact structure $(\phi,\,\xi,\,\eta)$ is called \emph{normal}.

We know that (see for example \cite[Theorem 6.5.9]{sasakigeo}) the almost contact structure $(\phi,\,\xi,\,\eta)$ of $M$ is normal if and only if $4\mathcal{N}_\phi=-\mathrm{d}\eta\otimes \xi;$ and, if and only if  (see for example \cite{gkn00})
\begin{enumerate}
\item[(a)]$[\Gamma(\mathcal{D}^{0,1}),\Gamma(\mathcal{D}^{0,1})]\subset \Gamma(\mathcal{D}^{0,1})$, i.e., the almost CR structure $(\mathcal{D},\phi_{\upharpoonright\mathcal{D}})$ is integrable;
    \item[(b)]$[\xi,\Gamma(\mathcal{D}^{0,1})]\subset \Gamma(\mathcal{D}^{0,1}) $,  i.e., $\mathcal{N}^{(3)}\equiv0.$
\end{enumerate}
A normal contact metric structure $(\phi,\,\xi,\,\eta,\,g)$ on $M$ is called a \emph{Sasakian structure}, and $M$ with this structure is called a \emph{Sasakian manifold}.

A contact metric manifold $(M,\,\phi,\,\xi,\,\eta,\,g)$ is Sasakian if its metric cone $(X,g_X:=\mathrm{d}r\otimes \mathrm{d}r+r^2g,\,\frac{1}{2}\mathrm{d}(r^2\eta),J)$ is K\"ahler (see for example \cite[Definitioin 6.5.15]{sasakigeo}). We know that
$$
\frac{1}{2}(\mathrm{d}(r^2\eta))(Y,Z)=g_X(JY,Z)=:\omega_X(Y,Z),\quad \forall\,Y,Z\in\mathfrak{X}(X).
$$
For any $p$ form $\varpi$ on the almost contact manifold $(M,\phi,\xi,\eta,\,g)$, we can define
\begin{equation}
\label{jkuozhangb}
(\phi\varpi)(X_1,\cdots,X_p):=(-1)^p\varpi(\phi X_1,\cdots,\phi X_p),\quad X_1,\cdots,X_p\in\mathfrak{X}(M).
\end{equation}
If $\mathcal{N}^{(3)}=0$, then we have
\begin{equation}
\label{n3zero}
  \phi[\xi, Z]=[\xi,\phi Z],\quad \forall\;Z\in\mathfrak{X}(M).
\end{equation}
This yields that if $\varpi$ is a basic $p$ form, then so is $\phi \varpi$, and that $\phi^2\varpi=-\varpi$.
Then we have $\bigwedge^1_{\mathrm{B}}\otimes_{\mathbb{R}}\mathbb{C}:=\bigwedge_{\mathrm{B}}^{1,0}+\bigwedge_{\mathrm{B}}^{0,1}$
and
\begin{equation*}
\Omega_{\mathrm{B}}^1\otimes_{\mathbb{R}}\mathbb{C}:=\Omega_{\mathrm{B}}^{1,0}+\Omega_{\mathrm{B}}^{0,1},
\end{equation*}
where $\Omega_{\mathrm{B}}^{1,0}$ and $\Omega_{\mathrm{B}}^{0,1}$  are eigenspaces of $\phi$ with eigenvalues $-\sqrt{-1}$ and $\sqrt{-1}$, respectively.
We also denote that $\bigwedge_{\mathrm{B}}^{p,q}:=\bigwedge^p\left(\bigwedge_{\mathrm{B}}^{1,0}\right)
\otimes\bigwedge^q\left(\bigwedge_{\mathrm{B}}^{0,1}\right)$, and
$$
\Omega_{\mathrm{B}}^{p,q}:=\bigwedge^p\left(\Omega_{\mathrm{B}}^{1,0}\right)
\otimes\bigwedge^q\left(\Omega_{\mathrm{B}}^{0,1}\right).
$$
Then we have $\bigwedge_{\mathrm{B}}^p\otimes_{\mathbb{R}}\mathbb{C}=\bigoplus_{r+s=p}\bigwedge_{\mathrm{B}}^{r,s}$, and
$$
\Omega_{\mathrm{B}}^p\otimes_{\mathbb{R}}\mathbb{C}=\bigoplus_{r+s=p}\Omega_{\mathrm{B}}^{r,s}.
$$
It is easy to find a local frame basis $\theta^1,\cdots,\theta^n$ of $\bigwedge_{\mathrm{B}}^{1,0}$ is and a local frame basis $e_1,\cdots,e_n$ of $\mathcal{D}^{1,0}$ such that
$$
\theta^i(e_j)=\delta^i_j,\quad \phi e_i=\sqrt{-1}e_i,\quad 1\leq i,\,j\leq n.
$$
Note that $\xi, e_1,\cdots,e_n,\bar e_1,\cdots,\bar e_n$ is a local frame basis of $TM\otimes_{\mathbb{R}}\mathbb{C}$ with dual $\eta,\theta^1,\cdots,\theta^n,\bar\theta^1,\cdots,\bar\theta^n$. We set
$$
[\xi,e_i]=C_{0i}^0\xi+C_{0i}^ke_k,\quad[e_i,e_j]=C_{ij}^0\xi+C_{ij}^ke_k+C_{ij}^{\bar k}\bar e_k,\quad
[e_i,\bar e_j]=C_{i\bar j}^0\xi+C_{i\bar j}^ke_k+C_{i\bar j}^{\bar k}\bar e_k,
$$
since $\mathcal{N}^{(3)}=0$ yields that $[\xi,e_i]^{(0,1)}=0$.
Then we get
\begin{equation}
\label{dtheta}
\mathrm{d}\theta^k=\mathrm{d}_{\mathrm{B}}\theta^k=-\frac{1}{2}C_{ij}^k\theta^i\wedge\theta^j-
\frac{1}{2}\overline{C_{ij}^{\bar k}}\bar\theta^i\wedge\bar\theta^j-C_{i\bar j}^k\theta^i\wedge\bar\theta^j,
\end{equation}
since $\mathrm{d}\theta^k$ is also basic.

From \eqref{dtheta},  we can split the basic exterior differential operator, $\mathrm{d}_{\mathrm{B}}: \Omega_{\mathrm{B}}^{\bullet} \otimes_{\mathbb{R}}\mathbb{C}\longrightarrow \Omega_{\mathrm{B}}^{\bullet+1} \otimes_{\mathbb{R}}\mathbb{C}$,  into four components (see for example \cite{angella} for the almost complex case)
\begin{equation}
\label{defndb}
\mathrm{d}_{\mathrm{B}}=A_{\mathrm{B}}+\partial_{\mathrm{B}}
+\overline{\partial}_{\mathrm{B}}+\overline{A}_{\mathrm{B}}
\end{equation}
with
\begin{align*}
A_{\mathrm{B}}:\;&\Omega_{\mathrm{B}}^{\bullet,\bullet}  \longrightarrow \Omega_{\mathrm{B}}^{\bullet+2,\bullet-1},  \\
\partial_{\mathrm{B}}:\;&\Omega_{\mathrm{B}}^{\bullet,\bullet}  \longrightarrow \Omega_{\mathrm{B}}^{\bullet+1,\bullet},\\
\overline{\partial}_{\mathrm{B}}:\;&\Omega_{\mathrm{B}}^{\bullet,\bullet}  \longrightarrow \Omega_{\mathrm{B}}^{\bullet,\bullet+1},  \\
\overline{A}_{\mathrm{B}}:\;&\Omega_{\mathrm{B}}^{\bullet,\bullet}  \longrightarrow \Omega_{\mathrm{B}}^{\bullet-1,\bullet+2}.
\end{align*}
In terms of these components, the condition $\mathrm{d}_{\mathrm{B}}^2=0$ can be written as
\begin{align*}
&A_{\mathrm{B}}^2=
\partial_{\mathrm{B}} A_{\mathrm{B}}+A_{\mathrm{B}}\partial_{\mathrm{B}}
=A_{\mathrm{B}}\overline{\partial}_{\mathrm{B}}+\partial_{\mathrm{B}}^2
+\overline{\partial}_{\mathrm{B}}A_{\mathrm{B}}=0,\\
&A_{\mathrm{B}}\overline{A}_{\mathrm{B}}
+\partial_{\mathrm{B}}\overline{\partial}_{\mathrm{B}}
+\overline{\partial}_{\mathrm{B}}\partial_{\mathrm{B}}
+\overline{A}_{\mathrm{B}}A_{\mathrm{B}}=0,\\
&\overline{A}_{\mathrm{B}}^2
=\overline{A}_{\mathrm{B}}\overline{\partial}_{\mathrm{B}}
+\overline{\partial}_{\mathrm{B}}\overline{A}_{\mathrm{B}}
=\partial_{\mathrm{B}}\overline{A}_{\mathrm{B}}
+\overline{\partial}_{\mathrm{B}}^2
+\overline{A}_{\mathrm{B}}\partial_{\mathrm{B}}
=0.
\end{align*}
For any basic function $\varphi\in C_{\mathrm{B}}^{\infty}(M,\,\mathbb{R})$,  from \eqref{jkuozhangb} and \eqref{waiweifen}, a direct computation yields
\begin{align*}
(\mathrm{d}_{\mathrm{B}}\phi\mathrm{d}_{\mathrm{B}}\varphi)(e_i,e_j)
=&-2\sqrt{-1}[e_i,e_j]^{(0,1)}(\varphi),\\
(\mathrm{d}_{\mathrm{B}} \phi\mathrm{d}_{\mathrm{B}}\varphi)(\overline{e}_i,\overline{e}_j)
=&2\sqrt{-1}[\overline{e}_i,\overline{e}_j]^{(1,0)}(\varphi),\\
(\mathrm{d}_{\mathrm{B}} \phi\mathrm{d}_{\mathrm{B}}\varphi)(e_i,\overline{e}_j)
=&2\sqrt{-1} \left(e_i\overline{e}_j(\varphi)- [e_i,\bar e_j]^{(0,1)}(\varphi)\right),
\end{align*}
where $[e_i,e_j]^{(0,1)}$ means the projection of $[e_i,e_j]$ to $\mathcal{D}^{0,1}$.
A direct calculation shows that
\begin{align*}
\sqrt{-1}\partial_{\mathrm{B}} \overline{\partial}_{\mathrm{B}}\varphi=\frac{1}{2}(\mathrm{d}_{\mathrm{B}} \phi\mathrm{d}_{\mathrm{B}}\varphi)^{(1,1)} =\sqrt{-1}\left(e_i\overline{e}_j(\varphi)-[e_i,\overline{e}_j]^{(0,1)}(\varphi)\right)\theta^i\wedge\overline{\theta}^j.
\end{align*}
We also deduce
\begin{equation*}
\mathcal{N}_\phi=-\Re\left(C_{ij}^{\bar k}(\theta^i\wedge \theta^j)\otimes \bar e_k\right)
-\frac{1}{4}\Re\left(\eta([e_i,e_j])(\theta^i\wedge \theta^j)\otimes\xi\right)+\frac{1}{4}\eta([e_i,\bar e_j])(\theta^i\wedge \bar\theta^j)\otimes\xi,
\end{equation*}
since $\mathcal{N}^{(3)}=0$ means $C_{0i}^{\bar k}=0$.

For the almost contact manifold $(M,\phi,\xi,\eta,\,g)$ with $\mathcal{N}^{(i)}=0,\,i=2,3,4$ (e.g., the K-contact manifold),  we have
$C^0_{0i}=C_{ij}^0=0$, and hence
\begin{equation*}
\mathcal{N}_\phi= -\Re\left(C_{ij}^{\bar k}(\theta^i\wedge \theta^j)\otimes \bar e_k\right)
 +\frac{1}{4}\eta([e_i,\bar e_j])(\theta^i\wedge \bar\theta^j)\otimes\xi .
\end{equation*}
If the almost contact manifold $(M,\phi,\xi,\eta,\,g)$ is normal (e.g., the Sasakian manifold), then there also holds $C_{0i}^{\bar k}=C_{ij}^{\bar k}=0$. This yields that
\begin{align}
\mathrm{d}_{\mathrm{B}}=& \partial_{\mathrm{B}}
+\overline{\partial}_{\mathrm{B}},\quad
 \partial_{\mathrm{B}}\overline{\partial}_{\mathrm{B}}
+\overline{\partial}_{\mathrm{B}}\partial_{\mathrm{B}}=0,\quad
\mathcal{N}_\phi= \frac{1}{4}\eta([e_i,\bar e_j])(\theta^i\wedge \bar\theta^j)\otimes\xi,\nonumber\\
\label{tddbar}
&\sqrt{-1}\partial_{\mathrm{B}} \overline{\partial}_{\mathrm{B}}\varphi=\frac{1}{2} \mathrm{d}_{\mathrm{B}} \phi\mathrm{d}_{\mathrm{B}}(\varphi),\quad \forall \;\varphi\in C_{\mathrm{B}}^\infty(M,\mathbb{R}).
\end{align}
\subsection{Connections on Almost Contact Manifolds}
 Let $(M,g)$ be a Riemannian manifold. Then we denote by $\nabla$ the Levi-Civita connection. We define the Riemannian curvature $R$ by
 $$
 R(W,Y)Z=\nabla_W\nabla_YZ-\nabla_Y\nabla_WZ-\nabla_{[W,Y]}Z,\quad \forall\,W,\,Y,\,Z\in\mathfrak{X}(M).
 $$

 Thanks to \cite[Theorem V-5.1]{kn}, an almost contact metric manifold $(M,\phi,\xi,\eta,g)$ is a Sasakian manifold if and only if
 \begin{equation}
 \label{sasanablaphi}
 (\nabla_Y\phi)Z=g(Z,\xi)Y-g(Y,Z)\xi,\quad\forall\,Y,\,Z\in\mathfrak{X}(M).
 \end{equation}
Let $(M,\,g)$ be a Riemannian manifold with $\dim_{\mathbb{R}}M=2n+1$ admitting a unit Killing vector field $\xi$. Then \cite[Theorem V-5.2 \& V-3.1]{kn} yields that $M$ is a Sasakian manifold if only if
\begin{equation*}
 R(Y,\xi)Z=g(Z,\xi)Y-g(Y,Z)\xi,\quad\forall\,Y,\,Z\in\mathfrak{X}(M),
\end{equation*}
or the sectional curvature for plane sections containing $\xi$ are equal to $1$ at every point of $M$.
\subsection{Transverse K\"ahler Structures on Sasakian Manifolds}
Let $(M,\phi,\xi,\eta,g)$ be a Sasakian manifold with $\dim_{\mathbb{R}}M=2n+1$. Then $(\mathcal{D},\phi_{\upharpoonright\mathcal{D}},\mathrm{d}\eta)$ (hence $(\nu(\mathcal{F}_\xi),\sigma^*\phi_{\upharpoonright\mathcal{D}},\sigma^*\mathrm{d}\eta)$ by $\sigma$ in \eqref{sigma}) gives $M$ a transverse K\"ahler structure with transverse K\"ahler form $\omega_
\dag:=\frac{1}{2}\mathrm{d}\eta$ and transverse K\"ahelr metric $g_\dag$ defined by \eqref{gdag}. We have the relationship that
$
g=g_\dag+\eta\otimes\eta.
$

Given a Sasakian structure $(\phi,\xi,\eta,g)$ on $M$ and any   $u\in C_{\mathrm{B}}^\infty(M,\mathbb{R})$, we define
\begin{equation}
\label{tildeeta}
\tilde\eta:=\eta+\sqrt{-1}(\overline{\partial}_{\mathrm{B}}-\partial_{\mathrm{B}})u.
\end{equation}
If $\mathrm{d}\tilde\eta=\mathrm{d}\eta+2\sqrt{-1}\partial_{\mathrm{B}}\overline{\partial}_{\mathrm{B}}u>0$ and $\tilde\eta\wedge(\mathrm{d}\tilde\eta)^n\not=0,$
then $(\xi,\tilde\eta,\tilde\phi,\tilde g)$ is also a Sasaki structure homologous to $(\xi,\eta,\phi,g)$ (see \cite{sasakigeo}) on $M$, where
\begin{align}
\label{tildephi}
\tilde \phi:=&\phi- \xi\otimes \sqrt{-1}(\overline{\partial}_{\mathrm{B}}-\partial_{\mathrm{B}})u\circ\phi,\\
\label{tildeg}
\tilde g:=&\frac{1}{2}\mathrm{d}\tilde\eta\circ(\mathrm{Id}\otimes\tilde\phi)+\tilde\eta\otimes\tilde\eta.
\end{align}
These deformations fix the Reeb field $\xi$ and change $(\eta,\phi,g)$ and hence $\mathcal{D}$; however, the quotient vector bundle $\nu(\mathcal{F}_\xi)$ is invariant. We equip $\nu(\mathcal{F}_\xi)$ with a complex structure $\mathbf{I}$ invariant under the deformations above given by
\begin{equation}
\label{ti}
\mathbf{I}(V):=\pi\circ\phi\circ\sigma(V),\quad \forall\;V\in  \nu(\mathcal{F}_\xi).
\end{equation}
\subsection{Vector Bundles on Sasakian Manifolds}\label{secvectorbundle} We recall the concepts of foliated/transverse vector bundles in \cite{bs2010} originated from \cite{molino1968} and \cite{kambertondeur1971,kambertondeur1971a}.
Let $M$ be a real smooth manifold and $S\subset TM$ an involutive subbundle of $TM$ (i.e., smooth sections of $S$ are closed under the operation of the Lie bracket). Then the partial connection of a vector bundle $E\to M$ is given by
\begin{equation*}
\nabla^0:\;\Gamma(S)\times \Gamma (E)\to \Gamma(E),\quad (Z,s)\mapsto \nabla^0_Zs,
\end{equation*}
satisfying
\begin{align*}
\nabla^0_{W+fZ}s=&\nabla^0_Ws+f\nabla^0_Zs,\\
\nabla^0_Z(s+ft)=&\nabla^0_Zs+(Zf)t+f\nabla^0_Zt,
 \quad \forall\,W,Z\in\Gamma(S),\;\forall\,s,t\in\Gamma(E),\;\forall\,f\in C^\infty(M,\mathbb{R}).
\end{align*}
Let $(M,\phi,\xi,\eta,g)$ be a Sasakian manifold with $\dim_{\mathbb{R}}M=2n+1$.  Then a foliated/transverse complex vector bundle on $M$ is a pair of $(E,\nabla^0),$ where $E$ is a smooth complex vector bundle on $M$ and $\nabla^0$ is a partial connection in the direction $L_\xi.$ A foliated/transverse Hermitian structure $h$ on $(E,\nabla^0)$ is a smooth Hermitian structure on the complex vector bundle $E$ which is preserved by $\nabla^0.$ We call $(E,\nabla^0,h)$ foliated/transverse complex Hermitian vector bundle. We use the notation
\begin{equation*}
\Gamma_{\mathrm{B}}(E):=\left\{s\in \Gamma(E):\;\nabla^0_\xi s=0\right\}.
\end{equation*}
The foliated/transverse Hermitian structure $h$ means a reduction of structure group of the associated frame bundle of $E$ from $\mathrm{GL}(r,\mathbb{C})$ with $\mathrm{rank}E=r$ to $\mathrm{U}(r)$ as a foliated principle bundle. Such a reduction does not always exist and hence not every foliated/transverse complex vector bundle admits a foliated/transverse Hermitian structure.

A foliated/transverse holomorphic vector bundle on $M$ is a pair of $(\{E,\nabla^0\},\nabla^1)$, where $(E,\nabla^0)$ is a foliated/transverse complex vector bundle and $\nabla^1$ is a flat partial connection of $E$ in the direction $(L_\xi\otimes_{\mathbb{R}}\mathbb{C})\oplus \mathcal{D}^{0,1}$ (an involutive subbundle) and coincide with $\nabla^0$ in the direction $\xi$. We call $(\{E,\nabla^0\},\nabla^1,h)$ foliated/transverse holomorphic Hermitian vector bundle. Given such a foliated/transverse holomorphic Hermitian vector bundle $(\{E,\nabla^0\},\nabla^1,h)$,  the adapted Chern connection $\nabla^{\mathrm{C}}$ on $E$ is the unique connection which preserve $h$ and coincides with $\nabla^1$ in the direction $(L_\xi\otimes_{\mathbb{R}}\mathbb{C})\oplus \mathcal{D}^{0,1}.$ We denote by $\mathcal{K}(E,h)$ the curvature of the connection $\nabla^{\mathrm{C}},$ which is a basic $(1,1)$ form with values in $\mathrm{End}(E).$ By the Chern-Weil theory \cite{chernclass}, the Chern form $c_j(E,h)$ of the foliated/transverse holomorphic Hermitian vector bundle $(\{E,\nabla^0\},\nabla^1,h)$ is defined by
\begin{equation*}
\det\left(\mathrm{Id}_E+\frac{\sqrt{-1}}{2\pi}\mathcal{K}(E,h)\right)=\sum_{i\geq 0}c_i(E,h),
\end{equation*}
where $c_i(E,h)$ is a closed basic real $(i,i)$ form for $i\geq 0.$ We say that
$$
c_i^{\mathrm{BC,b}}(E)=\{c_i(E,h)\}\in H^{i,i}_{\mathrm{BC,b}}(M,\nu(\mathcal{F}_\xi)):=\frac{\{\mathrm{d}_{\mathrm{B}}\text{-closed basic real $(i,i)$ forms} \}}{\sqrt{-1}\partial_{\mathrm{B}}\bar\partial_{\mathrm{B}} \{\text{basic real $(i-1,i-1)$ forms}\}}
$$
is the $i^{\mathrm{th}}$ basic Chern class of $E.$

\subsection{Local Coordinates on Sasakian Manifolds}
\label{seclocalcoordinate}
Let $(M,\,\phi,\,\xi,\,\eta,\,g)$ be a compact Sasakian manifold with $\dim_{\mathbb{R}}M=2n+1$. Then there is a foliated atlas $ \mathscr{U}=\{(U_\alpha,\varphi_\alpha)\} $  (see for example \cite[Theorem 1]{gkn00}) given by $$\varphi_\alpha=(x^{(\alpha)},z^{(\alpha)}_1,\cdots,z^{(\alpha)}_n):U_\alpha\to (-a,a)\times V_\alpha\subset \mathbb{R}\times \mathbb{C}^n$$
such that
\begin{equation*}
f_\beta=\tau_{\beta\alpha}\circ f_\alpha,\quad \text{on}\quad U_\alpha\cap U_\beta\not=\emptyset,
\end{equation*}
where $f_\alpha:\,U_\alpha\to V_\alpha$ is the natural composition of projection and $\varphi_\alpha$
and $\{\tau_{\alpha\beta}\}$ is a family of bi-holomorphic maps defined by
$$\tau_{\beta\alpha}:\,f_\alpha(U_\alpha\cap U_\beta)\to f_\beta(U_\alpha\cap U_\beta),\quad U_\alpha\cap U_\beta=\emptyset$$
 satisfying the cocycle conditions
\begin{equation*}
\tau_{\gamma\alpha}=\tau_{\gamma\beta}\circ\tau_{\beta\alpha},\quad\text{on}\quad U_\alpha\cap U_\beta\cap U_\gamma\not=\emptyset.
\end{equation*}
Moreover, on such a foliated coordinate patch $(U;x,z_1,\cdots,z_n)$, there holds
\begin{enumerate}
\item[1)] the Reeb vector field $\xi=\partial_x:=\partial/\partial x;$
\item[2)] a smooth function $K:\,V\to \mathbb{R}$ is basic, i.e., $\xi (K)=0;$
\item[3)]the contact form $\eta$ with
\begin{equation*}
\eta=\mathrm{d}x-\sqrt{-1}\sum_{i=1}^nK_i\mathrm{d}z_i
+\sqrt{-1}\sum_{j=1}^nK_{\overline{j}}\mathrm{d}\overline{z}_j,
\end{equation*}
where $K_i:=\partial K/\partial z_i,\,K_{\overline{j}}:=\partial K/\partial \overline{z}_j;$
\item[4)]the metric $g$ with
\begin{equation*}
g=\eta\otimes\eta+\sum_{i,j=1}^nK_{i\overline{j}}(\mathrm{d}z_i\otimes\mathrm{d}\overline{z}_j+
\mathrm{d}\overline{z}_j\otimes\mathrm{d}z_i),
\end{equation*}
where $K_{i\overline{j}}:=\partial^2K/\partial z_i\partial\overline{z}_j;$
\item[5)] the tensor field $\phi$ with
\begin{equation*}
  \phi=\sqrt{-1}\sum_{i=1}^n\left(\partial_i+\sqrt{-1}K_i\partial_x\right)\otimes \mathrm{d}z_i
  -\sqrt{-1}\sum_{j=1}^n\left(\partial_{\overline{j}}-\sqrt{-1}K_{\overline{j}}\partial_x\right)\otimes \mathrm{d}\overline{z}_j,
\end{equation*}
where $\partial_i:=\partial/\partial z_i,\,\partial_{\overline{j}}:=\partial/\partial\overline{z}_j;$
\item[6)] one can choose some basic smooth function $K:\,V\to \mathbb{R}$ such that
$$\xi (K)=0,\quad \partial_{\mathrm{B}} K(p)=\overline{\partial}_{\mathrm{B}}K(p)=0,\quad i,\,j=1,\cdots,n.$$
    Indeed, thanks to \cite{gkn00},  the following transformations
    $$
    K\mapsto K+f(z)+\overline{f(z)},\quad x\mapsto x+\sqrt{-1}\overline{f}-\sqrt{-1}f(z)
    $$
    where $f$ is a holomorphic function  of $z=(z_1,\cdots,z_n)$, do not change the Sasakian structure above.
We can take the transformation given by
\begin{align*}
 K\mapsto&K-2\Re\left(\sum_{i=1}^nK_i(p)z_i\right),\\
 x\mapsto&x-\sqrt{-1}\sum_{i=1}^nK_{\overline{j}}(p)\overline{z}_j+\sqrt{-1}\sum_{i=1}^nK_i(p)z_i,
 \end{align*}
as desired.
 \item[7)]   we can cover $M$ by finite foliated local coordinate charts $\{U_\alpha\}$ each of which is diffeomorphism to  $(-\varepsilon_0,\varepsilon_0)\times B_2(\mathbf{0})$ with $\varepsilon_0>0$ fixed since $M$ is compact, where $B_2(\mathbf{0})$ is the ball in $\mathbb{C}^n$ centered at origin with radius $2$, and on $B_2(\mathbf{0})$, there holds
 \begin{equation*}
C^{-1}\delta_{ij}\leq \omega_\dag\leq C\delta_{ij}
 \end{equation*}
 for a uniform constant $C$. Moreover, $\{\frac{1}{2}U_\alpha\}$ each of which is diffeomorphism to  $(-\varepsilon_0/2,\varepsilon_0/2)\times B_1(\mathbf{0})$ still cover $M.$
\end{enumerate}
Now we see that $(\mathbb{R}_+\times U;\,r,x,z_i)$ is a local coordinate patch of $X:=\mathbb{R}_+\times M $. We have
\begin{align*}
J\partial_x=&-r\partial_r,\quad
J\partial_r=\frac{1}{r}\partial_x,\\
 J\partial_i=&\sqrt{-1} \left(\partial_i+\sqrt{-1}K_i\partial_x\right)+\sqrt{-1}K_ir\partial_r,\\
J\partial_{\overline{j}}=&-\sqrt{-1} \left(\partial_{\overline{j}}-\sqrt{-1}K_{\overline{j}}\partial_x\right)-\sqrt{-1}K_{\overline{j}}r\partial_r.
\end{align*}For any $p$ form $\vartheta$, we define
$$
(J\vartheta)(X_1,\cdots,X_p):=(-1)^p\vartheta(JX_1,\cdots,JX_p),\quad \forall\,X_1,\cdots,X_p\in\mathfrak{X}(X).
$$
Then we have that
\begin{align*}
J(\mathrm{d}z_i)=&-\sqrt{-1}\mathrm{d}z_i,\quad
J(\mathrm{d}\overline{z}_j)=\sqrt{-1}\mathrm{d}\overline{z}_j,\\
J(\mathrm{d}x)=&-\frac{1}{r}\mathrm{d}r+K_i\mathrm{d}z_i+K_{\overline{j}}\mathrm{d}\overline{z}_j
=-\frac{1}{r}\mathrm{d}r+\mathrm{d}K,\\
J(\mathrm{d}r)=&r\mathrm{d}x-\sqrt{-1}rK_i\mathrm{d}z_i+\sqrt{-1}rK_{\overline{j}}\mathrm{d}\overline{z}_j
=r\mathrm{d}x+rJ\mathrm{d}K=r\eta.
\end{align*}
It follows that $(\mathbb{R}_+\times V;\,z_0:=\log r-K+\sqrt{-1}x,z_1,\cdots,z_n)$ is a local holomorphic coordinate patch of $X$, which is first proved by \cite[Lemma 2.1]{hl18}. Indeed, since $\mathrm{d}z_0\wedge\cdots\wedge\mathrm{d}z_n\wedge \mathrm{d}\overline{z}_0\wedge\cdots\wedge\mathrm{d}\overline{z}_n\not=0,$ we see that $(z_0,\cdots,z_n)$ is local coordinates. Furthermore,  the equalities above yield that
$$
 J(\mathrm{d}z_i)=-\sqrt{-1}\mathrm{d}z_i,\quad i=0,1,\cdots,n,
$$
as desired.

For later use, we give an equivalent statement of  (holomorphic) foliated/transverse complex vector bundle on Sasakian manifold $M$ in the  \v{C}ech language, and the equivalence follows from an argument using the Frobenius theorem. A complex vector bundle $E$ with rank $r$ on $M$ is foliated/transverse if there exists a family of  local trivialization maps $\phi_U:\;U \times \mathbb{C}^r\to E_{\upharpoonright U}$ a foliated open covering $\mathscr{U}=\{U,V,\cdots\}$ of $M$ satisfying
\begin{enumerate}
\item the transition functions $\{g_{UV}\in\mathrm{GL}(r,\mathbb{C})\}$ satisfy $$\xi (g_{UV})\equiv\mathbf{0}_{r\times r},$$
where
 $$g_{UV}(\mathbf{p}):= \phi_{U,\mathbf{p}}^{-1}\circ \phi_{V,\mathbf{p}}$$ for any $\mathbf{p}\in U \cap V$ with $\phi_{U,\mathbf{p}}(\mathbf{v}):=\phi_U(\mathbf{p},\mathbf{v}) ,\,\forall \, \mathbf{v}\in\mathbb{C}^r;$
 \item  Any local section $s\in \Gamma_{\mathrm{B}}(U,E)$ on $U$ can be written as $$s=\sum_\alpha f_\alpha s_{U,\alpha},\quad f_\alpha\in C^\infty_{\mathrm{B}}(U,\mathbb{C}),\quad 1\leq \alpha\leq r,$$
     where $$\left\{s_{U,\alpha}:\,U\to E,\,\mathbf{p}\mapsto\phi_{U}(\mathbf{p},\mathbf{e}_\alpha) \right\}_{1\leq \alpha\leq r}$$ is a (basic holomorphic if $E$ is holomorphic foliated/transverse vector bundle) basis of $\Gamma_{\mathrm{B}}(U_\alpha,E),$ where $\mathbf{e}_\alpha$ denotes the $i^{\mathrm{th}}$ standard basis vector in $\mathbb{C}^r $ with $$\nabla^0_\xi s_{U,\alpha}=0.$$
\end{enumerate}
Furthermore,   $E$ is called holomorphic  foliated/transverse if  such a family of local trivialization maps also satisfies $$Z(g_{UV})\equiv\mathbf{0}_{r\times r},\;\nabla^1_Z s_{U,\alpha}=0,\;\forall\,Z\in \Gamma((L_\xi\otimes_{\mathbb{R}}\mathbb{C})\oplus \mathcal{D}^{0,1}),\;\forall\,1\leq \alpha\leq r.$$
We give some calculation of the adapted Chern connection on $U$ and  write $s_{U,\alpha}=s_\alpha$ for short.  We use the notations
 $$
 h_{\alpha\bar\beta}:=h(s_\alpha,s_\beta),\quad \nabla^{\mathrm{C}}_{\partial_i}s_\alpha=:\Gamma_{i\alpha}^\beta s_\beta,\quad
 \nabla^{\mathrm{C}}_{\partial_i}\nabla^{\mathrm{C}}_{\partial_{\bar j}}-\nabla^{\mathrm{C}}_{\partial_{\bar j}}\nabla^{\mathrm{C}}_{\partial_i}s_\alpha=:R_{i\bar j\alpha}{}^\beta s_\beta,\quad
 R_{i\bar j\alpha\bar\beta}:=R_{i\bar j\alpha}{}^\gamma h_{\gamma\bar\beta},
 $$
 where $h_{\alpha\bar\beta}\in C^\infty_{\mathrm{B}}(U,\mathbb{R}),\;1\leq \alpha,\beta\leq r.$
 Then we have
 \begin{equation*}
 \Gamma_{i\alpha}^\beta=h^{\bar\gamma \beta}\left(\partial_ih_{\alpha\bar\gamma}\right),
 R_{i\bar j\alpha\bar\beta}=-\partial_i\partial_{\bar j}(h_{\alpha\bar\beta})
 +h^{\bar \gamma\delta}
 \left(\partial_ih_{\alpha\bar\gamma}\right)
 \left(\partial_{\bar j}h_{\delta\bar\beta}\right).
 \end{equation*}
 The basic Chern-Ricci form is given by
 \begin{equation}
 \label{ricciforme}
 \ric(E,h)=\sqrt{-1}R_{i\bar j\alpha}{}^\alpha \mathrm{d}z_i\wedge\mathrm{d}\bar z_j
 =-\partial_{\mathrm{B}}\bar\partial_{\mathrm{B}}\log \det(h_{\alpha\bar\beta})\in 2\pi c_1^{\mathrm{BC,b}}(E).
 \end{equation}
 It follows from \cite{swz10} that $(\{\nu(\mathcal{F}_\xi),\nabla^{\mathrm{B}}\},\bar\partial_{\mathrm{B}},g_\dag)$ is a foliated/transverse holomorphic Hermitian bundle, where $\nabla^{\mathrm{B}}$ is the Bott connection given by
 $$
 \nabla^{\mathrm{B}}_\xi V:=\pi\left([\xi,\sigma(V)]\right), \quad \forall \,V\in \Gamma(\nu(\mathcal{F}_\xi)).
 $$
 It follows from \eqref{sasanablaphi} that the transverse connection $\nabla^\dag$ induced from the Levi-Civita connection $\nabla$ of $g$ is the adapted Chern connection, where for any $V\in \Gamma(\nu(\mathcal{F}_\xi)),$ the connection $\nabla^\dag$ is given by
 \begin{equation}
\label{tconnectiondefn}
\nabla^\dag_{Z}V:=\left\{
\begin{array}{ll}
  \pi\left(\nabla_Z\sigma(V)\right), & \quad\text{if}\;Z\in\Gamma(\mathcal{D}); \\
  \nabla^{\mathrm{B}}_\xi V, & \quad\text{if}\;Z=\xi.
\end{array}
\right.
\end{equation}
We set
\begin{equation}
\label{eiej}
e_i:=\partial_i+\sqrt{-1}K_i\partial_x,\quad \bar e_j=e_{\bar j}:=\partial_{\bar j}-\sqrt{-1}K_{\bar j}\partial_x.
\end{equation}
It follows that $\{\eta,\mathrm{d}z^i,\mathrm{d}\bar z^j\}$ is the dual basis of $\{\partial_x,e_i,e_{\bar j}\}$.
We deduce that
$\{\sigma^{-1}(e_i)\}_{1\leq i\leq n}$ is a basic holomorphic basis of $\Gamma_{\mathrm{B}}(U,\nu(\mathcal{F}_\xi)),$ and that $\omega_{\dag}$ is the transverse Hermitian metric on $\nu(\mathcal{F}_\xi),$ where
$$
\omega_{\dag}= \frac{1}{2}\mathrm{d}\eta:=\mn \sum_{i,j=1}^n(g_\dag)_{i\overline{j}}\md z_i\wedge\md \overline{z}_j
=\mn \sum_{i,j=1}^nK_{i\overline{j}}\md z_i\wedge\md \overline{z}_j.
$$
In the following, we will use the induced connection, also denoted by $\nabla^\dag,$ on $\bigwedge^1_{\mathrm{B}}\otimes_{\mathbb{R}}\mathbb{C}$ which is the dual of $\nu(\mathcal{F}_\xi).$ For this aim, let us fix some notations.
\begin{equation*}
\nabla^\dag_{i}:= \nabla^\dag_{\partial_i},\quad
\nabla^\dag_{\bar j}:=\nabla^\dag_{\partial_{\bar j}},\quad
\nabla^\dag_{i}\mathrm{d}z_k=-\Gamma_{ij}^k\mathrm{d}z_j, \quad
 R(\partial_i,\partial_{\bar j})\mathrm{d}z_\ell:=  -R_{i\bar jk}{}^\ell \mathrm{d}z_k.
\end{equation*}
A direct calculation yields that (see for example  \cite{swz10})
\begin{equation*}
\nabla^\dag_{\partial_x}\mathrm{d}z_k=0,\quad
\Gamma_{ij}^k= (g_\dag)^{\overline{q}k}\partial_i(g_\dag)_{j\overline{q}},\quad
 R_{i\overline{j}k}{}^{\ell}= -\partial_{\overline{j}}\Gamma_{ik}^{\ell},\quad R_{i\overline{j}k\overline{\ell}}=R_{i\overline{j}k}{}^p (g_\dag)_{p\overline{\ell}}
\end{equation*}
such that
\begin{equation*}
\overline{R_{i\bar jk\bar \ell}}=R_{j\bar i\ell\bar k},\quad
R_{i\bar jk\bar \ell}=R_{k\bar j i\bar \ell}=R_{i\bar \ell k\bar j}= R_{k\bar \ell i\bar j},
\end{equation*}
where $((g_\dag)^{\overline{q}k})$ is the inverse matrix of $((g_\dag)_{j\overline{\ell}})$.

From \eqref{tddbar} and \eqref{ricciforme}, the basic Chern-Ricci form of $ \nu(\mathcal{F}_\xi)$ given by
\begin{equation}
\label{transversericciform}
\ric(\omega_\dag)
= -\sqrt{-1}\partial_{\mathrm{B}}\bar \partial_{\mathrm{B}}\log\det((g_{\dag})_{i\bar j})
= -\frac{\sqrt{-1}}{2}\mathrm{d}_{\mathrm{B}}\phi\mathrm{d}_{\mathrm{B}}\log\det((g_{\dag})_{i\bar j})
\in 2\pi c_1^{\mathrm{BC,b}}(\nu(\mathcal{F}_\xi))
\end{equation}
is a $\mathrm{d}_{\mathrm{B}}$-closed real basic $(1,1)$ form.

For a basic $(1,0)$ form $a=a_{\ell}\md z^{\ell}$,  we define covariant derivative $\nabla^\dag_ia_{\ell}$ by
\begin{align*}
\nabla^\dag_{i}a_{\ell}:=\partial_{i}a_{\ell}-\Gamma_{i\ell}^pa_p.
\end{align*}
Then we can deduce
\begin{align}\label{commutate}
[\nabla^\dag_{i},\nabla^\dag_{\overline{j}}]a_{\ell}=-R_{i\overline{j}\ell}{}^pa_p,\quad [\nabla^\dag_{i},\nabla^\dag_{\overline{j}}]a_{\overline{m}}=R_{i\overline{j}}{}^{\overline{q}}{}_{\overline{m}}a_{\overline{q}},
\end{align}
where $R_{i\overline{j}}{}^{\overline{q}}{}_{\overline{m}}=R_{i\overline{j}p}{}^{\ell}(g_\dag)^{\overline{q}p}(g_\dag)_{\ell \overline{m}}$.

For any $u\in C_{\mathrm{B}}^{\infty}(M,\mathbb{R})$, we have
\begin{align}\label{ricciidentityu}
\nabla^\dag_{i}u=\partial_{i}u=:u_i,\quad \nabla^\dag_{\overline{j}}u=\partial_{\overline{j}}u=:u_{\overline{j}},\quad \nabla^\dag_{\overline{j}}\nabla^\dag_{i}u=\partial_{\overline{j}}\partial_{i}u=:u_{i\overline{j}},\quad [\nabla^\dag_{i},\nabla^\dag_j]u=0.
\end{align}
Using \eqref{commutate}, we can get the following commutation formulae:
\begin{align}\label{ricciidentity}
\nabla^\dag_{\ell}u_{i\overline{j}}=&\nabla^\dag_{\overline{j}}\nabla^\dag_{\ell}u_i-R_{\ell\overline{j}i}{}^pu_p,\; \nabla^\dag_{\overline{m}}u_{p\overline{j}}=\nabla^\dag_{\overline{j}}u_{p\overline{m}},
\;\nabla^\dag_{\ell}u_{i\overline{q}}=\nabla^\dag_{i}u_{\ell\overline{q}},\\
\nabla^\dag_{\overline{m}}\nabla^\dag_{\ell}u_{i\overline{j}}=&\nabla^\dag_{\overline{j}}\nabla^\dag_{i}u_{\ell\overline{m}}
+R_{\ell\overline{m}i}{}^pu_{p\overline{j}}-R_{i\overline{j}\ell}{}^pu_{p\overline{m}}.
\nonumber
\end{align}
For any basic $(p,q)$ form
$$
\vartheta=\frac{1}{p!q!}\vartheta_{i_1\cdots i_p\overline{j_1}\cdots\overline{j_q}}\md z_{i_1}\wedge\cdots\wedge\md z_{i_p}\wedge\md \overline{z}_{j_1}\wedge\cdots\wedge\md \overline{z}_{j_q},
$$
the equalities \eqref{defnvol} and \eqref{defnbardag} yield that (see for example \cite{luqikeng} in the K\"ahler setup)
\begin{align}\label{starformulacomplex}
 \ast_{\dag} \vartheta =&\frac{(\mn)^n(-1)^{np+\frac{n(n-1)}{2}}\det g_{\dag}}{(n-p)!(n-q)!p!q!}\vartheta_{i_1\cdots i_p\overline{j_1}\cdots\overline{j_q}}g_{\dag}^{\overline{\ell_1}i_1}
\cdots g_{\dag}^{\overline{\ell_p}i_p}g_{\dag}^{\overline{j_1}k_1}\cdots g_{\dag}^{\overline{j_q}k_q}\\
&\delta_{\ell_1\cdots \ell_pb_1\cdots b_{n-p}}^{1\cdots\cdots n}\delta_{k_1\cdots k_qa_1\cdots a_{n-q}}^{1\cdots\cdots n}
\md z_{a_1}\wedge\cdots\wedge\md z_{a_{n-q}}\wedge\md \overline{z}_{b_1}\wedge\cdots\wedge\md\overline{z}_{b_{n-p}}\nonumber
\end{align}
and
\begin{align*}
\varpi\wedge\ast_{\dag} \overline{\vartheta}=\frac{1}{p!q!}\varpi_{\ell_1\cdots\ell_p\overline{k_1}\cdots \overline{k_{q}}} \overline{\vartheta_{i_1\cdots i_p\overline{j_1}\cdots\overline{j_q}}}g_{\dag}^{\overline{i_1}\ell_1}
\cdots g_{\dag}^{\overline{i_p}\ell_p}g_{\dag}^{\overline{k_1}j_1}\cdots g_{\dag}^{\overline{k_q}j_q} \frac{\omega_{\dag}^n}{n!},
\end{align*}
where $\varpi=\frac{1}{p!q!}\varpi_{i_1\cdots i_p\overline{j_1}\cdots\overline{j_q}}\md z_{i_1}\wedge\cdots\wedge\md z_{i_p}\wedge\md \overline{z}_{j_1}\wedge\cdots\wedge\md \overline{z}_{j_q}$ is another basic $(p,q)$ form and $\det g_{\dag}=\det((g_\dag)_{i\overline{j}})$. We also define the inner product of $\varpi$ and $\vartheta$ by
$$
\langle\varpi,\,\vartheta\rangle_{\mathrm{B}}:=\frac{1}{p!q!}\varpi_{\ell_1\cdots\ell_p\overline{k_1}\cdots \overline{k_{q}}} \overline{\vartheta_{i_1\cdots i_p\overline{j_1}\cdots\overline{j_q}}}g_\dag^{\overline{i_1}\ell_1}
\cdots g_\dag^{\overline{i_p}\ell_p}g_\dag^{\overline{k_1}j_1}\cdots g_\dag^{\overline{k_q}j_q}.
$$
Note that
$$
\ast_\dag 1=\ast \eta=\frac{\omega_\dag^n}{n!},
\quad \overline{\ast_\dag\vartheta}= \ast_\dag\overline{ \vartheta},
$$
where the second equality shows that $\ast_\dag$ is a real operator.

We fix a canonical orientation $\eta\wedge(\mathrm{d}\eta)^n$ and introduce the concepts of transverse  positivity on Sasakian manifolds (see \cite{co15} and \cite[Chapter III]{demaillybook1} for the complex case).

A basic $(p,p)$ form $\varpi$ is said to be transverse positive   if for any basic $(1,0)$ forms $\gamma_j,\,1\leq j\leq n-p$, then
$$
\varpi\wedge\sqrt{-1}\gamma_1\wedge\overline{\gamma_1}\wedge\cdots\wedge\sqrt{-1} \gamma_{n-p}\wedge\overline{\gamma_{n-p}}\wedge\eta
$$
is a positive volume form, and is said to be strongly transverse positive if $\varpi$ is a convex combination
\begin{equation*}
\varpi=\sum\Gamma_s\sqrt{-1}\gamma_{s,1} \wedge\overline{\gamma_{s,1}}\wedge\cdots\wedge\sqrt{-1} \gamma_{s,p}\wedge\overline{\gamma_{s,p}},
\end{equation*}
where $\gamma_{s,j}$'s are basic $(1,0)$ forms and $\Gamma_s\geq 0.$

Any transverse positive basic $(p,p)$ form $\varpi$ is real, i.e., $\overline{\varpi}=\varpi$. In particular, in the local coordinates, a real basic $(1,1)$ form
\begin{align}\label{11}
\upsilon=\mn\upsilon_{i\overline{j}}\md z_i\wedge\md\overline{z}_j
\end{align}
is transverse positive if and only if $(\upsilon_{i\overline{j}})$ is a semi-positive Hermitian matrix with $\xi(v_{i\bar j})\equiv0$ and we denote $\det \upsilon:=\det(\upsilon_{i\overline{j}})$.

Similarly, a real basic $(n-1,n-1)$ form
\begin{align}\label{tn-1}
\varrho=&(\mn)^{n-1}\sum\limits_{i,j=1}^n(-1)^{\frac{n(n+1)}{2}+i+j+1}\varrho^{\overline{j}i}\\
&\md z_1\wedge\cdots\wedge\widehat{\md z_i}\wedge\cdots\wedge\md z_n\wedge \md \overline{z}_1\wedge\cdots\wedge\widehat{\md\overline{z}_j}\wedge\cdots\wedge\md \overline{z}_n\nonumber
\end{align}
is transverse positive if and only if $(\varrho^{\overline{j}i})$ is a semi-positive Hermitian matrix with $\xi(\varrho^{\overline{j}i})\equiv0$ and we denote $\det\varrho:=\det(\varrho^{\overline{j}i})$.

We remark that one can call a real basic $(1,1)$ form $\upsilon$ ( resp. a real basic $(n-1,n-1)$ form $\varrho$) strictly transverse positive (denoted by $>_{\mathrm{b}}0$)
if the
Hermitian matrix $(\upsilon_{i \overline{j}})$ (resp. $(\varrho^{\overline{j} i})$) is positive
definite.

For a strictly transverse positive basic $(1,1)$ form $v$ defined as in \eqref{11}, we can deduce a strictly transverse positive $(n-1,n-1)$ form
\begin{align}\label{11n1n1formula}
\frac{v^{n-1}}{(n-1)!}=&(\mn)^{n-1}\sum\limits_{k,\ell=1}^n(-1)^{\frac{n(n+1)}{2}+k+\ell+1}\mathrm{det}(v_{i\overline{j}})\tilde{v}^{\overline{\ell}k}\\
&\md z_{ 1}\wedge\cdots\wedge\widehat{\md z_k}\wedge \cdots\wedge\md z_{n}\wedge\md\overline{z}_{ 1}\wedge\cdots\wedge\widehat{\md \overline{z}_{\ell}}\wedge\cdots\wedge\cdots\wedge \md\overline{z}_{n}\nonumber
\end{align}
where $(\tilde{v}^{\overline{\ell}k} )$ is the inverse matrix of $(v_{i\overline{j}})$, i.e.,  $\sum\limits_{\ell=1}^n\tilde{v}^{\overline{\ell}j}v_{k\overline{\ell}}=\delta_{k}^j$. Hence we have
\begin{align}\label{detchin-1}
\det\left(\frac{v^{n-1}}{(n-1)!}\right)=\left(\det v\right)^{n-1}.
\end{align}
From  \eqref{tn-1} and \eqref{11n1n1formula}, we get
 \begin{lem}
 \label{lemjia}
 There exists a bijection from the space of strictly transverse positive definite $(1,1)$ forms to strictly transverse positive definite $(n-1,n-1)$  forms, given by
\begin{align*}
v\mapsto \frac{v^{n-1}}{(n-1)!}.
\end{align*}
 \end{lem}
The above bijection is first proved in \cite{michelsohn} in the K\"ahler setup (cf. \cite[Formula (3.3)]{phong}).

For a basic real $(1,1)$ form $\phi$ defined as in \eqref{11} (need not be positive), \eqref{starformulacomplex} implies
\begin{align}
\ast_\dag\phi
=&(\mn)^{n-1}\sum\limits_{k,\ell=1}^n(-1)^{\frac{n(n-1)}{2}+n+k+\ell+1}(\det\omega)\phi^{\overline{\ell}k}\\
&\md z_1\wedge\cdots\wedge\widehat{\md z_k}\wedge\cdots\wedge \md z_n\wedge\md \overline{z}_1\wedge\cdots\wedge\widehat{\md\overline{z}_{\ell}}\wedge\cdots\wedge\md\overline{z}_n,\nonumber
\end{align}
where $\phi^{\overline{\ell}k}=g^{\overline{\ell}i}\phi_{i\overline{j}}g^{\overline{j}k}$. Hence, if $\xi$ is another basic real $(1,1)$ form with $\det\xi\neq0$, then we can deduce
\begin{align}\label{astdet}
\frac{\det(\ast_\dag\phi)}{\det(\ast_\dag\xi)}=\frac{\det\phi}{\det\xi}.
\end{align}
We need the following useful formulae   and the proofs are direct and complicated computation.

For any basic real $(1,1)$ form $\chi=\mn\chi_{i\overline{j}}\md z_i\wedge\md \overline{z}_j$, we have
\begin{align}
\label{formula1}
\ast_\dag(\chi\wedge\omega^{n-2})=(n-2)!\left[(\tr_{\omega_\dag}\chi)\omega_\dag-\chi\right].
\end{align}

For any $u\in C_{\mathrm{B}}^\infty(M,\mathbb{R})$, we define
\begin{equation}
  \label{defnlapt}
\Delta_{\mathrm{B}}u:=\frac{n\mathrm{d}_{\mathrm{B}}\phi\mathrm{d}_{\mathrm{B}}u\wedge\omega_\dag^{n-1}}{2\omega_\dag^n}
=\frac{n\sqrt{-1}\partial_{\mathrm{B}}\overline{\partial}_{\mathrm{B}}u\wedge\omega_\dag^{n-1}}
{\omega_\dag^{n}}
=g_\dag^{\overline{j}i}\partial_i\partial_{\overline{j}}u.
\end{equation}
If $M$ is a Sasakian manifold, then we have $\Delta_{\mathrm{B}}u=-\frac{1}{2}\Delta_{\mathrm{d}_{\mathrm{B}}}u$; otherwise, the difference of these two operators is another operator with order no larger than one (see \cite{kt87,kt88}).
\section{Zero Order Estimate}\label{sec0order}
In this section, we prove the $L^\infty(M)$ estimate of the solution to \eqref{equ1}.
\begin{thm}
\label{thm0order}
Suppose that $u\in C_{\mathrm{B}}^\infty(M,\mathbb{R})$ is a solution to \eqref{equ1} with $\sup_Mu=0,$ and that the function $\underline{u}\in C_{\mathrm{B}}^\infty(M,\mathbb{R})$ is a transverse $\mathcal{C}$-subsolution to \eqref{equ1}. There exists a uniform constant $C$ depending only on $\eta,\,g,\,\omega_\dag$ and $\underline{u}$ such that
\begin{equation}
\label{equ0order}
\sup_M|u|\leq C.
\end{equation}
\end{thm}
\begin{proof}
We denote by $\Delta_g$ the usual Laplacian of the Riemannian metric $g$. It follows from \cite{ek86} that $\Delta_g$ is the extension of $-\Delta_{\mathrm{d}_{\mathrm{B}}}$, i.e., for any   $u\in C_{\mathrm{B}}^\infty(M,\mathbb{R})$, we have
\begin{equation*}
\Delta_gu=\frac{2n\sqrt{-1}\partial_{\mathrm{B}}\overline{\partial}_{\mathrm{B}}u\wedge\omega_\dag^{n-1}\wedge\eta}
{\omega_\dag^n\wedge\eta}
=2\Delta_{\mathrm{B}}u=-\Delta_{\mathrm{d}_{\mathrm{B}}}u.
\end{equation*}
Note that for general $p$ form ($p\geq 1$), the extension of $\Delta_{\mathrm{d}_{\mathrm{B}}}$ is a strongly elliptic second order operator $\tilde\Delta=\Delta_{\mathrm{d}}-\tilde \varsigma$, where $\tilde\varsigma$ is an operator with order no larger than one (see \cite{kt87,kt88}).

Recall that $\Gamma\subset\left\{\mathbf{x}=(x_1,\cdots,x_n)\in\mathbb{R}^n:\,\sum_{i=1}^nx_i>0\right\}$ (see \cite{cnsacta}).  Indeed, if $\mathbf{x}\in \Gamma$, then the symmetry and convexity of $\Gamma$  yields that
$$
\frac{\sum_ix_i}{n} \mathbf{1}\in \Gamma,
$$
where $\mathbf{1}$ is the $n$-tuple with each component $1.$
If there exists a point $\mathbf{x}\in \Gamma$ with $\sum_ix_i\leq 0$,  then we get that the origin
$
\mathbf{0}\in \Gamma
 $
 by the convexity of $\Gamma\supset\Gamma_n$, a contradiction.
This yields that $\mathrm{tr}_{\omega_\dag}h_\dag=\sum_{i=1}^n(A_{\dag,u})_i{}^i>0$, and hence
\begin{equation}
\label{deltaulowerbound}
\Delta_{\mathrm{B}}u=\sum_{i=1}^n(A_{\dag,u})_i{}^i-\frac{n\beta_\dag\wedge\omega_\dag^{n-1}}{\omega_\dag^n}
>-\frac{n\beta_\dag\wedge\omega_\dag^{n-1}}{\omega_\dag^n}\geq-C.
\end{equation}
Suppose that $u(\mathbf{p})=\sup_Mu=0$. We have
\begin{equation}
\label{greenu}
0=u(\mathbf{p})=\frac{1}{\int_M\eta\wedge\omega_\dag^n}\int_Mu(\mathbf{y})\eta(\mathbf{y})\wedge\omega_\dag^n(\mathbf{y})
-\frac{1}{n!}\int_MG(\mathbf{p},\mathbf{y})\Delta_gu(\mathbf{y})\eta(\mathbf{y})\wedge\omega_\dag^n(\mathbf{y}),
\end{equation}
where $G(\mathbf{q},\mathbf{y})$ is Green's function associated to $\Delta_g$, normalized so that
\begin{equation*}
G(\mathbf{q},\mathbf{y})\geq0\quad \text{and}\quad\int_MG(\mathbf{q},\mathbf{y}) \eta(\mathbf{y})\wedge\omega_\dag^n(\mathbf{y})\leq C,\quad \forall\;\mathbf{q}\in M.
\end{equation*}
From \eqref{deltaulowerbound} and \eqref{greenu}, it follows that
\begin{equation}
\label{l1est}
  \int_M(-u)\eta\wedge\omega_\dag^n\leq C.
\end{equation}
It is sufficient to obtain the lower bound of $L:=\inf_Mu<0$ to get \eqref{equ0order}. We assume that the transverse $\mathcal{C}$-subsolution $\underline{u}=0$; otherwise we can change $\beta_\dag$ to get this.
Then by Definition \ref{tcsubsol}, the set
\begin{equation*}
\left(\lambda\left(A_{\dag,0}\right)+\Gamma_n\right) \cap \Gamma^{\psi(\mathbf{x})},\quad
\forall\;\mathbf{x}\in M
\end{equation*}
is uniformly bounded. There exist $\delta>0$ and $R>0$ such that there holds
\begin{equation}
\label{eigenbd}
\left(\lambda\left(A_{\dag,0}\right)-\delta\mathbf{1}+\Gamma_n \right)\cap \Gamma^{\psi(\mathbf{x})}\subset B_R(\mathbf{0}),\quad\forall\;\mathbf{x}\in M.
\end{equation}
We assume that $u$ attains its minimum at the origin of the foliated chart $(-\varepsilon_0,\varepsilon_0)\times B_2(\mathbf{0})$; otherwise we can get this by translation. Let us work in $B_1(\mathbf{0})$. For $\epsilon>0$ sufficiently small, we set $v=u+\epsilon|\mathbf{z}|^2$. We have
\begin{equation*}
\inf_{|\mathbf{w}|\leq 1}v(\mathbf{w})=v(\mathbf{0})= L,\quad v (\mathbf{z})\geq L+\epsilon,\quad \forall\;\mathbf{z}\in \partial B_1(\mathbf{0}).
\end{equation*}
It follows from \cite[Proposition 11]{gaborjdg} that
\begin{equation}
\label{plebesgue}
c_0\epsilon^{2n}\leq \int_P\det(D^2v),
\end{equation}
where the integration is respect to the Lebesgue measure, and the set $P$ is given by
\begin{equation}
 \label{setp}
 P:=\left\{\mathbf{x}\in B_1(\mathbf{0}):|Dv(\mathbf{x})|\leq \frac{\epsilon}{2},\quad v(\mathbf{y})\geq v(\mathbf{x})+Dv(\mathbf{x})\cdot(\mathbf{y}-\mathbf{x}),\quad\forall\, \mathbf{y}\in B_1(\mathbf{0})\right\}.
 \end{equation}
For any $\mathbf{x}\in P$,  we have $|Dv(\mathbf{x})|\leq \frac{\epsilon}{2}$ and $D^2v(\mathbf{x})\geq 0$ which shows $\partial_i\partial_{\overline{j}}u(\mathbf{x})\geq -\epsilon \delta_{ij}$ and that
 \begin{equation}
   \label{blockiinequality}
 \det(D^2v)\leq 2^{2n}\left(\det(\partial_i\partial_{\overline{j}}v)\right)^2
 \end{equation}
 from the argument in \cite{blockiscience}. We choose $\epsilon$ sufficiently small depending only on $\delta$ and $\omega_\dag$ such that
 \begin{equation}
 \label{eigenbd1}
 \lambda\left(A_{\dag,u}\right)\in
 \lambda\left(A_{\dag,0}\right)-\delta\mathbf{1}+\Gamma_n,\quad \forall\;\mathbf{x}\in P.
 \end{equation}
 On the other hand, since $u$ is a solution to \eqref{equ1}, we have
 \begin{equation}
 \label{eigenbd2}
 \lambda\left(A_{\dag,u}\right)\in\partial \Gamma^{\psi(\mathbf{x})},\quad \forall\;\mathbf{x}\in P.
 \end{equation}
 From \eqref{eigenbd}, \eqref{eigenbd1} and \eqref{eigenbd2}, we deduce that
 $|u_{i\bar j}|$ and hence $|v_{i\bar j}|$ is bounded from above at any point $\mathbf{x}\in P$.
 This, together with \eqref{plebesgue} and \eqref{blockiinequality}, yields that
 \begin{equation}
 \label{volpomegalower}
 c_0\epsilon^{2n}\leq C'\mathrm{Vol}_{\omega_\dag}(P).
 \end{equation}
 From \eqref{setp}, we have
 \begin{equation*}
 v(\mathbf{x})<L+\epsilon/2<0,
 \end{equation*}
 where without loss of generality we assume that $L\ll-1$, from which we have
 \begin{equation}
 \label{volpomegaupper}
 \mathrm{Vol}_{\omega_\dag}(P)\leq C''\frac{\int_M(-v)\eta\wedge\omega_\dag^n}{|L+\epsilon/2|}.
 \end{equation}
 Thanks to \eqref{l1est}, \eqref{volpomegalower} and \eqref{volpomegaupper}, we get that $L$ is uniformly bounded from below, as required.

 There is another more local argument for obtaining a bound for $\||u|^p\|_{L^1}$ for some $p>0$ by using the weak Harnack inequality \cite[Theorem 9.22]{gt1998}. Indeed, we assume that $M$ is covered by finite the foliated charts $U_i$'s diffeomorphism to $(-\varepsilon_0,\varepsilon_0)\times B_{2}(\mathbf{0})\subset \mathbb{R}\times \mathbb{C}^n$ such that
$\{\frac{1}{2}U_i\}$ each of which is diffeomorphism to  $(-\varepsilon_0/2,\varepsilon_0/2)\times B_1(\mathbf{0})$ still covers $M.$ We work with the quantities of complex variables in the balls $B_{2}(\mathbf{0})$ and hence the upper bound for $\||u|^p\|_{L^1}$ follows from \eqref{deltaulowerbound} and the argument in the proof of \cite[Proposition 10]{gaborjdg}.
 \end{proof}
 \section{Second Order Estimate}\label{sec2order}
 In this section, we prove the second order estimate of the solution to \eqref{equ1}.
\begin{thm}
\label{thm2order}
Suppose that the function $u\in C_{\mathrm{B}}^\infty(M,\mathbb{R})$ is a solution to \eqref{equ1} with $\sup_Mu=0,$ and that the function $\underline{u}\in C_{\mathrm{B}}^\infty(M,\mathbb{R})$ is a transverse $\mathcal{C}$-subsolution to \eqref{equ1}. There exists a uniform constant $C$ depending only on $\eta,\,\omega_\dag$ and $\underline{u}$ such that
\begin{equation}
\label{equ2order}
\sup_M|\partial_{\mathrm{B}}\overline{\partial}_{\mathrm{B}}u|_{\omega_\dag}\leq CK,
\end{equation}
where $K:=1+|\nabla u|_{\omega_\dag}^2$.
\end{thm}
We need some preliminaries from \cite{gaborjdg} (cf. \cite{trudinger1995,guan2014}). For any $\sigma>\sup_{\partial\Gamma}f$, the set $\Gamma^\sigma=\{\lambda\in\Gamma:\;f(\lambda)>\sigma\}$ is open and convex since $f$ is a smooth symmetric concave function, and $\partial\Gamma^\sigma=f^{-1}(\sigma)$ is a smooth hypersurface since $f_i>0$ for $1\leq i\leq n$.
For any $\lambda\in\partial\Gamma^\sigma,$ we denote, by $\mathbf{n}(\lambda),$  the inward pointing unit normal vector, i.e.,
\begin{equation*}
\mathbf{n}(\lambda):=\frac{\nabla f}{|\nabla f|}(\lambda).
\end{equation*}
We also write $\mathcal{F}(\lambda):=\sum_{k=1}^nf_k(\lambda).$ It follows from the Cauchy-Schwarz inequality that $|\nabla f|\leq \mathcal{F}\leq \sqrt{n}|\nabla f|$.

Following \cite{trudinger1995}, we set
\begin{equation*}
\Gamma_\infty:=\left\{(\lambda_1,\cdots,\lambda_{n-1}):\;(\lambda_1,\cdots,\lambda_n)\in\Gamma\;\text{for some}\;\lambda_n\right\}.
\end{equation*}
For any $\mu\in\mathbb{R}^n$, the set $\left(\mu+\Gamma_n\right)\cap \partial\Gamma^{\sigma}$ is bounded, if and only if
\begin{equation}
\label{conbd}
\lim_{t\to+\infty}f(\mu+t\mathbf{e}_i)>\sigma,\quad\forall\;1\leq i\leq n,
\end{equation}
where $\mathbf{e}_i$ denotes the $i^{\mathrm{th}}$ standard basis vector.  This limit is well defined as long as any $(n-1)$ tuple $\mu'$ in $\mu$ satisfies $\mu'\in\Gamma_\infty$, i.e., on the set $\tilde \Gamma$ given by
\begin{equation*}
\tilde\Gamma:=\left\{\mu\in\mathbb{R}:\;\text{there exists}\;t>0\;\text{such that}\;\mu+t\mathbf{e}_i\in\Gamma\;\text{for all}\;i\right\}.
\end{equation*}
It is easy to see that $\Gamma\subset\tilde\Gamma.$
For any $\lambda'=(\lambda_1,\cdots,\lambda_{n-1})\in\Gamma_\infty,$ it follows from the concavity of $f$ that the limit
\begin{equation*}
\lim_{\lambda_n\to+\infty}f(\lambda_1,\cdots,\lambda_n)
\end{equation*}
is either finite for all $\lambda'$ or infinite for all $\lambda'$ (see \cite{trudinger1995}).

If the limit is infinite, then $\left(\mu+\Gamma_n\right)\cap \partial\Gamma^\sigma$ is bounded for all $\sigma$ and $\mu\in\tilde\Gamma$. In particular, any transverse admissible $\underline{u}$ is a transverse $\mathcal{C}$-subsolution; however, a transverse $\mathcal{C}$-subsolution need not be admissible.

If the limit is finite, then we define the function $f_\infty$ on $\Gamma_\infty$ by
\begin{equation}
f_\infty(\lambda_1,\cdots,\lambda_{n-1})=\lim_{t\to+\infty}f(\lambda_1,\cdots,\lambda_{n-1},t).
\end{equation}
In this case, for $\mu\in\tilde\Gamma,$ the set $\left(\mu+\Gamma_n\right)\cap\partial\Gamma^\sigma$ is bounded if and only if $f_\infty(\mu')>\sigma$, where $\mu'\in\Gamma_\infty$ denotes any $(n-1)$ tuple of entries of $\mu.$
\begin{prop}[Sz\'ekelyhidi \cite{gaborjdg}]
\label{gaborprop5}
Given $\delta,\,R>0,$ if $\mu\in\mathbb{R}^n$ such that
\begin{equation}
\label{equgabor22}
\left(\mu-2\delta\mathbf{1}+\Gamma_n\right)\cap \partial\Gamma^\sigma\subset B_R(\mathbf{0}),
\end{equation}
where $B_R(\mathbf{0})\subset\mathbb{R}^n$ is the ball with center $\mathbf{0}$ and radius $R$,
then there exists a constant $\kappa>0$ depending only on $\delta$ and $\mathbf{n}$ on $\partial\Gamma^\sigma$ such that for any $\lambda\in\partial\Gamma^\sigma$ with $|\lambda|>R$, we have either
\begin{equation}
\sum_{j=1}^nf_j(\lambda)(\mu_j-\lambda_j)>\kappa\mathcal{F},
\end{equation}
or
\begin{equation}
f_i(\lambda)>\kappa \mathcal{F}(\lambda),\quad \forall\;1\leq i\leq n.
\end{equation}
\end{prop}
\begin{proof}
See the proof of \cite[Proposition 5]{gaborjdg}. Here we just sketch. We set
\begin{equation}
A_\delta:=\left\{\mathbf{v}\in\Gamma:\;f(\mathbf{v})\leq \sigma,\quad \text{and}\quad \mathbf{v}\in \mu-\delta\mathbf{1}+\overline{\Gamma_n}\right\}.
\end{equation}
Since $f_i>0,\;1\leq i\leq n,$ it follows from the argument above and \eqref{equgabor22} that $$\lim_{t\to+\infty}f(\mu-\delta\mathbf{1}+t\mathbf{e}_i)
\geq\lim_{t\to+\infty}f(\mu-2\delta\mathbf{1}+t\mathbf{e}_i)>\sigma,\quad 1\leq i\leq n,$$
and hence
$A_\delta$ is bounded.

If $\Gamma=\Gamma_n$, then $\mu-2\delta\mathbf{1}\in\Gamma_n$ and hence $A_\delta$ is compact; if $\Gamma_n\varsubsetneqq\Gamma$, then $\mu-2\delta\mathbf{1}\in\mathbb{R}^n\setminus \Gamma^\sigma$ and hence $A_\delta$ is not closed.
For each $\mathbf{v}\in\overline{A_\delta},$ we define the cone $\mathcal{C}_\mathbf{v}$ with vertex at the origin given by
\begin{equation*}
  \mathcal{C}_\mathbf{v} = \left\{ \mathbf{w}\in \mathbb{R}^n\,:\, \mathbf{v} + t\mathbf{w}\in (\mu -
  2\delta\mathbf{1} + \Gamma_n)\cap \Gamma^\sigma\,\text{ for some }
  t > 0\right\}.
  \end{equation*}
Since $\overline{A_\delta}$ is compact, the conclusion follows from applying the argument in the proof of \cite[Proposition 5]{gaborjdg} to any $\lambda\in \overline{A_\delta}$.
\end{proof}
\begin{lem}[Sz\'ekelyhidi \cite{gaborjdg}]
\label{lem9}
Let $f$ be the smooth symmetric function defined on $\Gamma$ satisfying Assumption \eqref{assum1}, \eqref{assum2} and \eqref{assum3} in the introduction. Then for any $\sigma\in(\sup_{\partial\Gamma}f,\sup_\Gamma f)$, we have
\begin{enumerate}
\item \label{gaborlem91}there exists an $N>0$ depending only on $\sigma$ such that $\Gamma+N\mathbf{1}\in \Gamma^\sigma;$
\item \label{gaborlem92}there is a $\tau>0$ depending only on $\sigma$ such that $\mathcal{F}(\lambda)>\tau,\;\forall\,\lambda\in\partial\Gamma^\sigma.$
\end{enumerate}
\end{lem}
\begin{proof}
See the proof of \cite[Lemma 9]{gaborjdg}. Here we just give some details for Part \eqref{gaborlem92}.
We choose $\alpha>1$ such that $\sup_\Gamma f>\sigma':=f(\alpha N\mathbf{1})>f( N\mathbf{1})$. Then the arguments in Part \eqref{gaborlem91} yields that for any $\mathbf{x}\in\Gamma$, there holds $\mathbf{x}+\alpha N\mathbf{1}\in\Gamma^{\sigma'}.$ In particular, we have
$$
f(\lambda+\alpha N\mathbf{1})>\sigma',\quad \forall\;\lambda\in\partial\Gamma^\sigma.
$$
This, together with the concavity of $f$, yields that, for any $\lambda\in\partial\Gamma^\sigma,$ there holds
\begin{equation*}
\sigma'<f(\lambda + \alpha N\mathbf{1}) \leq f(\lambda)+\nabla f(\lambda)\cdot \alpha N\mathbf{1}=f(\lambda) + \alpha N\sum_{i=1}^n
   f_i(\lambda)=\sigma+ \alpha N\mathcal{F}(\lambda) ,
\end{equation*}
which implies $\mathcal{F}(\lambda) \geq (\alpha N)^{-1}(\sigma' -\sigma)>0$, as required.
\end{proof}
Let us recall some basic formulae for the derivatives of eigenvalues (see for example \cite{spruck}).
\begin{lem}[Spruck \cite{spruck}]
The first and second order derivatives of the eigenvalue $\lambda_i$ at a diagonal matrix $((A_\dag)_{i}{}^j)$ $($consider it as a Hermitian matrix$)$ with distinct eigenvalue are
\begin{align}
\label{eigenvalue1st}\lambda_{i}^{pq}=&\delta_{pi}\delta_{qi},\\
\label{eigenvalue2nd}\lambda_{i}^{pq,rs}=&(1-\delta_{ip})\frac{\delta_{iq}\delta_{ir}\delta_{ps}}{\lambda_i-\lambda_p}
+(1-\delta_{ir})\frac{\delta_{is}\delta_{ip}\delta_{rq}}{\lambda_i-\lambda_r},
\end{align}
where
$$
\lambda_{i}^{pq}=\frac{\partial\lambda_i}{\partial  (A_\dag)_{p}{}^q}, \quad
\lambda_{i}^{pq,rs}=\frac{\partial^2\lambda_i}{\partial  (A_\dag)_{p}{}^q\partial  (A_\dag)_{r}{}^s}.
$$
\end{lem}
\begin{lem}[Gerhardt \cite{gerhardt}]
\label{lemgerhardt}
If $F(A_\dag)=f(\lambda_1,\cdots, \lambda_n)$ in terms of a smooth symmetric function of the eigenvalues, then at a diagonal matrix $((A_\dag)_{i}{}^j)$ $($consider it as a Hermitian matrix$)$ with distinct eigenvalues there hold
\begin{align}
\label{f1stdaoshu}F^{ij}=&\delta_{ij} f_i,\\
\label{f2nddaoshu}F^{ij,rs}=&f_{ir}\delta_{ij}\delta_{rs}+\frac{f_i-f_j}{\lambda_i-\lambda_j}(1-\delta_{ij})\delta_{is}\delta_{jr},
\end{align}
where
$$
F^{ij}=\frac{\partial F}{\partial  (A_\dag)_{i}{}^j}, \quad F^{pq,rs}=\frac{\partial^2F}{\partial  (A_\dag)_{i}{}^j\partial  (A_\dag)_{r}{}^s}.
$$
\end{lem}
These formulae make sense even if the eigenvalues are not distinct. Indeed, if $f$ is smooth and symmetric, then $f$ is a smooth function of elementary symmetric polynomials which are smooth on the space of matrices by Vieta's formulas and hence $F$ is a smooth function on the space of matrices.
In particular, we have $f_i\longrightarrow f_j$ as $\lambda_i\longrightarrow \lambda_j$. If $f$ is concave and symmetric, then we have that $\frac{f_i- f_j}{\lambda_i-\lambda_j}\leq 0$ (see \cite[Lemma 2]{eh89} or \cite{spruck}). In particular, if $\lambda_i\leq \lambda_j$, then we have $f_i\geq f_j$. Also it follows that
\begin{equation}
\label{cabor67}
F^{ij,rs}x_{i\overline{j}} x_{r\overline{s}}
= f_{ij}x_{i\overline{i}} x_{j\overline{j}}
+\sum\limits_{p\neq q}\frac{f_p- f_q}{\lambda_p-\lambda_q}|x_{p\overline{q}}|^2
\leq f_{ij} x_{i\overline{i}}x_{j\overline{j}}
+\sum\limits_{p>1}\frac{f_1- f_p}{\lambda_1-\lambda_p}|x_{p\overline{1}}|^2.
\end{equation}
For any function $u\in C_{\mathrm{B}}^\infty(M,\mathbb{R})$, we define an elliptic operator $L$ (linearized operator of $F$ given in \eqref{equ1}) by
\begin{equation}
\label{defnl}
L(u)=F^{ij}(g_\dag)^{\overline{q}j}\partial_i\partial_{\overline{q}}u,
\end{equation}
in the foliated local coordinate patch $\left((-\varepsilon_0,\varepsilon_0)\times B_2(\mathbf{0});x,\mathbf{z}=(z^1,\cdots,z^n)\right)$.
\begin{proof}
[Proof of Theorem \ref{thm2order}]
We can assume that the transverse subsolution $\underline{u}=0$ since otherwise we can  modify the background basic form $\beta_\dag$. This means that for any $\mathbf{p}\in M$ the sets $(\lambda(A_{\dag,0}(\mathbf{p}))+\Gamma_n)\cap \partial \Gamma^{\psi(\mathbf{p})}$ are bounded, where we recall that $A_{\dag,0}$ is the  transverse Hermitian endomorphism determined by $\beta_{\dag,0}.$
Since $M$ is compact, there exist uniform $\delta,\,R>0$ such that at each point $\mathbf{p}$ there holds
\begin{equation}
(\lambda(A_{\dag,0})-2\delta\mathbf{1}+\Gamma_n)\cap \partial\Gamma^{\psi(\mathbf{p})}\subset B_R(\mathbf{0}).
\end{equation}
By the argument in the proof \cite[Proposition 6]{gaborjdg} (based on Proposition \ref{gaborprop5}), we can find a uniform constant $\kappa>0$ such that at any point $\mathbf{p}\in M$, if $|\lambda(A_{\dag,u})|>R$ and $A_{\dag,u}=\mathrm{diag}\{\lambda_1,\cdots,\lambda_n\}$ with $\lambda_1\geq \cdots\geq\lambda_n$, then there holds either
\begin{equation}
\label{3.4}
  F^{ii}(A_{\dag,u})>\kappa\sum_qF^{qq}(A_{\dag,u})\quad\text{for all}\;i,
\end{equation}
or
\begin{equation}
\label{3.5}
  \sum_q F^{qq}(A_{\dag,u})\left((A_{\dag,0})_q{}^q-\lambda_q\right)>\kappa\sum_qF^{qq}(A_{\dag,u}).
\end{equation}
From Lemma \ref{lem9}, it follows that there exists a uniform constant $\tau>0$ such that
\begin{equation}
\label{fxiajie}
 \sum_qF^{qq}(A_{\dag,u}(\mathbf{p}))>\tau,\quad\forall\;\mathbf{p}\in M.
\end{equation}
Since $\mathrm{tr}_{\omega_\dag}h_\dag=\sum_{i=1}^n\lambda_i>0$, it is sufficient to get the upper bound of $\lambda_1$ to complete the proof.

Our background is a counterpart to the K\"ahler case in \cite{gaborjdg}, and we use the method revised from \cite{stw1503} (see also \cite{cw01,ctwjems,houmawu,gaborjdg}), where another auxiliary function $\varphi(t)=D_1e^{-D_2t}$ is introduced. Since there are no terms about the first order derivatives of $u$ in $(h_\dag)_{i\overline{j}},$ and the adapted Chern connection is torsion free, we need not estimate as many terms as done in \cite{stw1503} (see also \cite{gaborjdg}).
Actually, we consider the function
\begin{equation*}
H=\log\lambda_1+\rho\left(|\nabla u|^2_{\omega_\dag}\right)+\varphi(u),
\end{equation*}
where
\begin{equation*}
\rho(t)=-\frac{1}{2}\log\left(1-\frac{t}{2K}\right),\quad\varphi(t)=D_1e^{-D_2t},
\end{equation*}
with sufficiently large uniform constants $D_1,D_2>0$ to be determined later.
A direct calculation yields that
\begin{equation*}
\rho\left(|\nabla u|^2_{\omega_\dag}\right)\in[0,\,2\log 2]
\end{equation*}
and
\begin{equation*}
\frac{1}{4K}<\rho'<\frac{1}{2K},\quad \rho''=2(\rho')^2.
\end{equation*}
We work at a point $\mathbf{x}_0$ where $H$ achieves its maximum.  After translating and orthonormal transforms, we can assume that $\mathbf{x}_0$ is in a foliated local coordinate patch such that $\mathbf{x}_0$ is the origin, $\omega_\dag$ is a unit matrix and $h_\dag$ is diagonal with $\lambda_1=(h_\dag)_{1\overline{1}}\geq\cdots\geq (h_\dag)_{n\overline{n}}$. To make sure that $H$ is smooth at this point, we fix a diagonal matrix $B$ with $B_{1}{}^1=0,\; 0<B_{2}{}^2<\cdots<B_{n}{}^n$, and define
$\tilde A=A_{\dag,u}-B$ with eigenvalues denoted by $\tilde\lambda_1,\cdots,\tilde\lambda_{n}$. Clearly, at this point  $\mathbf{x}_0$, there hold
\begin{equation*}
\tilde\lambda_1=\lambda_1,\quad \tilde\lambda_i=\lambda_i-B_{i}{}^{i},\quad i=2,\cdots,n
\end{equation*}
and
$
\tilde\lambda_1>\cdots>\tilde\lambda_n.
$
Since we assume
$
\sum\limits_{i=1}^n\lambda_i >0,
$
we can choose  $B$  small enough such that
$$
\sum\limits_{i=1}^n\tilde\lambda_i>-1
$$
and
\begin{equation}\label{6.18}
  \sum\limits_{p>1}\frac{1}{\lambda_1-\tilde\lambda_p}\leq C
\end{equation}
where $C$ is a fixed constant depending only on $n$. We give some remarks about $B$. It can also be considered as a section in $\Gamma_{\mathrm{B}}\left(\nu(\mathcal{F}_\xi)\otimes \left(\lambda_{\mathrm{B}}\otimes_{\mathbb{R}}\mathbb{C}\right)\right)$ which is represented by a constant diagonal matrix $(B_{i}{}^j)$ in these foliated local coordinates. Hence, we get
\begin{align*}
\nabla^{\dag}_{\overline{j}}B_{r}{}^s=&\partial_{\overline{j}}B_{r}{}^s=0,\quad \nabla^{\dag}_{i}\nabla^{\dag}_{\overline{j}}B_{r}{}^s=0,\\ \nabla^{\dag}_{i}B_{r}{}^s=&\partial_{i}B_{r}{}^s-\Gamma_{ir}^pB_{p}{}^s+\Gamma_{ip}^sB_{r}{}^p=\Gamma_{ir}^s\left(B_{r}{}^r-B_{s}{}^s\right),\\
\nabla^{\dag}_{\overline{j}}\nabla^{\dag}_{i}B_{r}{}^s=&\partial_{\overline{j}}\nabla^{\dag}_{i}B_{r}{}^s=R_{i\overline{j}r}{}^s\left(B_{s}{}^s-B_{r}{}^r\right).
\end{align*}
Here and hereafter, let $\nabla^\dag$ denote the adapted Chern connection of $\omega_\dag$, and we still use the notations $\nabla^\dag_{i}:=\nabla^\dag_{\partial_i},\,\nabla^\dag_{\bar j}:=\nabla^\dag_{\partial_{\bar j}}.$

Now consider the quantity
\begin{equation*}
\tilde H=\log\tilde\lambda_1+\rho\left(|\nabla u|^2_{\omega_\dag}\right)+\varphi(u),
\end{equation*}
which is smooth in this foliated chart and attains its maximum at the point $\mathbf{x}_0$. All the following calculation is at this point.
We may assume $\lambda_1\gg K>1$. We use subscripts $k$ and $\overline{\ell}$ to denote the partial derivatives $\partial/\partial z_k$ and $\partial/\partial \overline{z}_{\ell}$. At the point $\mathbf{x}_0$, we have
\begin{align}
\label{hq}\tilde H_q=&\frac{\tilde{\lambda}_{1,q}}{\lambda_1} + \rho'V_q  +
\varphi' u_q=0, \quad \textrm{for } V_q : = u_ru_{\overline{r}q} +u_{\overline{r}}\nabla^{\dag}_qu_{r}.\\
\label{hqq}\tilde{H}_{q\overline{q}}=& \frac{\tilde{\lambda}_{1,q\overline{q}}}{\lambda_1} -
\frac{|\tilde{\lambda}_{1,q}|^2}{\lambda_1^2}
+ \rho'\Big( u_r\nabla^{\dag}_{\overline{q}}u_{\overline{r}q}+u_{\overline{r}}\nabla^{\dag}_{\overline{q}}\nabla^{\dag}_{q}u_r  +
|\nabla^{\dag}_{q}u_{r}|^2 + |u_{\overline{r}q}|^2 \Big)  \\
& + \rho''|V_q|^2 + \varphi''|u_q|^2 + \varphi' u_{q\overline{q}}. \nonumber
\end{align}
From \eqref{eigenvalue1st}, we get
\begin{align}\label{lambda1p}
\tilde\lambda_{1,p}
= \tilde\lambda_{1}^{rs}\left((g_\dag)^{\overline{j}s}\nabla^{\dag}_{p}(h_\dag)_{r\overline{j}}-\nabla^{\dag}_pB_{r}{}^s\right)
=\nabla^{\dag}_{p}(h_\dag)_{1\overline{1}}-\nabla^{\dag}_pB_{1}{}^1=\nabla^{\dag}_{p}(h_\dag)_{1\overline{1}},
\end{align}
where we use the fact that $\nabla^{\dag}_pB_{1}{}^1=0$. Then using this formula and \eqref{eigenvalue1st} and \eqref{eigenvalue2nd}, we can deduce
\begin{align*}
\tilde\lambda_{1,p\overline{q}}
=&\tilde\lambda_{1}^{rs,ab}\left((g_\dag)^{\overline{j}s} \nabla^{\dag}_{p}(h_\dag)_{r\overline{j}}-\nabla^{\dag}_{p}B_{r}{}^s\right) \left((g_\dag)^{\overline{\ell}b} \nabla^{\dag}_{\overline{q}}(h_\dag)_{a\overline{\ell}}-\nabla^{\dag}_{\overline{q}}B_{a}{}^b\right)\\
&+\tilde\lambda_{1}^{rs}\left((g_\dag)^{\overline{j}s}\nabla^{\dag}_{\overline{q}}\nabla^{\dag}_{p}(h_\dag)_{r\overline{j}}
-\nabla^{\dag}_{\overline{q}}\nabla^{\dag}_{p}B_{r}{}^s \right)  \\
=&\tilde\lambda_{1}^{rs,ab}\left(\nabla^{\dag}_{p}(h_\dag)_{r\overline{s}}-\nabla^{\dag}_{p}B_{r}{}^s\right) \left(\nabla^{\dag}_{\overline{q}}(h_\dag)_{a\overline{b}}\right)
+\tilde\lambda_{1}^{rs}\left(\nabla^{\dag}_{\overline{q}}\nabla^{\dag}_{p}(h_\dag)_{r\overline{s}}
-\nabla^{\dag}_{\overline{q}}\nabla^{\dag}_{p}B_{r}{}^s \right)\nonumber  \\
=&\sum\limits_{r\neq 1}\frac{1}{\lambda_1-\tilde\lambda_{r}}\left(\nabla^{\dag}_{p}(h_\dag)_{r\overline{1}} -\nabla^{\dag}_{p}B_{r}{}^1\right)\left(\nabla^{\dag}_{\overline{q}}(h_\dag)_{1\overline{r}}\right)\nonumber\\
&+\sum\limits_{a\neq 1}\frac{1}{\lambda_1-\tilde\lambda_{a}}\left(\nabla^{\dag}_{p}(h_\dag)_{1\overline{a}}
-\nabla^{\dag}_{p}B_{1}{}^a\right)\left(\nabla^{\dag}_{\overline{q}}(h_\dag)_{a\overline{1}}\right)
 +\nabla^{\dag}_{\overline{q}}\nabla^{\dag}_{p}(h_\dag)_{1\overline{1}},      \nonumber
\end{align*}
where we also use the fact that $\nabla^{\dag}_{\overline{q}}B_{a}{}^b=\nabla^{\dag}_{\overline{q}}\nabla^{\dag}_{p}B_{1}{}^1=0$. In particular, we have
\begin{align}\label{lambda1qq}
\tilde\lambda_{1,q\overline{q}}=&\nabla^{\dag}_{\overline{q}}\nabla^{\dag}_{q}(h_\dag)_{1\overline{1}}+ \sum\limits_{r> 1}\frac{|\nabla^{\dag}_{q}(h_\dag)_{r\overline{1}} |^2+|\nabla^{\dag}_{q}(h_\dag)_{1\overline{r}}|^2}{\lambda_1-\tilde\lambda_{r}}\\
 &-\sum\limits_{r> 1}\frac{\left(\nabla^{\dag}_{q}B_{r}{}^1\right)\left(\nabla^{\dag}_{\overline{q}}(h_\dag)_{1\overline{r}}\right)
 +\left(\nabla^{\dag}_{q}B_{1}{}^r\right)\left(\nabla^{\dag}_{\overline{q}}(h_\dag)_{r\overline{1}}\right)}{\lambda_1-\tilde\lambda_{r}}.\nonumber
\end{align}
Recall that our choice of $B$ yields that $\sum_i\tilde \lambda_i>-1$, which means that $$(\lambda_1-\tilde\lambda_r)^{-1}\geq (n\lambda_1+1)^{-1}$$ for $r>1$. This, together with \eqref{lambda1qq} and Young's inequality, shows that
\begin{align}
\label{lambda1qqest}
\tilde\lambda_{1,q\overline{q}}\geq&\nabla^{\dag}_{\overline{q}}\nabla^{\dag}_{q}(h_\dag)_{1\overline{1}}+\frac{1}{2(n\lambda_1+1)} \sum\limits_{r> 1}\left(|\nabla^{\dag}_{q}(h_\dag)_{r\overline{1}} |^2+|\nabla^{\dag}_{q}(h_\dag)_{1\overline{r}}|^2\right)-C\\
\geq&\nabla^{\dag}_{\overline{q}}\nabla^{\dag}_{q}(h_\dag)_{1\overline{1}}+\frac{1}{4n\lambda_1} \sum\limits_{r> 1}\left(|\nabla^{\dag}_{q}(h_\dag)_{r\overline{1}} |^2+|\nabla^{\dag}_{q}(h_\dag)_{1\overline{r}}|^2\right)-C,\nonumber
\end{align}
where we use the assumption that $\lambda_1\gg  1$.

From \eqref{ricciidentity}, it follows that
\begin{align}\label{qqg11-11gqq}
&\nabla^{\dag}_{\overline{q}}\nabla^{\dag}_q(h_\dag)_{1\overline{1}}-\nabla^{\dag}_{\overline{1}}\nabla^{\dag}_1(h_\dag)_{q\overline{q}}\\
=&\nabla^{\dag}_{\overline{q}}\nabla^{\dag}_q(\beta_\dag)_{1\overline{1}}-\nabla^{\dag}_{\overline{1}}\nabla^{\dag}_1(\beta_\dag)_{q\overline{q}}
 +\nabla^{\dag}_{\overline{q}}\nabla^{\dag}_qu_{1\overline{1}}-\nabla^{\dag}_{\overline{1}}\nabla^{\dag}_1u_{q\overline{q}}\nonumber\\
=&\nabla^{\dag}_{\overline{q}}\nabla^{\dag}_q(\beta_\dag)_{1\overline{1}}-\nabla^{\dag}_{\overline{1}}\nabla^{\dag}_1(\beta_\dag)_{q\overline{q}}
  +R_{q\overline{q}1}{}^pu_{p\overline{1}}
 -R_{1\overline{1}q}{}^pu_{p\overline{q}}=O(\lambda_1).\nonumber
\end{align}
Thanks to \eqref{lambda1qqest} and \eqref{qqg11-11gqq}, we deduce that
\begin{equation}
\label{fqqlambda1qq}
F^{qq}\tilde\lambda_{1,q\overline{q}}
\geq F^{qq}\nabla^{\dag}_{\overline{1}}\nabla^{\dag}_{1}(h_\dag)_{q\overline{q}}+\frac{1}{4n\lambda_1} \sum\limits_{r> 1}F^{qq}\left(|\nabla^{\dag}_{q}(h_\dag)_{r\overline{1}} |^2+|\nabla^{\dag}_{q}(h_\dag)_{1\overline{r}}|^2\right)-C\lambda_1\mathcal{F},
\end{equation}
where $\mathcal{F}=\sum_q F^{qq}\geq \tau>0$ by \eqref{fxiajie} and at the point $x_0$, from \eqref{equ2} and \eqref{eigenvalue2nd},  we have
$$
F^{ij}=\left\{\begin{array}{cc}
                0, & \text{if}\quad i\not=j, \\
                f_i, & \text{if}\quad  i=j.
              \end{array}
\right.
$$
At the point $x_0$, applying $\nabla^{\dag}_{ i}$  to   \eqref{equ1}, we get
\begin{equation}
\label{psii}
\psi_i
= F^{pq}\nabla_i(h_\dag)_{p\overline{t}}(g_\dag)^{\overline{t}q}
=F^{jj}\nabla_i\left((\beta_\dag)_{j\overline{j}}+u_{j\bar j}\right) .
\end{equation}
Applying $\nabla^{\dag}_{\bar i}$ to \eqref{psii}, we deduce that
\begin{align}
\label{psiii}
\psi_{i\bar i}=&F^{pq,rs}\left(\nabla^{\dag}_{i}(h_\dag)_{p\overline{t}}(g_\dag)^{\overline{t}q}\right)
\left(\nabla^{\dag}_{\bar i}(h_\dag)_{r\overline{\ell}}(g_\dag)^{\overline{\ell}s}\right) \\
&+F^{pq}\nabla^{\dag}_{\bar i}\nabla^{\dag}_i(h_\dag)_{p\overline{t}}(g_\dag)^{\overline{t}q}\nonumber\\
=&F^{pq,rs}\left(\nabla^{\dag}_{i}(h_\dag)_{p\overline{q}} \right)
\left(\nabla^{\dag}_{\bar i}(h_\dag)_{r\overline{s}} \right)
 +F^{pq}\nabla^{\dag}_{\bar i}\nabla^{\dag}_i(h_\dag)_{p\overline{q}}.\nonumber
\end{align}
From \eqref{fqqlambda1qq} and \eqref{psiii}, it follows that
\begin{align}
\label{fqqlambda1qq2}
F^{qq}\tilde\lambda_{1,q\overline{q}}
\geq& \frac{1}{4n\lambda_1} \sum\limits_{r> 1}F^{qq}\left(|\nabla^{\dag}_{q}(h_\dag)_{r\overline{1}} |^2+|\nabla^{\dag}_{q}(h_\dag)_{1\overline{r}}|^2\right)\\
&-F^{pq,rs}\left(\nabla^{\dag}_{1}(h_\dag)_{p\overline{q}} \right)
\left(\nabla^{\dag}_{\bar 1}(h_\dag)_{r\overline{s}} \right)-C\lambda_1\mathcal{F},\nonumber
\end{align}
where we absorb the term of $\psi_{1\bar 1}$ into $-C\lambda_1\mathcal{F}$ since $\mathcal{F}\geq\tau>0$.

Combining \eqref{hqq} and \eqref{fqqlambda1qq2},  we get
\begin{align}
\label{fqqhqqdoth2}
0\geq& L(\tilde H)=F^{qq}\tilde H_{q\bar q}\\
\geq&-\frac{1}{\lambda_1}F^{pq,rs}\left(\nabla^{\dag}_{1}(h_\dag)_{p\overline{q}} \right)
\left(\nabla^{\dag}_{\bar 1}(h_\dag)_{r\overline{s}} \right)-C \mathcal{F}\nonumber\\
&-\frac{F^{qq}|\tilde{\lambda}_{1,q}|^2}{\lambda_1^2}
+ \rho'\Big( F^{qq}u_r\nabla^{\dag}_{\overline{q}}u_{\overline{r}q}
+F^{qq}u_{\overline{r}}\nabla^{\dag}_{\overline{q}}\nabla^{\dag}_{q}u_r
 \Big)\nonumber  \\
 &+\sum_r\frac{F^{qq}}{4K}\Big(|\nabla^{\dag}_{q}u_{r}|^2
+ |u_{\overline{r}q}|^2\Big)\nonumber\\
& + \rho''F^{qq}|V_q|^2 + \varphi''F^{qq}|u_q|^2 + \varphi' F^{qq}u_{q\overline{q}},\nonumber
\end{align}
since $\rho'\geq 1/(4K)$.

Using the Ricci identity \eqref{ricciidentity}, \eqref{psii} and the fact that $\rho'\leq 1/(2K)$, we can get
\begin{align}\label{fenleiqian}
\rho'F^{qq}u_r\nabla^{\dag}_{\overline{q}}u_{\overline{r}q}
=&\rho'F^{qq}u_r \nabla^{\dag}_{\overline{r}}u_{q\overline{q}} \\
=&\rho'F^{qq}u_r\left(\nabla^{\dag}_{\overline{r}}(h_\dag)_{q\overline{q}}
-\nabla^{\dag}_{\overline{r}}(\beta_\dag)_{q\overline{q}}
\right)\geq -\frac{C\mathcal{F}}{K^{1/2}},\nonumber\\
\label{fenleiqian1}
\rho'F^{qq}u_{\overline{r}} \nabla^{\dag}_{\overline{q}}\nabla^{\dag}_{q}u_{r}
=&\rho'F^{qq}u_{\overline{r}}\nabla^{\dag}_{\overline{q}} \nabla^{\dag}_{r}u_{q}   \\
=&\rho'F^{qq}u_{\overline{r}}\left(\nabla^{\dag}_{r}u_{q\overline{q}}-R_{r\overline{q}q}{}^pu_p
 \right)\nonumber\\
=&\rho'F^{qq}u_{\overline{r}}\left(\nabla^{\dag}_{r}(h_\dag)_{q\overline{q}}-\nabla^{\dag}_{r}(\beta_\dag)_{q\overline{q}}
-R_{r\overline{q}q}{}^pu_p\right) \geq -C\mathcal{F}.\nonumber
\end{align}
From \eqref{fqqhqqdoth2}, \eqref{fenleiqian}, \eqref{fenleiqian1} and the fact that $\rho'\geq 1/(4K)$, it follows that
\begin{align}
\label{fqqhqq3}
0\geq &-\frac{1}{\lambda_1}F^{pq,rs}\left(\nabla^{\dag}_{1}(h_\dag)_{p\overline{q}} \right)
\left(\nabla^{\dag}_{\bar 1}(h_\dag)_{r\overline{s}} \right)-C \mathcal{F}\\
&-\frac{F^{qq}|\tilde{\lambda}_{1,q}|^2}{\lambda_1^2}
+ \sum_r\frac{F^{qq}}{4K}\left(
 |\nabla^{\dag}_{q}u_{r}|^2
+  |u_{\overline{r}q}|^2 \right)\nonumber  \\
& + \rho''F^{qq}|V_q|^2 + \varphi''F^{qq}|u_q|^2 + \varphi' F^{qq}u_{q\overline{q}}, \nonumber
\end{align}
where we use the fact that $K>1$ and hence $K^{-1/2}<1$ and absorb the constant into $C\mathcal{F}$.

Next, we deal with two cases separately.

\textbf {Case 1:} $\delta\lambda_1\geq -\lambda_n$, where the constant $\delta$ will be determined later.
We set
$$
I:=\Big\{i:\;F^{ii}>\delta^{-1}F^{11}\Big\}.
$$
From \eqref{hq}, the Cauchy-Schwarz inequality yields that
\begin{align}
\label{case1qnotini}
-\sum_{q\not\in I}\frac{F^{qq} |\tilde\lambda_{1,q} |^2}{\lambda_1^2}
=&-\sum_{q\not\in I}F^{qq}\left|\rho'V_q+\varphi'u_q\right|^2\\
\geq&-2\rho'^2\sum_{q\not\in I}F^{qq}\left| V_q\right|^2-2\varphi'^2\sum_{q\not\in I}F^{qq}\left|\varphi'u_q\right|^2\nonumber\\
\geq&-\rho''\sum_{q\not\in I}F^{qq}\left| V_q\right|^2-2\varphi'^2\delta^{-1}F^{11}K.\nonumber
\end{align}
Similarly, for $q\in I$ we also have
\begin{equation}
\label{case1qini}
-2\delta\sum_{q\in I}\frac{F^{qq} |\tilde\lambda_{1,q} |^2}{\lambda_1^2}
\geq -2\delta\rho''\sum_{q\in I}F^{qq}|V_q|^2-4\delta\varphi'^2\sum_{q\in I}F^{qq}|u_k|^2.
\end{equation}
Since $\varphi''>0$, we can choose $\delta$ sufficiently small (i.e., depending on $D_1,\,D_2$ and $\sup_M|u|$) such that
\begin{equation}
\label{deltachoice}
4\delta\varphi'^2<\frac{1}{2}\varphi''.
\end{equation}
Substituting \eqref{case1qnotini}, \eqref{case1qini} and \eqref{deltachoice}
into \eqref{fqqhqq3} shows that
\begin{align}
\label{mainequ1}
0\geq &-\frac{1}{\lambda_1}F^{pq,rs}\left(\nabla^{\dag}_{1}(h_\dag)_{p\overline{q}} \right)
\left(\nabla^{\dag}_{\bar 1}(h_\dag)_{r\overline{s}} \right)-C \mathcal{F}\\
&-(1-2\delta)\sum_{q\in I}\frac{F^{qq}|\tilde{\lambda}_{1,q}|^2}{\lambda_1^2}
+ \sum_r\frac{F^{qq}}{4K}\left(
 |\nabla^{\dag}_{q}u_{r}|^2
+  |u_{\overline{r}q}|^2 \right)\nonumber  \\
&-2\varphi'^2\delta^{-1}F^{11}K + \frac{1}{2}\varphi''F^{qq}|u_q|^2 + \varphi' F^{qq}u_{q\overline{q}}, \nonumber
\end{align}
The assumption that $\delta\lambda_1\geq -\lambda_n$ implies that
\begin{equation}
\label{deltaprop}
  \frac{1-\delta}{\lambda_1-\lambda_k}\geq \frac{1-2\delta}{\lambda_1}.
\end{equation}
Substituting \eqref{deltaprop} into \eqref{cabor67}  with $k=1$ means that
\begin{align}
\label{fpqrsest}
-\frac{1}{\lambda_1}F^{pq,rs}\left(\nabla^{\dag}_{1}(h_\dag)_{p\overline{q}} \right)
\left(\nabla^{\dag}_{\bar 1}(h_\dag)_{r\overline{s}} \right)\geq&
 \sum_{q\in I} \frac{F^{qq}- F^{11}}{\lambda_1 - \lambda_q}
|\nabla^{\dag}_1(h_\dag)_{q\bar 1}|^2\\
\geq&\frac{1-2\delta}{\lambda_1}\sum_{q\in I}F^{qq}|\nabla^{\dag}_1(h_\dag)_{q\bar 1}|^2.\nonumber
\end{align}
From \eqref{mainequ1} and \eqref{fpqrsest}, it follows that
\begin{align}
\label{mainequ2}
0\geq &\frac{1-2\delta}{\lambda_1^2}\sum_{q\in I}F^{qq}\left(|\nabla^{\dag}_1(h_\dag)_{q\bar 1}|^2-|\tilde{\lambda}_{1,q}|^2\right)-C \mathcal{F}\\
&
+ \sum_r\frac{F^{qq}}{4K}\left(
 |\nabla^{\dag}_{q}u_{r}|^2
+  |u_{\overline{r}q}|^2 \right)\nonumber  \\
&-2\varphi'^2\delta^{-1}F^{11}K + \frac{1}{2}\varphi''F^{qq}|u_q|^2 + \varphi' F^{qq}u_{q\overline{q}}, \nonumber
\end{align}
From \eqref{ricciidentity} and \eqref{lambda1p}, a direct calculation gives
\begin{align*}
\nabla^{\dag}_1(h_\dag)_{q\bar 1}
=&\nabla^{\dag}_1(\beta_\dag)_{q\bar 1}+\nabla^{\dag}_1u_{q\bar 1}\\
=&\nabla^{\dag}_1(\beta_\dag)_{q\bar 1}+\nabla^{\dag}_qu_{1\bar 1}\nonumber\\
=&\nabla^{\dag}_1(\beta_\dag)_{q\bar 1}-\nabla^{\dag}_q(\beta_\dag)_{1\bar 1}+\nabla^{\dag}_q(h_\dag)_{1\bar 1}\nonumber\\
=&\nabla^{\dag}_1(\beta_\dag)_{q\bar 1}-\nabla^{\dag}_q(\beta_\dag)_{1\bar 1}+\tilde\lambda_{1,q}
=\tilde\lambda_{1,q}+O(1),\nonumber
\end{align*}
which yields that
\begin{equation}
\label{hqbra1difftildelambda1q}
|\nabla^{\dag}_1(h_\dag)_{q\bar 1}|^2-|\tilde{\lambda}_{1,q}|^2\geq -C(1+|\tilde{\lambda}_{1,q}|).
\end{equation}
By \eqref{hq},  for any $\varepsilon>0$, there exists a constant $C_\varepsilon$ such that
\begin{align}
\label{tildelambda1q}
|\tilde{\lambda}_{1,q}|=&\lambda_1\left|\rho'\left(u_ru_{\overline{r}q}
+u_{\overline{r}}\nabla^{\dag}_qu_{r}\right)  +
\varphi' u_q\right|\\
\leq&\frac{C\lambda_1}{K^{1/2}}\left(\sum_r|u_{\overline{r}q}| +\sum_r|\nabla^{\dag}_qu_{r}|\right)
+\lambda_1|\varphi'||u_q|\nonumber\\
\leq&\frac{C\lambda_1}{K^{1/2}}\left(\sum_r|u_{\overline{r}q}| +\sum_r|\nabla^{\dag}_qu_{r}|\right)
-\varepsilon\lambda_1 \varphi'-C_\varepsilon \lambda_1\varphi'|u_q|^2,\nonumber
\end{align}
where we also use the assumption that $\lambda\gg1.$

Substituting \eqref{hqbra1difftildelambda1q} and \eqref{tildelambda1q} into \eqref{mainequ2}, Young's inequality shows
\begin{align}
\label{mainequ3}
0\geq &-C\mathcal{F}
+ \sum_r\frac{F^{qq}}{8K}\left(
 |\nabla^{\dag}_{q}u_{r}|^2
+  |u_{\overline{r}q}|^2 \right)   \\
&-2\varphi'^2\delta^{-1}F^{11}K + \frac{1}{2}\varphi''F^{qq}|u_q|^2 + \varphi' F^{qq}u_{q\overline{q}}
+C\varepsilon \varphi'\mathcal{F}+C_\varepsilon \varphi'F^{qq}|u_q|^2  \nonumber\\
\geq&F^{11}\left(\frac{\lambda_1^2}{16K}-2\varphi'^2\delta^{-1}K\right)
+\left(\frac{1}{2}\varphi''+C_\varepsilon \varphi'\right)F^{qq}|u_q|^2\nonumber\\
&-C\mathcal{F}+C\varepsilon \varphi'\mathcal{F}
-\varphi'F^{qq}\left((\beta_\dag)_{q\overline{q}}-(h_\dag)_{q\overline{q}}\right),\nonumber
\end{align}
where we use the fact that $|u_{1\bar 1}|^2\geq \lambda_1^2/2-C$ and the assumption that $\lambda_1\gg1$.

By \eqref{3.4}, \eqref{3.5} and \eqref{fxiajie}, one of the two possibilities happens.
\begin{enumerate}
\item[(a)] The inequality \eqref{3.5} holds, i.e., we have $F^{qq}\left((\beta_\dag)_{q\overline{q}}-(h_\dag)_{q\overline{q}}\right)>\kappa \mathcal{F}$.
   We apply this to  \eqref{mainequ3} and get
   \begin{align}
   \label{mainequ4}
     0\geq & F^{11}\left(\frac{\lambda_1^2}{16K}-2\varphi'^2\delta^{-1}K\right)
+\left(\frac{1}{2}\varphi''+C_\varepsilon \varphi'\right)F^{qq}|u_q|^2 \\
&-C\mathcal{F}+C\varepsilon \varphi'\mathcal{F}
-\kappa\varphi'\mathcal{F}.\nonumber
   \end{align}
   We first choose $\varepsilon>0$ such that $\varepsilon<\kappa/2$. Then we choose $D_2$ in the definition of $\varphi(t)=D_1e^{-D_2t}$ sufficiently large such that
   \begin{equation*}
     \frac{1}{2}\varphi''>C_\varepsilon|\varphi'|.
   \end{equation*}
 This, together with \eqref{mainequ4}, shows that
   \begin{equation*}
     0\geq   F^{11}\left(\frac{\lambda_1^2}{16K}-2\varphi'^2\delta^{-1}K\right)
 -C\mathcal{F}-\frac{1}{2}\kappa\varphi'\mathcal{F}.
   \end{equation*}
   Now we choose $D_1$ so large that $-\frac{1}{2}\kappa\varphi'>C$, which yields that
    \begin{equation*}
     0\geq    \frac{\lambda_1^2}{16K}-2\varphi'^2\delta^{-1}K .
   \end{equation*}
   Recall that $\delta$ is determined by the choice of $D_1,\,D_2$ from \eqref{deltachoice}. Hence we get the upper bound of $\lambda_1/K$.
\item[(b)]   The inequality \eqref{3.4} holds, which yields that $F^{11}\geq \kappa\mathcal{F}$. With the choice of constants chosen above, from \eqref{mainequ3}, it follows that
   \begin{equation*}
0\geq  \kappa\mathcal{F}\left(\frac{\lambda_1^2}{16K}-2\varphi'^2\delta^{-1}K\right)
 -C\mathcal{F}+C\varepsilon \varphi'\mathcal{F}
-\varphi'F^{qq}(h_\dag)_{q\overline{q}}.
\end{equation*}
 This, together with $F^{qq}(h_\dag)_{q\overline{q}}\leq \lambda_1\mathcal{F}$, shows that
 \begin{equation*}
0\geq \frac{ \kappa\lambda_1^2}{16K^2}-C\left(1+K^{-1}+\lambda_1K^{-1}\right),
\end{equation*}
from which we get the upper bound of $\lambda_1/K$ required.
\end{enumerate}
\textbf {Case 2:} $\delta\lambda_1\leq -\lambda_n$, where the constants $\delta,\,D_1$ and $D_2$ are fixed as in the previous case.
Note that $F^{nn}\geq \frac{1}{n}\mathcal{F}$ and $\lambda_n^2\geq \delta^2\lambda_1^2$. This yields that
\begin{align}
\label{case2fqq}
\sum_r\frac{F^{qq}}{4K}\left(
 |\nabla^{\dag}_{q}u_{r}|^2
+  |u_{\overline{r}q}|^2 \right)
\geq&\frac{F^{nn}}{4K}|u_{n\bar n}|^2\geq   \frac{F^{nn}}{4K}\left|\lambda_n-(\beta_\dag)_{n\bar n}\right|^2\\
\geq& \frac{\delta^2}{8nK}\mathcal{F}|\lambda_1|^2-C\mathcal{F}.\nonumber
\end{align}
Similar to \eqref{case1qnotini}, we get
\begin{align}
\label{case2qnotini}
- \frac{F^{qq} |\tilde\lambda_{1,q} |^2}{\lambda_1^2}
=&- F^{qq}\left|\rho'V_q+\varphi'u_q\right|^2\\
\geq&-2\rho'^2 F^{qq}\left| V_q\right|^2-2\varphi'^2 F^{qq}\left|\varphi'u_q\right|^2\nonumber\\
\geq&-\rho'' F^{qq}\left| V_q\right|^2-CK\mathcal{F}.\nonumber
\end{align}
Substituting \eqref{case2fqq} and \eqref{case2qnotini} into \eqref{fqqhqq3}, we get
\begin{equation}
\label{case2fqqhqq1}
0\geq \frac{\delta^2}{8nK}\mathcal{F}|\lambda_1|^2  -C\lambda_1 \mathcal{F},
\end{equation}
where we use the  fact that $\varphi' F^{qq}u_{q\overline{q}}=O(\lambda_1\mathcal{F})$ and the assumption that $\lambda_1\gg K$.
This also yields the required upper bound of $\lambda_1/K$.
\end{proof}
\section{First Order Estimate}\label{sec1st}
In this section, we prove the first order estimate of the solution to \eqref{equ1} and complete the proof of Theorem \ref{mainthm}.
\begin{thm}
\label{thm1order}
Suppose that the function $u\in C_{\mathrm{B}}^\infty(M,\mathbb{R})$ is a solution to \eqref{equ1} with $\sup_Mu=0,$ and that the function $\underline{u}\in C_{\mathrm{B}}^\infty(M,\mathbb{R})$ is a transverse $\mathcal{C}$-subsolution to \eqref{equ1}. There exists a uniform constant $C$ depending only on $\eta,\,\omega_\dag$ and $\underline{u}$ such that
\begin{equation}
\label{1stest}
\sup_M|\nabla u|_{\omega_\dag}\leq C.
\end{equation}
\end{thm}
Let us recall some concepts about $\Gamma$-solution from \cite{gaborjdg}.
\begin{defn}
\label{gammasol}
Suppose that $u:\,\mathbb{C}^n\to\mathbb{R}$ is a continuous function. We call $u$ is a $($viscosity$)$ $\Gamma$-subsolution if for all $h\in C^2(\mathbb{C}^n,\mathbb{R})$ such that $u-h$ has local maximum at $\mathbf{z},$ then there holds $\lambda(h_{i\bar j}(\mathbf{z}))\in \overline{\Gamma},$ where $\lambda(A)$ is the eigenvalues of the Hermitian metric $A$.

We call that a  $\Gamma$-subsolution $u$ is a $\Gamma$-solution if for all $\mathbf{z}\in \mathbb{C}^n$, $h\in C^2(\mathbb{C}^n,\mathbb{R})$ such that $u-h$ has local minimum at $\mathbf{z},$ then $\lambda(h_{i\bar j}(\mathbf{z}))\in\mathbb{R}^n\setminus\Gamma$.
\end{defn}
Note that $\Gamma$-subsolution is subharmonic since $\Gamma\subset\left\{\mathbf{x}=(x_1,\dots,x_n)\in\mathbb{R}^n:\;\sum_{i=1}^nx_i>0\right\}$.
\begin{thm}[Sz\'{e}kelyhidi \cite{gaborjdg}]
\label{thmgaborthm20}
Let $u:\,\mathbb{C}^n\to\mathbb{R}$ be a Lipschitz $\Gamma$-solution such that $|u|\leq C$ and $u$ has Lipschitz constant bounded by $C$. Then $u$ is a constant function.
\end{thm}
\begin{proof}[Proof of Theorem \ref{thm1order}]
We use the blowup argument in \cite{gaborjdg,twjams} originated from \cite{dkajm}.
Assume for a contradiction that \eqref{1stest} does not hold. There exist a sequence of basic real $(1,1)$ form $\beta_{\dag,j}$ and basic smooth functions $\psi_j$ and $u_j$ such that
\begin{align}
&\|\beta_{\dag,j}\|_{C^2(M,g)}+\|\psi_{j}\|_{C^2(M,g)}\leq C\nonumber\\
\label{1ordercon}&\left[\inf_{(x,j)\in M\times\mathbb{N}^*}\psi_j,\sup_{(x,j)\in M\times\mathbb{N}^*}\psi_j\right]\subset\left(\sup_{\partial \Gamma}f,\sup_\Gamma f\right),\\
\label{fequj}
F(A_{\dag,j})=&\psi_j,\quad \text{with}\quad C_j:=\sup_M|\nabla u_j|_{\omega_\dag}\to+\infty,\quad \text{as}\quad j\to +\infty,
\end{align}
where \begin{equation*}
\sup_Mu_j=0,\quad\left(A_{\dag,j}\right)_r{}^s= (g_\dag)^{\overline{q}s}\left((\beta_{\dag,j})_{r\overline{q}}+(u_j)_{r\overline{q}}\right).
\end{equation*}
We also assume that $0$ is a transverse $\mathcal{C}$-subsolution of  equations given by \eqref{fequj}.

From Theorem \ref{thm0order}, it follows that $\sup_M|u_j|\leq C$. For each $j$, there exists $\mathbf{p}_j\in M$ such that $|\nabla u_j|(\mathbf{p}_j)=C_j$.
Without loss of generality, we assume that $\lim_{j\to \infty}\mathbf{p}_j=\mathbf{p}_0\in M$ and that $\mathbf{p}_0$ is the center of the foliated chart $(-\varepsilon_0,\varepsilon_0)\times B_{2}(0)\subset \mathbb{R}\times \mathbb{C}^n$ with all the points $\mathbf{p}_j\in (-\varepsilon_0/2,\varepsilon_0/2)\times B_{1}(\mathbf{0})$.  Now we just need consider $\omega_{\dag},\,\beta_{\dag,j},\,u_j,\,\psi_j,$ as quantities on $B_2(\mathbf{0})$. We also assume that $\mathbf{z}=(z_1,\cdots,z_n)$ is the coordinates on $\mathbb{C}^n$ and that $\omega_\dag(\mathbf{0})=\gamma$, where $\gamma$ is the standard Hermitian metric on $\mathbb{C}^n$.

We define
\begin{equation*}
\hat u_j(\mathbf{z}):=u_j(\mathbf{z}(\mathbf{p}_j)+\mathbf{z}/C_j),\quad \text{on}\quad B_{C_j}(\mathbf{0}),
\end{equation*}
which satisfies
\begin{equation*}
\sup_{B_{C_j}(\mathbf{0})}|\hat u_j|\leq C,\quad \text{and}\quad \sup_{B_{C_j}(\mathbf{0})}|\nabla \hat u_j|\leq C,
\end{equation*}
where the gradient  and norm are the Euclidean ones. Moreover, from the definition of these functions, it follows that
\begin{equation*}
 \partial_k \hat u_j (\mathbf{0})=C_j^{-1} \partial_k u_j(\mathbf{z}(\mathbf{p}_j)) ,\quad k=1,\cdots,n,
\end{equation*}
which yields that
\begin{equation*}
|\partial \hat u_j| (\mathbf{0})>c>0.
\end{equation*}
From Theorem \ref{thm2order}, it follows that
\begin{equation}
\sup_{B_{C_j}(\mathbf{0})}|\sqrt{-1}\partial_{\mathrm{B}}\overline{\partial}_{\mathrm{B}}\hat u_{j}|_{\gamma}\leq CC_j^{-2}\sup_M|\sqrt{-1}\partial_{\mathrm{B}}\overline{\partial}_{\mathrm{B}}u_j|_{\omega_\dag}\leq C.
\end{equation}
Thanks to the elliptic estimates for $\Delta_{\gamma}$ and the Sobolev embedding, we see that for each given compact set $K\subset \mathbb{C}^n$, each $\alpha\in (0,1)$ and $p>1$, there exists a constant $C$ such that
\begin{equation*}
\|\hat u_{j}\|_{C^{1,\alpha}(K)}+\|\hat u_{j}\|_{W^{2,p}(K)}\leq C.
\end{equation*}
This yields that there exists a subsequence of $\hat u_j$ that converges strongly in $C_{\mathrm{loc}}^{1,\alpha}(\mathbb{C}^n)$ as well as weakly in $W^{2,p}_{\mathrm{loc}}(\mathbb{C}^n)$ to a function $u\in C_{\mathrm{loc}}^{1,\alpha}(\mathbb{C}^n)\cap W^{2,p}_{\mathrm{loc}}(\mathbb{C}^n)$ with $\sup_{\mathbb{C}^n}(|u|+|\nabla u|)\leq C$ and $|\nabla u|(\mathbf{0})\geq c>0$. In particular, $u$ is non-constant.

We set
\begin{equation*}
\Phi_j:\,\mathbb{C}^n\to\mathbb{C}^n,\quad \mathbf{z}\mapsto C_j^{-1}\mathbf{z}+\mathbf{x}_j,\quad \mathbf{x}_j:=\mathbf{z}(\mathbf{p}_j).
\end{equation*}
Then we have
\begin{align*}
\hat u_j=&u_j\circ\Phi_j,\quad\text{on}\quad B_{C_j}(\mathbf{0}),\\
\gamma_j:=&C_j^2\Phi_j^*\omega_\dag\to \gamma,\quad\text{smoothly on compact set of $\mathbb{C}^n$ as}\quad j\to \infty.
\end{align*}
In particular, $\Phi_j^*\omega_\dag \to 0$ smoothly. Similarly, $\chi_j:=\Phi_j^*\beta_{\dag,j}\to 0$ smoothly. This also shows that
\begin{equation}
\label{eigenvaluematrix}
  \left|\lambda((\gamma_j)^{\bar q s}\left((\chi_j)_{r\overline{q}}+(\hat u_j)_{r\overline{q}}\right))-\lambda((\hat u_j)_{r\overline{s}})\right|\to 0,\quad\text{smoothly}.
\end{equation}
We rewrite \eqref{fequj} as
\begin{equation}
\label{fequj2}
F\left(C_j^{2}(\gamma_j)^{\bar q s}\left((\chi_j)_{r\overline{q}}+(\hat u_j)_{r\overline{q}}\right)\right)=\psi_j.
\end{equation}
We claim that $u$ is a $\Gamma$-solution.
Indeed, we first suppose that there exists a $C^2$ function $v$ such that $v\geq u$ and $v(\mathbf{z}_0)=u(\mathbf{z}_0)$ for some point $\mathbf{z}_0$. By the construction of $u$, for any $\epsilon>0$, there exists a large $N\in\mathbb{N}$ such that if $j\geq N$, then there exist $a_j,\,\mathbf{z}_j$ with $|a_j|<\epsilon,\,|\mathbf{z}_j-\mathbf{z}_0|<\epsilon$ such that
\begin{equation*}
v+\epsilon|\mathbf{z}-\mathbf{z}_0|+a_j\geq \hat u_j,\quad\text{on}\;B_1(\mathbf{z}_0),\;\text{with equality at $\mathbf{z}_j$},
\end{equation*}
and that $\lambda( (\hat u_j)_{k\overline{\ell}})$ lies in the $2\epsilon$ neighborhood of $\Gamma\supset\Gamma_n$ by \eqref{eigenvaluematrix} and \eqref{fequj2}.
This means that $v_{k\overline{\ell}}(\mathbf{z}_j)+\epsilon\delta_{k\ell}\geq (\hat u_j)_{k\overline{\ell}}(\mathbf{z}_j)$ and hence  $v_{k\overline{\ell}}(\mathbf{z}_j)+\epsilon\delta_{k\ell}$ lies in the $2\epsilon$ neighborhood of $\Gamma$, from which we can deduce that $\lambda((v_{k\overline{\ell}}(\mathbf{z}_0)))\in\overline{\Gamma}$ by letting $\epsilon\to0$.

We second suppose that $v$ is a $C^2$ function such that $v\leq u$ and $v(\mathbf{z}_0)=u(\mathbf{z}_0)$. As above, for any $\epsilon>0$, there exists $N_1\in\mathbb{N}$ sufficiently large such that for any $j>N_1$, we can find $a_j,\,\mathbf{z}_j$ with $|a_j|<\epsilon,\,|\mathbf{z}_j-\mathbf{z}_0|<\epsilon$ satisfying
\begin{equation*}
v-\epsilon|\mathbf{z}-\mathbf{z}_0|+a_j\leq \hat u_j,\quad\text{on}\;B_1(\mathbf{z}_0),\;\text{with equality at $\mathbf{z}_j$},
\end{equation*}
This yields that $\left(v_{k\overline{\ell}}(\mathbf{z}_j)-\epsilon\delta_{k\ell}\right)\leq \left((u_j)_{k\overline{\ell}}(\mathbf{z}_j)\right)$. If $\lambda\left(\left(v_{k\overline{\ell}}(\mathbf{z}_j)-3\epsilon\delta_{k\ell}\right)\right)\in \Gamma$, then
$\lambda\left(\left((u_j)_{k\overline{\ell}}(\mathbf{z}_j)\right)\right)\in \Gamma+2\epsilon\mathbf{1}$.

From \eqref{eigenvaluematrix}, it follows that
\begin{equation}
\lambda((\gamma_j)^{\bar q s}\left((\chi_j)_{r\overline{q}}+(\hat u_j)_{r\overline{q}}\right))\in \Gamma+\epsilon\mathbf{1},
\end{equation}
 and hence
\begin{equation*}
f\left(C_j^2(\gamma_j)^{\bar q s}\left((\chi_j)_{r\overline{q}}+(\hat u_j)_{r\overline{q}}\right)\right)>\sigma
\end{equation*}
for any $j>N_1$ ($N_1$ may be chosen larger if necessary), where  $\sigma\in\left(\sup_{(x,j)\in M\times\mathbb{N}^*}\psi_j ,\sup_\Gamma f\right)$ by Assertion \eqref{gaborlem91} in Lemma \ref{lem9}.
This is a contradiction to \eqref{fequj2}. Finally, we deduce that
$\lambda((v_{k\overline{\ell}}(\mathbf{z}_j)))\in \mathbb{R}^n\setminus(\Gamma+3\epsilon\mathbf{1})$
 and hence $\lambda((v_{k\overline{\ell}}(\mathbf{z}_0)))\in \mathbb{R}^n\setminus \Gamma$
by letting $\epsilon\to0.$

Now we get a non-constant Lipschitz $\Gamma$-solution $u$ since $|\nabla u|(\mathbf{0})\geq c>0$, which is a contradiction to Theorem \ref{thmgaborthm20}. This contradiction yields the desired \eqref{1stest}.
\end{proof}
\begin{proof}
  [Proof of Theorem \ref{mainthm}] In the foliated local coordinate patch $(-\varepsilon_0,\varepsilon_0)\times B_2(\mathbf{0})$,  we work in $B_2(\mathbf{0})$. Therefore, given \eqref{equ0order}, \eqref{1stest} and \eqref{equ2order}, the $C^{2,\alpha}$ estimate follows from the Evans-Krylov theory (see \cite{twwycvpde,chucvpde}).
\end{proof}
\section{Applications}\label{secapp}
In this section, we prove Corollary \ref{corjia1}, Corollary \ref{corjia2}, Corollary \ref{cor1} and Corollary \ref{cor2} by the method from \cite{gaborjdg,twjams10,twjams,twcrelle}.
\begin{proof}
[Proof of Corollary \ref{corjia1}]
We use the continuity method modified from \cite{twjams10,twjams,twcrelle}. Here for convenience we give all details and later we can be sketch. Fix a basic function $F\in C_{\mathrm{B}}^\infty(M,\mathbb{R})$  to find $(u,b)\in C_{\mathrm{B}}^\infty(M,\mathbb{R})\times \mathbb{R}$
such that
\begin{equation}
\label{twhesstcmamodify}
\left(\omega_h+\frac{1}{n-1}\left[\left(\Delta_{\mathrm{B}}u\right)\omega_\dag
-\sqrt{-1}\partial_{\mathrm{B}}\overline{\partial}_{\mathrm{B}}u \right] \right)^k\wedge\omega_\dag^{n-k}\wedge\eta=e^{F+b}
\omega_h^n\wedge\eta,
\end{equation}
with
\begin{equation}
\label{tildeomegath}
\omega_h+
\frac{1}{n-1}\left[\left(\Delta_{\mathrm{B}}u \right)\omega_\dag
-\sqrt{-1}\partial_{\mathrm{B}}\overline{\partial}_{\mathrm{B}}u  \right]>_{\mathrm{b}}0,
\quad \sup_Mu =0
\end{equation}
for transverse positive basic real $(1,1)$ forms $\omega_h$ and $\omega_\dag$ with $\mathrm{d}_{\mathrm{B}}\omega_\dag=0.$
Note the \eqref{twhesstcmamodify} is slightly different from \eqref{twhesstcma} with $e^F=e^G\frac{\omega_\dag^n\wedge \eta}{\omega_h^n\wedge \eta}.$

We consider the family of equations for $(u_t,b_t)\in C_{\mathrm{B}}^{2,\alpha}(M,\mathbb{R})\times \mathbb{R}$ with $t\in[0,1]$
\begin{equation}
\label{twhesstcmat}
\left(\omega_h+\frac{1}{n-1}\left[\left(\Delta_{\mathrm{B}}u_t\right)\omega_\dag
-\sqrt{-1}\partial_{\mathrm{B}}\overline{\partial}_{\mathrm{B}}u_t \right] \right)^k\wedge\omega_\dag^{n-k}\wedge\eta=e^{tF+b_t}
\omega_h^n\wedge\eta,
\end{equation}
with
\begin{equation}
\label{tildeomegat}
\omega_h+
\frac{1}{n-1}\left[\left(\Delta_{\mathrm{B}}u_t\right)\omega_\dag
-\sqrt{-1}\partial_{\mathrm{B}}\overline{\partial}_{\mathrm{B}}u_t \right]>_{\mathrm{b}}0,
\quad \sup_Mu_t=0.
\end{equation}
We define a set
\begin{equation*}
\mathscr{I}:=\left\{t'\in [0,1]:\exists\,(u_t,b_t)\in C_{\mathrm{B}}^{2,\alpha}(M,\mathbb{R})\times \mathbb{R}\,\text{solves}\,\eqref{twhesstcmat}\,\text{and}\,\eqref{tildeomegat}\,\text{for}\,t\in[0,t']\right\}.
\end{equation*}
Note that $0\in \mathscr{I}.$ We need prove the openness of $\mathscr{I}.$ Assume that there exists a solution of \eqref{twhesstcmat} and  \eqref{tildeomegat} for $t=\hat t$. We write
\begin{equation*}
\hat\omega:=\omega_h+
\frac{1}{n-1}\left[\left(\Delta_{\mathrm{B}}u_{\hat t}\right)\omega_\dag
-\sqrt{-1}\partial_{\mathrm{B}}\overline{\partial}_{\mathrm{B}}u_{\hat t} \right].
\end{equation*}
It is sufficient to prove that, for some $\varepsilon>0,$ there exists $(v_t,b_t)\in C_{\mathrm{B}}^{2,\alpha}(M,\mathbb{R})\times \mathbb{R}$  for $t\in [\hat t,\hat t+\varepsilon)$ solves
\begin{equation*}
\left(\hat\omega +\frac{1}{n-1}\left[\left(\Delta_{\mathrm{B}}v_t\right)\omega_\dag
-\sqrt{-1}\partial_{\mathrm{B}}\overline{\partial}_{\mathrm{B}}v_t \right] \right)^k\wedge\omega_\dag^{n-k}\wedge\eta=e^{(t-\hat t)F+b_t-b_{\hat t}}
\hat\omega ^n\wedge\eta
\end{equation*}
with
\begin{equation*}
\hat\omega+
\frac{1}{n-1}\left[\left(\Delta_{\mathrm{B}}v_t\right)\omega_\dag
-\sqrt{-1}\partial_{\mathrm{B}}\overline{\partial}_{\mathrm{B}}v_t \right]>_{\mathrm{b}}0,
\quad  v_{\hat t}=0
\end{equation*}
for some $b_t\in \mathbb{R}.$ Indeed, given such a $(v_t,b_t)$, the function $u_t:=u_{\hat t}+v_{t}$ solves \eqref{twhesstcmat} with $\sup_M u_t=0$ by adding a time-depending constant, as desired.
Consider the linear differential operator (strongly transverse elliptic \cite{elk90})
\begin{equation}
\label{deflv}
L(v)=\frac{k\left[\left(\Delta_{\mathrm{B}}v\right)\omega_\dag
-\sqrt{-1}\partial_{\mathrm{B}}\overline{\partial}_{\mathrm{B}}v\right]\wedge
\hat\omega^{k-1}\wedge\omega_\dag^{n-k}\wedge\eta}{(n-1)\hat\omega^k\wedge\omega_\dag^{n-k}\wedge\eta}=:g^{\bar j i}v_{i\bar j}.
\end{equation}
Note that $\omega:=\sqrt{-1}g_{i\bar j}\mathrm{d}z_i\wedge \mathrm{d}\bar z_j>_{\mathrm{b}}0$. We claim that there exists a smooth function $\sigma\in C_{\mathrm{B}}^\infty(M,\mathbb{R})$ such that
 \begin{equation}
 \label{gauduchonhua}
 e^\sigma g^{\bar ji}\hat\omega^k\wedge\omega_\dag^{n-k}\wedge\eta=g_{G}^{\bar ji}\omega_G^n\wedge \eta,
 \end{equation}
 where $\omega_G:=(g_G)_{i\bar j}\mathrm{d}z_i\wedge\mathrm{d}\bar z_j$ is a transverse Gauduchon metric. Indeed, it follows from \cite[Theorem 3.10]{baragliahekmati2018} that there exists a smooth basic
function $\sigma'\in C_{\mathrm{B}}^\infty(M,\mathbb{R})$ such that $(g_G)_{i\bar j}:=e^{\sigma'}g_{i\bar j}$ is a transverse Gauduchon metric. Then set $\sigma:=-\sigma'+\log\frac{\omega_G^n\wedge\eta}{\hat\omega^k\wedge\omega_\dag^{n-k}\wedge \eta},$ as desired.
 Up to adding a constant to $\sigma$, we assume that
\begin{equation}
\label{sigmatiaojian}
\int_Me^\sigma\hat\omega^{k}\wedge\omega_\dag^{n-k}\wedge\eta=1.
\end{equation}
We can find $v_t\in C_{\mathrm{B}}^{2,\alpha}(M,\mathbb{R})$ for $t\in[\hat t,\hat t+\varepsilon)$ such that \begin{align*}
&\left(\hat\omega +\frac{1}{n-1}\left[\left(\Delta_{\mathrm{B}}v_t\right)\omega_\dag
-\sqrt{-1}\partial_{\mathrm{B}}\overline{\partial}_{\mathrm{B}}v_t \right] \right)^k\wedge\omega_\dag^{n-k}\wedge\eta\\
=&e^{(t-\hat t)F+c_t}\hat\omega^n\wedge\eta\left(\int_Me^\sigma \left(\hat\omega +\frac{1}{n-1}\left[\left(\Delta_{\mathrm{B}}v_t\right)\omega_\dag
-\sqrt{-1}\partial_{\mathrm{B}}\overline{\partial}_{\mathrm{B}}v \right] \right)^k\wedge\omega_\dag^{n-k}\wedge\eta\right),
\end{align*}
where $c_t\in \mathbb{R}$ such that
 $$\int_Me^{(t-\hat t)F+c_t}\hat\omega^n\wedge\eta=1.$$
We set
\begin{align*}
\mathscr{A}_1
:=&\left\{v\in C_{\mathrm{B}}^{2,\alpha}(M,\mathbb{R}):\;\int_Mve^\sigma\hat\omega^k\wedge\omega_\dag^{n-k}\wedge\eta=0,\,
\hat\omega +\frac{1}{n-1}\left[\left(\Delta_{\mathrm{B}}v\right)
-\sqrt{-1}\partial_{\mathrm{B}}\overline{\partial}_{\mathrm{B}}v \right] >_{\mathrm{b}}0\right\},\\
\mathscr{A}_2
:=&\left\{w\in C_{\mathrm{B}}^{\alpha}(M,\mathbb{R}):\;\int_Me^we^\sigma\hat\omega^k\wedge\omega_\dag^{n-k}\wedge\eta=1\right\}.
\end{align*}
Then we define $\Phi:\,\mathscr{A}_1\to\mathscr{A}_2$ by
\begin{align*}
\Phi(v):=&\log\frac{\left(\hat\omega +\frac{1}{n-1}\left[\left(\Delta_{\mathrm{B}}v\right)\omega_\dag
-\sqrt{-1}\partial_{\mathrm{B}}\overline{\partial}_{\mathrm{B}}v \right] \right)^k\wedge\omega_\dag^{n-k}\wedge\eta}{\hat\omega^n\wedge\eta}\\
&-\log\left(\int_Me^\sigma \left(\hat\omega +\frac{1}{n-1}\left[\left(\Delta_{\mathrm{B}}v\right)\omega_\dag
-\sqrt{-1}\partial_{\mathrm{B}}\overline{\partial}_{\mathrm{B}}v \right] \right)^k\wedge\omega_\dag^{n-k}\wedge\eta\right).
\end{align*}
We need find $v_t\in C_{\mathrm{B}}^{2,\alpha}(M,\mathbb{R})$ with $\Phi(v_t)=(t-\hat t)F+c_t.$
Since $\Phi(0)=0$, the inverse function theorem yields that it suffices to prove the invertibility of
$$
(D\Phi)_0:\;T_0\mathscr{A}_1\to T_0\mathscr{A}_2,
$$
where
\begin{align*}
T_0\mathscr{A}_1
=&\left\{\zeta\in C_{\mathrm{B}}^{2,\alpha}(M,\mathbb{R}):\;\int_M\zeta e^\sigma\hat\omega^k\wedge\omega_\dag^{n-k}\wedge\eta=0\right\},\\
T_0\mathscr{A}_2
=&\left\{g\in C_{\mathrm{B}}^{ \alpha}(M,\mathbb{R}):\;\int_Mge^\sigma\hat\omega^k\wedge\omega_\dag^{n-k}\wedge\eta=0\right\}
\end{align*}
denote the tangent spaces of $\mathscr{A}_1$ and $\mathscr{A}_2$ at $0.$  We have
\begin{equation*}
(D\Phi)_0(\zeta)=g^{\bar ji}\zeta_{i\bar j}-\int_Me^\sigma\left(g^{\bar ji}\zeta_{i\bar j}\right)\hat\omega^k\wedge\omega_\dag^{n-k}\wedge\eta=g^{\bar ji}\zeta_{i\bar j},
\end{equation*}
where for the first equality we use \eqref{deflv} and \eqref{sigmatiaojian}, and for the second equality we use \eqref{gauduchonhua}
to get
$$
\int_Me^\sigma\left(g^{\bar ji}\zeta_{i\bar j}\right)\hat\omega^k\wedge\omega_\dag^{n-k}\wedge\eta
=\int_M\left((g_G)^{\bar ji}\zeta_{i\bar j}\right)\omega_G^{n}\wedge\eta
=\int_Mn\sqrt{-1}\partial_{\mathrm{B}}\overline{\partial}_{\mathrm{B}}\zeta\wedge\omega_G^{n-1}\wedge\eta=0.
$$
It follows from the strong maximum principle that $(D\Phi)_0$ is injective.  To show that it is surjective, take $g\in T_0\mathscr{A}_2.$ The let $\zeta$ solve
\begin{equation*}
(g_G)^{\bar ji}\zeta_{i\bar j}=ge^\sigma\frac{\hat\omega^k\wedge\omega_\dag^{n-k}\wedge\eta}{\omega_G^n\wedge\eta},\quad
\int_M \zeta e^\sigma e^\sigma\hat\omega^k\wedge\omega_\dag^{n-k}\wedge\eta=0.
\end{equation*}
Such solution $\zeta$ exists since there holds $\int_Mge^\sigma\frac{\hat\omega^k\wedge\omega_\dag^{n-k}\wedge\eta}{\omega_G^n\wedge\eta}
\omega_G^n\wedge\eta=\int_Mge^\sigma\hat\omega^k\wedge\omega_\dag^{n-k}\wedge\eta=0$ and $\omega_G$ is a transverse Gauduchon metric. This, together with \eqref{gauduchonhua}, yields that
\begin{equation*}
\frac{\omega_G^n\wedge\eta}{\hat\omega^k\wedge\omega_\dag^{n-k}\wedge\eta}e^{-\sigma}
\left((g_G)^{\bar ji}\zeta_{i\bar j}\right)=g^{\bar ji}\zeta_{i\bar j}=(D\Phi)_0(\zeta)=g,
\end{equation*}
as desired. This establishes the openness of $\mathscr{I}.$

To show the closeness of $\mathscr{I},$ we need a priori estimate of $(u_t,b_t).$ The uniform bound of $b_t$ follows from the maximum principle. To get the uniform estimate of $u_t,$ we need write \eqref{twhesstcma} in the form of \eqref{equ1}. For this aim, we set
\begin{equation}
\beta_{\dag}=(\tr_{\omega_\dag}\omega_h)\omega_\dag-(n-1)\omega_h.
\end{equation}
Let $A_{\dag,u}$ be the transverse Hermitian endomorphism of $\nu(\mathcal{F}_\xi)$ with respect to $\omega_\dag$ defined by $\beta_{\dag,u},$ and let
\begin{align*}
T_k(\lambda(A_{\dag,u}))=&\sum_{\ell\not=k}\lambda_\ell(A_{\dag,u}),\quad 1\leq k\leq n,\\
\mathbf{T} (\lambda(A_{\dag,u}))=&(T_1(\lambda(A_{\dag,u})),\cdots,T_n(\lambda(A_{\dag,u}))).
\end{align*}
Then \eqref{twhesstcma} is rewritten as
\begin{equation}
f(\lambda(A_{\dag,u}))=\log\sigma_k(\mathbf{T} (\lambda(A_{\dag,u})))=G+b,
\end{equation}
where we recall that $\sigma_k$ is the $k^{\mathrm{th}}$ elementary symmetric polynomial defined on $\mathbb{R}^n.$ If $\omega_h$ is a transverse $k$-positive basic real $(1,1)$ form, then $f$ is defined on $\Gamma:=\mathbf{T}^{-1}(\Gamma_k)$ and $0$ is a transverse $\mathcal{C}$-subsolution.
Hence a priori estimates of $u_t$ follows from Theorem \ref{mainthm} and a bootstrapping argument. The solution to \eqref{twhesstcmat} and \eqref{tildeomegat} at $t=1$ solves \eqref{twhesstcma}.

Uniqueness of the solution $(u,b)\in C_{\mathrm{B}}^\infty(M,\mathbb{R})\times\mathbb{R}$ to \eqref{twhesstcmat} and \eqref{tildeomegat}. Suppose that $(u,b),\,(u',b')\in C_{\mathrm{B}}^\infty(M,\mathbb{R})\times\mathbb{R}$ solve \eqref{twhesstcmat} and \eqref{tildeomegat}. Write
\begin{align*}
\tilde\omega_\dag:=&\omega_h+
\frac{1}{n-1}\left[\left(\Delta_{\mathrm{B}}u\right)\omega_\dag
-\sqrt{-1}\partial_{\mathrm{B}}\overline{\partial}_{\mathrm{B}}u \right],\\
\tilde\omega_\dag':=&\omega_h+
\frac{1}{n-1}\left[\left(\Delta_{\mathrm{B}}u'\right)\omega_\dag
-\sqrt{-1}\partial_{\mathrm{B}}\overline{\partial}_{\mathrm{B}}u' \right].
\end{align*}
Then we have
$$
\frac{\left(\tilde\omega_\dag +
\frac{1}{n-1}\left[\left(\Delta_{\mathrm{B}}(u'-u)\right)\omega_\dag
-\sqrt{-1}\partial_{\mathrm{B}}\overline{\partial}_{\mathrm{B}}(u'-u) \right]\right)^k\wedge\omega_\dag^{n-k}\wedge\eta}{\tilde\omega_\dag^k\wedge\omega_\dag^{n-k}\wedge\eta}
=e^{b'-b}.
$$
Considering the points where $u'-u$ attains a maximum and minimum, we get $b=b'$ and hence
there holds
\begin{equation*}
\left(\tilde\omega_\dag +
\frac{1}{n-1}\left[\left(\Delta_{\mathrm{B}}(u'-u)\right)\omega_\dag
-\sqrt{-1}\partial_{\mathrm{B}}\overline{\partial}_{\mathrm{B}}(u'-u) \right]\right)^k\wedge\omega_\dag^{n-k}\wedge\eta
=\tilde\omega_\dag^k\wedge\omega_\dag^{n-k}\wedge\eta,
\end{equation*}
i.e.,
\begin{equation*}
\left[\left(\Delta_{\mathrm{B}}(u'-u)\right)\omega_\dag
-\sqrt{-1}\partial_{\mathrm{B}}\overline{\partial}_{\mathrm{B}}(u'-u) \right]
\wedge\sum_{i=0}^{k-1}\tilde\omega_\dag^i\wedge(\tilde\omega_\dag')^{k-1-i}\wedge\omega_\dag^{n-k}\wedge\eta=0.
\end{equation*}
Since $\sup_Mu=\sup_Mu'=0,$  the strong maximum principle yields that $u\equiv u',$ as desired.
\end{proof}
\begin{proof}
[Proof of  Corollary \ref{corjia2}] From Lemma \ref{lemjia} and \eqref{formula1}, we deduce that there exists a basic real $(1,1)$ form $\omega_0>_{\mathrm{b}}0$ such that
\begin{equation}
\label{tformula1}
\frac{1}{(n-1)!}\ast_\dag\left(\omega_0^{n-1}+\sqrt{-1}\partial_{\mathrm{B}}\overline{\partial}_{\mathrm{B}}\wedge\omega_\dag^{n-2}\right)
=\omega_h+\frac{1}{n-1}\left[\left(\Delta_{\mathrm{B}}u\right)\omega_\dag
-\sqrt{-1}\partial_{\mathrm{B}}\overline{\partial}_{\mathrm{B}}u \right].
\end{equation}
It follows from \eqref{twhesstcma} with $k=n$, \eqref{detchin-1}, \eqref{astdet}, \eqref{tformula1} that $\tilde \omega_u$ given by \eqref{twu}
\begin{equation}
\label{tformula2}
 \tilde\omega_u^n\wedge\eta=e^{
\frac{G+b}{n-1}}\omega_\dag^n\wedge\eta.
\end{equation}
Given $\Psi_\dag\in c_1(\nu(\mathcal{F}_\xi)),$ the basic $\partial_{\mathrm{B}}\overline{\partial}_{\mathrm{B}}$-lemma (see \cite{sasakigeo}) yields that there exists a basic function $G\in C_{\mathrm{B}}^\infty(M,\mathbb{R})$ such that
\begin{equation*}
\ric(\omega_\dag)-\frac{\sqrt{-1}}{n-1}\partial_{\mathrm{B}}\overline{\partial}_{\mathrm{B}}G=\Psi,
\end{equation*}
which, together with taking $-\sqrt{-1}\partial_{\mathrm{B}}\overline{\partial}_{\mathrm{B}}$ on both sides of \eqref{tformula2}, yields Corollary \ref{corjia2}.
\end{proof}
\begin{proof}
[Proof of Corollary \ref{cor1}]
The conclusion follows form the argument in Corollary \ref{corjia1} by replacing
$\omega_0+\frac{1}{n-1}\left[\left(\Delta_{\mathrm{B}}u\right)\omega_\dag
-\sqrt{-1}\partial_{\mathrm{B}}\overline{\partial}_{\mathrm{B}}u \right] $ with $\omega_h+\sqrt{-1}\partial_{\mathrm{B}}\overline{\partial}_{\mathrm{B}}u,$
and replacing $\sigma_k(\mathbf{T}_k(\lambda))$ with $\sigma_k(\lambda)$ wherever they occur.
\end{proof}
\begin{proof}
[Proof of Corollary \ref{cor2}]
We rewrite \eqref{thessquotient} in the form $F(A_\dag)=f(\lambda)=-c$, where $F$ is given by
\begin{equation*}
f(\lambda):=-\frac{\binom{n}{\ell}^{-1}\sigma_\ell(\lambda)}{\binom{n}{k}^{-1}\sigma_k(\lambda)},
\end{equation*}
which is concave and $\frac{\partial f}{\partial \lambda_i}>0$ for $1\leq i\leq n$ (see \cite{spruck,gaborjdg}).

We solve the family of equations
\begin{equation}
\label{thessquotientt}
t\frac{\left(\omega_h+\sqrt{-1}\partial_{\mathrm{B}}\overline{\partial}_{\mathrm{B}}u_t\right)^\ell\wedge\omega_\dag^{n-\ell}\wedge\eta}
{\left(\omega_h+\sqrt{-1}\partial_{\mathrm{B}}\overline{\partial}_{\mathrm{B}}u_t\right)^k\wedge\omega_\dag^{n-k}\wedge\eta}
+(1-t)
\frac{\omega_\dag^n\wedge\eta}
{\left(\omega_h+\sqrt{-1}\partial_{\mathrm{B}}\overline{\partial}_{\mathrm{B}}u_t\right)^k\wedge\omega_\dag^{n-k}\wedge\eta}
=c_t,
\end{equation}
for $t\in[0,1]$ and $c_1=c$.

We set
\begin{equation*}
\mathcal{T}:=\left\{t\in[0,1]:\;\text{Equation \eqref{thessquotientt} has a solution $u_t$ at time}\,t\right\}.
\end{equation*}
It follows from Corollary \ref{cor1} that $0\in \mathcal{T}$.

For the openness, the linearized operator $\tilde L$ of \eqref{thessquotientt} at time $t$ is given by
\begin{align*}
\tilde L(w)=&-t\frac{\sqrt{-1}\partial_{\mathrm{B}}\overline{\partial}_{\mathrm{B}}w\wedge\left(\omega_h+\sqrt{-1}\partial_{\mathrm{B}}\overline{\partial}_{\mathrm{B}}u_t\right)^{\ell-1}\wedge\omega_\dag^{n-\ell}\wedge\eta}
{\left(\omega_h+\sqrt{-1}\partial_{\mathrm{B}}\overline{\partial}_{\mathrm{B}}u_t\right)^k\wedge\omega_\dag^{n-k}\wedge\eta}\\
&-c_t\frac{\sqrt{-1}\partial_{\mathrm{B}}\overline{\partial}_{\mathrm{B}}w\wedge\left(\omega_h+\sqrt{-1}\partial_{\mathrm{B}}\overline{\partial}_{\mathrm{B}}u_t\right)^{k-1}\wedge\omega_\dag^{n-k}\wedge\eta}
{\left(\omega_h+\sqrt{-1}\partial_{\mathrm{B}}\overline{\partial}_{\mathrm{B}}u_t\right)^k\wedge\omega_\dag^{n-k}\wedge\eta},\quad \forall\;w\in C_{\mathrm{B}}^\infty(M,\mathbb{R}).
\end{align*}
Note that $-\tilde L$ is strongly  transverse  elliptic operator. It follows from   $\mathrm{d}_{\mathrm{B}}\omega_h=0$ that $-\tilde L$ is self-adjoint with respect to the volume form $\left(\omega_h+\sqrt{-1}\partial_{\mathrm{B}}\overline{\partial}_{\mathrm{B}}u_t\right)^k\wedge\omega_\dag^{n-k}\wedge\eta$. Since \begin{equation*}
  C_{\mathrm{B}}^{2,\alpha}(M,\mathbb{R})\times\mathbb{R}\to C_{\mathrm{B}}^{ \alpha}(M,\mathbb{R}),\quad (u,c)\mapsto -\tilde L(u)+c,
\end{equation*}
is surjective, the
 openness of $\mathcal{T}$ follows.

By integration on both sides of \eqref{thessquotientt} with respect to the volume $\left(\omega_h+\sqrt{-1}\partial_{\mathrm{B}}\overline{\partial}_{\mathrm{B}}u_t\right)^k\wedge\omega_\dag^{n-k}\wedge\eta,$
we get $c_t\geq tc$ for any $t\in[0,1]$. We rewrite \eqref{thessquotientt} as
\begin{equation*}
f_t(\lambda)=
-t\frac{\binom{n}{\ell}^{-1}\sigma_\ell(\lambda)}{\binom{n}{k}^{-1}\sigma_k(\lambda)}
-(1-t)\frac{1}{\binom{n}{k}^{-1}\sigma_k(\lambda)}=-c_t.
\end{equation*}
We claim that $\underline{u}=0$ is a transverse $\mathcal{C}$-subsolution. Indeed, let $\mu$ be the $n$-tuple of eigenvalues of $\lambda\left(\omega_h\right)$ with respect to $g_\dag$, and it is sufficient to check that for any $(n-1)$-tuple $\mu'$ of $\mu$, there holds
\begin{equation*}
\lim_{s\to+\infty}f_t((\mu',s))>-c_t,
\end{equation*}
which is equivalent to
\begin{equation*}
-t\frac{\binom{n}{\ell}^{-1}\sigma_{\ell-1}(\mu')}{\binom{n}{k}^{-1}\sigma_{k-1}(\mu')}>-c_t,
\end{equation*}
i.e.,
\begin{equation}
\label{tiaojian1}
nc_t\binom{n}{k}^{-1}\sigma_{k-1}(\mu')-nt\binom{n}{\ell}^{-1}\sigma_{\ell-1}(\mu') >0,
\end{equation}
by the argument for \eqref{conbd}.

At any given point, we choose the foliated local chart $(-\varepsilon_0,\varepsilon_0)\times B_2(\mathbf{0})$ such that $(g_\dag)_{i\bar j}=\delta_{ij}$ and $\omega_h$ is diagonal. Then we write the basic $(n-1,n-1)$ form
\begin{equation}
\label{tiaojian2}
kc_t\omega_h^{k-1}\wedge\omega_\dag^{n-k}-t\ell\omega_h^{\ell-1}\wedge\omega_\dag^{n-\ell}
\end{equation}
in the form of \eqref{tn-1}, and easily see that the left side of \eqref{tiaojian1} is the eigenvalues of the coefficient matrix of  the basic $(n-1,n-1)$ form given by \eqref{tiaojian2}.  This yields that \eqref{tiaojian1} is equivalent to the fact that is a strictly transverse positive $(n-1,n-1)$ form (cf.  \cite{songweinkove2008,gaborjdg}), where we recall that the transverse positivity is independent of the choice of local frames (cf. the concepts of positivity in \cite[Chapter \uppercase\expandafter{\romannumeral3}]{demaillybook1}).  Since $\omega_h$ is strictly transverse $k$-positive and $c_t\geq tc$, \eqref{tiaojian2} follows from \eqref{thessquotientcon}.  Then Theorem \ref{mainthm} gives uniform a priori estimates for any $t$ in the compact interval $[\epsilon_0,1]$ for any small $\epsilon_0>0$, as desired.
\end{proof}
It is a meaningful problem to find geometric conditions to ensure the existence of transverse $\mathcal{C}$-subsolution. For this aim, let us recall some concepts (see for example \cite{baragliahekmati2018,chl2018}). A local basic function $u\in C_{\mathrm{B}}^{\infty}(U,\mathbb{C})$ is called basic holomorphic if $\bar\partial_{\mathrm{B}}u=0.$ We denote by $\mathscr{O}_M$ the germ sheaf of local basic holomorphic functions. A subset $V$ of $M$ is called transverse analytic subvariety if for each point $\mathbf{p}\in V$ there exists local basic holomorphic functions $f_1,\cdots,f_r$ on $U$ such that $V\cap U=\{f_1=\cdots=f_r=0\}.$ All the local properties in complex analytic geometry can be extended to the setting of transverse analytic subvarieties.
Motivated by \cite{lejmiszekelyhidi2015,gaborjdg}, we hope the positive answer to the following
\begin{question}
\label{question}
Let $(M,\phi,\xi,\eta,g)$ be a compact Sasakian manifold with $\dim_{\mathbb{R}}M=2n+1$ $($$n\geq 2$$)$ and $\omega_\dag=\frac{1}{2}\mathrm{d}\eta=g(\phi\cdot,\cdot)$ as its transverse K\"ahler form, and let $\omega_h$ be a   closed strictly transverse $k$-positive basic $(1,1)$ form. Then we can find a strictly transverse $k$-positive basic real $(1,1)$ form $\omega_h'\in [\omega_h]\in H^{1,1}_{\mathrm{B}}(M,\mathbb{R})$ such that the real basic $(n-1,n-1)$ form given in \eqref{thessquotientcon} with $\omega_h'$ instead of $\omega_h$ if and only if for all transverse analytic subvarieties $V\subset M$ of dimension $p=n-\ell,\cdots,n-1$ we have
\begin{equation}
\int_V c\frac{k!}{(k-n+p)!}\omega_h^{k-n+p}\wedge\omega_\dag^{n-k}\wedge\eta
-\frac{\ell !}{(\ell-n+p)!}\omega_h^{\ell-n+p}\wedge\omega_\dag^{n-\ell}\wedge\eta>0.
\end{equation}
\end{question}
It seems that we can modify the method in \cite{collinszekelyhidijdg2017} to answer the question in the case $k=n,\ell=n-1$ on toric Sasakian manifolds.
\section{On the Transverse Complex Geometry}\label{secfoliation}
In this section, we point out that our argument works on a compact oriented, taut, transverse foliated manifold with complex codimension $n.$
Let us recall some preliminaries from \cite{baragliahekmati2018} quickly, and for a more detailed explanation we refer the reader to their original paper (cf.\cite{ton97}).

Let $M$ be a smooth manifold with $\dim_{\mathbb{R}}M=n+r$ and a foliation $\mathcal{F}$ of dimension $r.$ We denote by $L:=T\mathcal{F}$ the tangent distribution of the foliation and by $Q=TM/L$ the normal bundle. If $L$ is oriented,  then the foliation $\mathcal{F}$ is said to be tangentially oriented.

Let $g$ be a Riemannian metric on $M.$ Then $TM$ splits orthogonally as $TM=L\oplus L^{\bot}$ and there exists a smooth bundle isomorphism $\sigma:Q\to L^{\bot}$ splitting the exact sequence
\begin{equation*}
0\to L\to TM\to Q\to0,
\end{equation*}
i.e., satisfying $\pi\circ\sigma=\mathrm{Id}.$ The metric $g$ is a direct sum $g=g_L\oplus g_{L^{\bot}}.$ With
$g_Q:=\sigma^*g_{L^{\bot}},$ the splitting map $\sigma:(Q,g_Q)\to (L^\bot,g_{L^\bot})$ is a metric isomorphism. If $\mathcal{L}_\xi g_Q=0$ for any $\xi\in \Gamma(L),$ then $g$ is called bundle-like metric and $g_Q$ is said to be holonomy invariant (see \cite[Chapter 5]{ton97}).

A foliation $\mathcal{F}$ is said to be taut if $M$ admits a Riemannian metric $g$ such that every leaf of $\mathcal{F}$ is a minimal submanifold (e.g., the  K-contact manifold). We denote by $\chi_{\mathcal{F}}$ the canonical volume form associated to $g_L$ (i.e., the characteristic form of $\mathcal{F},$ see \cite{ton97}).
For this $\chi_{\mathcal{F}},$ we have  (see  \cite{rummler1979} or \cite[Formula(4.26)]{ton97}),
\begin{equation}
\label{dchif}
\mathrm{d}\chi_{\mathcal{F}}=\kappa \wedge \chi_{\mathcal{F}}+\varphi_0=\varphi_0,
\end{equation}
$\kappa$ is the mean curvature form associated to $g$ (see \cite[Formula (3.20)]{ton97}) and $\varphi_0\in\bigwedge^{r+1}$ satisfies
\begin{equation*}
\iota_{\xi_1}\cdots\iota_{\xi_r}\varphi_0=0,\quad \forall\,\xi_1,\cdots,\xi_r\in\Gamma(L).
\end{equation*}
\begin{lem}
[Basic Stokes' theorem \cite{baragliahekmati2018}]
Let $(M^{n+r},\mathcal{F})$ be a closed oriented manifold with a taut foliation $\mathcal{F}$ and
$\dim_{\mathbb{R}}\mathcal{F}=r.$  Then we have
\begin{equation}
\label{bst}
\int_M(\mathrm{d}_{\mathrm{B}}\varphi)\wedge\chi_{\mathcal{F}}=0,\quad \forall\;\varphi\in\Omega_{\mathrm{B}}^{n-1},
\end{equation}
where $\chi_{\mathcal{F}}$ satisfies \eqref{dchif}.
\end{lem}
\begin{proof}
See \cite[Proposition 2.1]{baragliahekmati2018}.
\end{proof}

Let $M$ be a foliated manifold with $\dim_{\mathbb{R}}M=2m+r$ such that the foliation $\mathcal{F}$ has real codimension $2n,$ and let $L=T\mathcal{F}$ be the tangent distributions to the foliation and $Q=TM/L$ the normal bundle. A transverse almost complex structure is an endomorphism $I:\,Q\to Q$ such that $I^2=-\mathrm{Id}.$
The transverse almost complex structure is called integrable if $M$ can be covered by foliated charts $U_\alpha=V_\alpha\times W_\alpha \subset \mathbb{R}^r\times \mathbb{C}^n,$  where $I_{\upharpoonright U_\alpha}$ is the natural complex structure on $\mathbb{C}^n.$

A transverse Hermitian structure on a Riemannian foliated manifold $M$ is a pair $(g,I)$, where $g$ is the bundle-like metric and $I$ is a transverse integrable complex structure such that $g(IY_1,IY_2)=g(Y_1,Y_2)$ for all $Y_1,\,Y_2\in \Gamma(Q).$ We define a real basic $(1,1)$ form $\omega_\dag(Y_1,Y_2)=g(IY_1,Y_2)$ for all $Y_1,\,Y_2\in \Gamma(Q).$

Replacing $L_\xi,\eta,\mathcal{F}_\xi,\nu(\mathcal{F}_\xi)$ by $L,\chi_{\mathcal{F}},\mathcal{F},Q$ respectively,  all the concepts on Sasakian manifolds, such as basic forms, basic exterior differential, transverse Hodge star operator, foliated holomorphic vector bundle and adapted Chern connection, basic Chern form/classes, also hold here. Note that
$\mathrm{d}_{\mathrm{B}}=\partial_{\mathrm{B}}+\bar\partial_{\mathrm{B}}$ on transverse complex manifolds. Also we can propose transverse fully nonlinear equation defined by \eqref{equ1} with $f$ satisfying Assumptions \eqref{assum1}, \eqref{assum2} and \eqref{assum3} in Section \ref{secintroduction}, and the transverse $\mathcal{C}$-subsolution and transverse admissible subsolution of \eqref{equ1} are the same.

We define  the \emph{basic Aeppli cohomology group} (cf. \cite{twcrelle})
$$
H_{\mathrm{A,b}}^{n-1,n-1}(M,Q):=\frac{\{\partial_{\mathrm{B}}\bar\partial_{\mathrm{B}}{\textrm-}\text{closed real basic }(n-1,n-1)\text{ forms}\}}{\left\{\partial_{\mathrm{B}}\gamma+\overline{\partial_{\mathrm{B}} \gamma}:\;\gamma \,\in\,\bigwedge_{\mathrm{B}} ^{n-2,n-1}(M)\right\}}.
$$
The basic Stokes' theorem \cite[Proposition 2.1]{baragliahekmati2018} (see \eqref{bst}) yields that this space is naturally in duality with the finite dimensional \emph{basic Bott-Chern cohomology group} with the nondegenerated pairing
$$
H_{\mathrm{A,b}}^{n-1,n-1}(M,Q)\otimes H_{\mathrm{BC,b}}^{1,1}(M,Q)\longrightarrow \mathbb{R}
$$
given by
$$
\left([\alpha]_{\mathrm{A,b}},[\beta]_{\mathrm{BC,b}}\right)\mapsto \int_M\alpha\wedge\beta\wedge\chi_{\mathcal{F}}.
$$
For any $u\in C_{\mathrm{B}}^{\infty}(M,\mathbb{R})$, we set
$$
\gamma:=\frac{\sqrt{-1}}{2}\bar\partial_{\mathrm{B}} u\wedge\chi^{n-2},
$$
where $\chi$ is a real basic $(1,1)$ form. Then we have
\begin{align}\label{betau}
\beta_{\dag,u}:=\partial_{\mathrm{B}}\gamma+\overline{\partial_{\mathrm{B}} \gamma}=\sqrt{-1}\partial_{\mathrm{B}}\bar\partial_{\mathrm{B}} u\wedge \chi^{n-2}+\Re\left(\sqrt{-1}\partial_{\mathrm{B}} u\wedge\bar\partial_{\mathrm{B}}(\chi^{n-2})\right).
\end{align}
$\beta_{\dag,u}$ is $\partial_{\mathrm{B}}\bar\partial_{\mathrm{B}}$-closed and $\bar\partial_{\mathrm{B}}\beta_{\dag,u}$ is $\partial_{\mathrm{B}}$-exact. Indeed, it is the $(n-1,n-1)$ part of the $\mathrm{d}_{\mathrm{B}}$-exact $(2n-2)$ form $\mathrm{d}_{\mathrm{B}} \left(\mathrm{d}_{\mathrm{B}}^cu\wedge\chi^{n-2}\right)$, where $$\mathrm{d}_{\mathrm{B}}^c=\frac{\sqrt{-1}}{2}(\bar\partial_{\mathrm{B}}-\partial_{\mathrm{B}})$$ with $\mathrm{d}_{\mathrm{B}}\mathrm{d}_{\mathrm{B}}^c=\sqrt{-1}\partial_{\mathrm{B}}\bar\partial_{\mathrm{B}}$.
Let $\alpha'$ and $\alpha''$ be strongly positive basic $(1,1)$ forms on $M$. Then
\begin{align*}
\sqrt{-1}\partial_{\mathrm{B}}\bar\partial_{\mathrm{B}} \left(\log\frac{\alpha'^n\wedge\chi_{\mathcal{F}}}{\alpha''^n\wedge\chi_{\mathcal{F}}}\right)\wedge \chi^{n-2}+\Re\left[\sqrt{-1}\partial_{\mathrm{B}}\left(\log\frac{\alpha'^n\wedge\chi_{\mathcal{F}}}
{\alpha''^n\wedge\chi_{\mathcal{F}}}\right)\wedge\bar\partial_{\mathrm{B}}(\chi^{n-2})\right],
\end{align*}
is well-defined  $\partial_{\mathrm{B}}\bar\partial_{\mathrm{B}}$-closed since
$$
\log\frac{\alpha'^n\wedge\chi_{\mathcal{F}}}{\alpha''^n\wedge\chi_{\mathcal{F}}}\in C_{\mathrm{B}}^{\infty}(M,\mathbb{R}).
$$
Let $(M^{2n+r},\mathcal{F},g,I)$ be a compact oriented, taut, transverse Hermitian foliated manifold, where $\mathcal{F}$ is the foliation with complex codimension $n$ and $g$ is the bundle-like metric. Then we split $g$ as $g=g_L\oplus g_{L^\bot}$ where $L=T\mathcal{F}$ and denote by $\chi_{\mathcal{F}}$ be the characteristic form satisfying \eqref{dchif}. We assume that $\omega_\dag=g(I\cdot,\cdot)$ is a transverse Gauduchon metric without loss of generality (see \cite[Theorem 3.10]{baragliahekmati2018}). We define the transverse Hodge star operator $\ast_\dag$ (see \cite[Formula (7.2)]{ton97}) as
\begin{equation*}
\ast_\dag\varphi=\ast(\chi_{\mathcal{F}}\wedge\varphi),\quad \forall\varphi\in \Omega_{\mathrm{B}}^p,
\end{equation*}
where $\ast$ is the ordinary Hodge star operator associated to $g.$
Now we can define a new transverse Hermitian metric $\tilde\omega_{\dag,u}$ on $M$ by
\begin{align}\label{tomegan-1}
\tilde\omega_{\dag,u}^{n-1}=&\omega_0^{n-1}
+\partial_{\mathrm{B}}\left(\frac{\sqrt{-1}}{2}\bar\partial_{\mathrm{B}} u \wedge\omega_\dag^{n-2}\right)
+\overline{\partial_{\mathrm{B}}\left(\frac{\sqrt{-1}}{2}\bar\partial_{\mathrm{B}} u \wedge\omega_\dag^{n-2}\right)}\\
=&\omega_0^{n-1}
+\sqrt{-1}\partial_{\mathrm{B}}\bar\partial_{\mathrm{B}} u \wedge\omega_\dag^{n-2}
+\Re\left(\sqrt{-1} \partial_{\mathrm{B}}u \wedge\bar\partial_{\mathrm{B}}(\omega_\dag^{n-2})\right)>_{\mathrm{b}}0,\nonumber
\end{align}
where $u \in C_{\mathrm{B}}^\infty(M,\mathbb{R})$ and recall that $>_{\mathrm{b}}0$ means strictly transverse positivity.  If $\omega_0$ is a transverse (strongly) Gauduchon metric, then so is  $\tilde\omega_{\dag,u}.$

We consider the transverse Gauduchon's question (cf.\cite[Chapter IV.5]{gauduchon2}): given a representative $\Psi\in c_1^{\mathrm{BC,b}}(Q),$ we hope to find a transverse Gauduchon metric $\tilde\omega_{\dag,u}$ with $\tilde\omega_{\dag,u}^{n-1}\in [\omega_\dag^{n-1}]_{\mathrm{A,b}}\in H_{\mathrm{A,b}}^{n-1,n-1}(M,Q)$ such that $\ric(\tilde\omega_{\dag,u})=\Psi.$ For this aim, we consider $\tilde\omega_{\dag,u}$ uniquely determined by \eqref{tomegan-1}, and it is sufficient to solve
\begin{equation*}
\tilde\omega_{\dag,u}^n=e^{\frac{G+b}{n}}\omega_\dag^n,\quad \sup_Mu=0,
\end{equation*}
where  $G\in C_{\mathrm{B}}^{\infty}(M,\mathbb{R})$ satisfies $\ric(\omega_\dag)=\Psi- \frac{\sqrt{-1} }{n} \partial_{\mathrm{B}}\bar\partial_{\mathrm{B}} G.$
It follows from \eqref{detchin-1} and \eqref{astdet} that it suffices to solve
$$
\log\frac{\det\left(\ast_{\dag}\frac{\tilde\omega_{\dag,u}^{n-1}}{(n-1)!}\right)}{\det\omega_\dag}
=\log\left(\frac{\det\tilde\omega_{\dag,u}}{\det\omega_\dag}\right)^{n-1}=G+b,\quad \sup_Mu=0,
$$
i.e.,
\begin{align}
\label{tstw1503equ}
\log\frac{\left(\omega_h+\frac{1}{n-1}\left[(\Delta_{\mathrm{B}} u)\omega_\dag-\sqrt{-1}\partial_{\mathrm{B}}\bar\partial_{\mathrm{B}} u\right]+Z(u)\right)^n\wedge\chi_{\mathcal{F}}}{\omega_\dag^n\wedge\chi_{\mathcal{F}}}=G+b,\quad \sup_Mu=0,
\end{align}
where
$
\omega_h=\frac{1}{(n-1)!}\ast_{\dag}\omega_0^{n-1},
$
$
\Delta_{\mathrm{B}} u=\frac{\sqrt{-1}\partial_{\mathrm{B}}\bar\partial_{\mathrm{B}}u\wedge\omega_\dag^{n-1}\wedge\chi_{\mathcal{F}}}
{\omega_\dag^n\wedge\chi_{\mathcal{F}}},
$
\begin{align}
\label{ttildeomega}
\omega_h+\frac{1}{n-1}\left[(\Delta_{\mathrm{B}} u)\omega_\dag-\sqrt{-1}\partial_{\mathrm{B}}\bar\partial_{\mathrm{B}} u\right]+Z(u)>_{\mathrm{b}}0,
\end{align}
and
\begin{align}
\label{zu}
Z(u)=\frac{1}{(n-1)!}\ast_\dag\Re\left[\sqrt{-1}\partial_{\mathrm{B}} u\wedge\bar\partial_{\mathrm{B}}(\alpha^{n-2})\right].
\end{align}
\begin{thm}
\label{thmfoliation}
Let $(M^{2n+r},\mathcal{F},I)$ be a compact oriented, taut, transverse Hermitian foliated manifold, where $\mathcal{F}$ is the foliation with complex codimension $n$ and  $I$ is the integrable transverse almost complex structure on $TM/(T\mathcal{F}).$ Then for any $G\in C_{\mathrm{B}}^\infty(M,\mathbb{R})$, transverse Hermitian metric $\alpha_0$ and transverse Gauduchon metric $\alpha,$ there exists a unique pair $(u,b)\in C_{\mathrm{B}}^\infty(M,\mathbb{R})\times\mathbb{R}$ solving \eqref{tstw1503equ}. In particular, given a representative $\Psi\in c_1^{\mathrm{BC,b}}(Q),$ there exists a unique $\omega$ uniquely determined by  \eqref{tomegan-1} such that $\tilde\omega_{\dag,u}^{n-1}\in [\alpha^{n-1}]_{\mathrm{A,b}}\in H_{\mathrm{A,b}}^{n-1,n-1}(M,Q)$ with $\ric(\tilde\omega_{\dag,u})=\Psi.$ Moreover, $c_1^{\mathrm{BC,b}}(Q)=0$ holds if and only if there exist basic Chern-Ricci flat transverse Gauduchon metrics on $M.$
\end{thm}
This theorem is a transverse version of the Gauduchon conjecture solved by \cite{stw1503}.
\begin{proof}[Proof of Theorem \ref{thmfoliation}]We just need prove the existence and uniqueness of the solution to \eqref{tstw1503equ}.
 In the foliated local coordinate patch $(U;x_1,\cdots,x_r,z_1,\cdots,z_n),$ it is easy to see that Equation \eqref{tstw1503equ}, as an example of  \eqref{equ1} with $\psi=G+b$, is the same as the one in \cite[Theorem 1.4]{stw1503}.
Hence this theorem is a transverse version of  \cite[Theorem 1.4]{stw1503}, and we can use the  calculations and estimates at some fixed point there directly.  Here we just point out some difference. We fix a characteristic form $\chi_{\mathcal{F}}$ of $\mathcal{F}$ satisfying \eqref{dchif}. Note that $\underline{u}=0$ is the transverse subsolution to \eqref{tstw1503equ}. The proof splits into two steps.

\textbf{Step 1}: A priori estimates. This step also splits the following five sub-steps.

\textbf{Step 1.1}: The maximum principle yields that $|b|\leq \sup_M|G|+C.$

\textbf{Step 1.2}: $L^\infty$ estimate
\begin{equation}
\label{foliation0}
\sup_M|u|\leq C.
\end{equation}
For $p$ sufficiently large, it follows from \eqref{bst} that
\begin{align}
\label{equiii}
&\int_Me^{-pu}\sqrt{-1}\partial_{\mathrm{B}}\bar\partial_{\mathrm{B}} u \wedge \left(\omega_0^{n-1}+\tilde\omega_{\dag,u}^{n-1}\right)\wedge\chi_{\mathcal{F}}\\
=&p\int_Me^{-pu}\sqrt{-1}\partial_{\mathrm{B}}u\wedge\bar\partial_{\mathrm{B}}u
\wedge\left(\omega_0^{n-1}+\tilde\omega_{\dag,u}^{n-1}\right)\wedge\chi_{\mathcal{F}}\nonumber\\
&+\int_Me^{-pu} \sqrt{-1}\bar\partial_{\mathrm{B}} u\wedge \partial_{\mathrm{B}}  \left(\omega_0^{n-1}+\tilde\omega_{\dag,u}^{n-1}\right)\wedge\chi_{\mathcal{F}}\nonumber\\
=:&(I)+(II),\nonumber
\end{align}
where $\tilde\omega_{\dag,u}^{n-1}$ is given by \eqref{tomegan-1}. Since $\tilde\omega_{\dag,u}^{n-1}>_{\mathrm{b}}0,$ there exists a uniform constant $C$ such that
\begin{equation}
\label{equi}
(I)\geq \frac{1}{Cp}\int_M\sqrt{-1}\partial_{\mathrm{B}}e^{-\frac{pu}{2}}\wedge\bar\partial_{\mathrm{B}}e^{-\frac{pu}{2}}
\wedge\omega_\dag^{n-1}\wedge\chi_{\mathcal{F}}.
\end{equation}
 Since $e^{-pu}\sqrt{-1}\bar\partial_{\mathrm{B}}u=-\frac{1}{p}\sqrt{-1}\bar\partial_{\mathrm{B}}e^{-pu},$ we can deduce from \eqref{bst} that
 \begin{align}
 \label{equii}
 (II)=
 &\frac{1}{p}\int_Me^{-pu}\sqrt{-1}\bar\partial_{\mathrm{B}}\partial_{\mathrm{B}}
 \left(\omega_0^{n-1}+\tilde\omega_{\dag,u}^{n-1}\right)\wedge\chi_{\mathcal{F}}\\
 =&\frac{1}{p}\int_Me^{-pu}\sqrt{-1}\bar\partial_{\mathrm{B}}\partial_{\mathrm{B}}u
 \wedge\sqrt{-1}\bar\partial_{\mathrm{B}}\partial_{\mathrm{B}}\omega_\dag^{n-2} \wedge\chi_{\mathcal{F}}\nonumber\\
 &+\frac{2}{p}\int_Me^{-pu}\sqrt{-1}\bar\partial_{\mathrm{B}}\partial_{\mathrm{B}}\omega_0^{n-1}
 \wedge\chi_{\mathcal{F}}\nonumber\\
 =&\int_Me^{-pu}\sqrt{-1}\partial_{\mathrm{B}}u\wedge\bar\partial_{\mathrm{B}}u
 \wedge\sqrt{-1}\bar\partial_{\mathrm{B}}\partial_{\mathrm{B}}\omega_\dag^{n-2} \wedge\chi_{\mathcal{F}}\nonumber\\
 &+\frac{2}{p}\int_Me^{-pu}\sqrt{-1}\bar\partial_{\mathrm{B}}\partial_{\mathrm{B}}\omega_0^{n-1}
 \wedge\chi_{\mathcal{F}}\nonumber\\
 =&\frac{4}{p^2}\int_M\sqrt{-1}\partial_{\mathrm{B}}e^{-\frac{pu}{2}}
 \wedge\bar\partial_{\mathrm{B}}e^{-\frac{pu}{2}}
 \wedge\sqrt{-1}\bar\partial_{\mathrm{B}}\partial_{\mathrm{B}}\omega_\dag^{n-2} \wedge\chi_{\mathcal{F}}\nonumber\\
 &+\frac{2}{p}\int_Me^{-pu}\sqrt{-1}\bar\partial_{\mathrm{B}}\partial_{\mathrm{B}}\omega_0^{n-1}
 \wedge\chi_{\mathcal{F}}\nonumber\\
 \geq&-\frac{1}{2Cp}\int_M\sqrt{-1}\partial_{\mathrm{B}}e^{-\frac{pu}{2}}
 \wedge\bar\partial_{\mathrm{B}}e^{-\frac{pu}{2}}
 \wedge \omega_\dag^{n-1} \wedge\chi_{\mathcal{F}}\nonumber\\
 &-\frac{1}{C'p}\int_Me^{-pu} \omega_\dag^{n}
 \wedge\chi_{\mathcal{F}},\nonumber
 \end{align}
where we choose the constant $C$ the same as the one in \eqref{equi} and use the assumption that $p$ is sufficiently large.

On the other hand, using the foliated local coordinate patch $(U;x_1,\cdots,x_r,z_1,\cdots,z_n),$  it follows from  \eqref{formula1} and \cite[Lemma 5.1]{twcrelle} that
\begin{align}
&\sqrt{-1}\partial_{\mathrm{B}}\bar\partial_{\mathrm{B}} u \left(2\omega_0^{n-1}+\sqrt{-1}\partial_{\mathrm{B}}\bar\partial_{\mathrm{B}} u \wedge\omega_\dag^{n-2}+\Re\left(\sqrt{-1} \partial_{\mathrm{B}}u \wedge\bar\partial_{\mathrm{B}}(\omega_\dag^{n-2})\right)\right)\nonumber\\
\leq&C\left(1+|\partial_{\mathrm{B}}u|^2\right)\omega_\dag^n-\Re\left(\sqrt{-1} \partial_{\mathrm{B}}u \wedge\bar\partial_{\mathrm{B}}(\omega_\dag^{n-2})\right),\nonumber
\end{align}
which implies
\begin{align}
\label{e-pu}
&\int_Me^{-pu}\sqrt{-1}\partial_{\mathrm{B}}\bar\partial_{\mathrm{B}} u \left(2\omega_0^{n-1}+\sqrt{-1}\partial_{\mathrm{B}}\bar\partial_{\mathrm{B}} u \wedge\omega_\dag^{n-2}+\Re\left(\sqrt{-1} \partial_{\mathrm{B}}u \wedge\bar\partial_{\mathrm{B}}(\omega_\dag^{n-2})\right)\right)\wedge\chi_{\mathcal{F}}\\
\leq&C\int_Me^{-pu}\omega_\dag^n\wedge\chi_{\mathcal{F}}
+C\int_Me^{-pu}|\partial_{\mathrm{B}}u|^2\omega_\dag^n\wedge\chi_{\mathcal{F}}\nonumber\\
&-\int_Me^{-pu}\sqrt{-1}\partial_{\mathrm{B}}\bar\partial_{\mathrm{B}} u \wedge\Re\left(\sqrt{-1} \partial_{\mathrm{B}}u \wedge\bar\partial_{\mathrm{B}}(\omega_\dag^{n-2})\right)\wedge\chi_{\mathcal{F}}.\nonumber
\end{align}
From  \eqref{bst}, we can deduce
\begin{align}
\label{foliationguji1}
&-\int_Me^{-pu}\sqrt{-1}\partial_{\mathrm{B}}\bar\partial_{\mathrm{B}} u \wedge\Re\left(\sqrt{-1} \partial_{\mathrm{B}}u \wedge\bar\partial_{\mathrm{B}}(\omega_\dag^{n-2})\right)\wedge\chi_{\mathcal{F}}\\
=&-\Re\int_Me^{-pu}\sqrt{-1}\partial_{\mathrm{B}}\bar\partial_{\mathrm{B}} u \wedge \sqrt{-1} \partial_{\mathrm{B}}u \wedge\bar\partial_{\mathrm{B}}(\omega_\dag^{n-2}) \wedge\chi_{\mathcal{F}}\nonumber\\
=&\frac{1}{p}\Re\int_M\sqrt{-1} \partial_{\mathrm{B}}e^{-pu}\wedge\sqrt{-1}\partial_{\mathrm{B}}\bar\partial_{\mathrm{B}} u \wedge\bar\partial_{\mathrm{B}}(\omega_\dag^{n-2}) \wedge\chi_{\mathcal{F}}\nonumber\\
=&\frac{1}{p}\Re\int_M\sqrt{-1} \partial_{\mathrm{B}}e^{-pu}\wedge\sqrt{-1}\bar\partial_{\mathrm{B}} u \wedge\partial_{\mathrm{B}}\bar\partial_{\mathrm{B}}(\omega_\dag^{n-2}) \wedge\chi_{\mathcal{F}}\nonumber\\
=&-\frac{4}{p^2}\Re\int_M\sqrt{-1} \partial_{\mathrm{B}}e^{-\frac{pu}{2}}\wedge\sqrt{-1}\bar\partial_{\mathrm{B}} e^{-\frac{pu}{2}} \wedge\partial_{\mathrm{B}}\bar\partial_{\mathrm{B}}(\omega_\dag^{n-2}) \wedge\chi_{\mathcal{F}}\nonumber\\
\leq&\frac{C}{p^2}\Re\int_M\sqrt{-1} \partial_{\mathrm{B}}e^{-\frac{pu}{2}}\wedge\sqrt{-1}\bar\partial_{\mathrm{B}} e^{-\frac{pu}{2}} \wedge\omega_\dag^{n-1} \wedge\chi_{\mathcal{F}}.\nonumber
\end{align}
 Note that
\begin{equation}
\label{foliationguji2}
\int_Me^{-pu}|\partial_{\mathrm{B}}u|^2\omega_\dag^n\wedge\chi_{\mathcal{F}}
=\frac{4n}{p^2}\int_M\sqrt{-1}\partial_{\mathrm{B}}e^{-\frac{pu}{2}}
\wedge\bar\partial_{\mathrm{B}}e^{-\frac{pu}{2}}\wedge\omega_\dag^{n-1}\wedge\chi_{\mathcal{F}}.
\end{equation}
Thanks to \eqref{equiii}, \eqref{equi}, \eqref{equii}, \eqref{e-pu}, \eqref{foliationguji1} and \eqref{foliationguji2}, we get
\begin{equation}
 \int_M\left|\partial_{\mathrm{B}}e^{-\frac{pu}{2}}\right|^2\omega_\dag^n\wedge\chi_{\mathcal{F}}
\leq Cp\int_Me^{-pu}\omega_\dag^n\wedge\chi_{\mathcal{F}}.
\end{equation}
We set $v=u-\inf_Mu.$
Then from \eqref{ttildeomega} we have
\begin{align}
\label{deltabv}
\Delta_{\mathrm{B}}v
=&\frac{\sqrt{-1}\partial_{\mathrm{B}}\bar \partial_{\mathrm{B}}u\wedge\omega_\dag^{n-1}\wedge\chi_{\mathcal{F}}}{\omega_\dag^{n}\wedge\chi_{\mathcal{F}}}\\
=&\tr_{\omega_\dag}\left(\ast_{\dag}\frac{\tilde\omega_{\dag,u}^{n-1}}{(n-1)!}\right)
-\tr_{\omega_\dag}\omega_h-\tr_{\omega_\dag }Z(u)\nonumber\\
\geq&-\tr_{\omega_\dag}\omega_h-\tr_{\omega_\dag} Z(u)\nonumber\\
=&-\tr_{\omega_\dag}\omega_h-\frac{\frac{n}{(n-1)!}\ast_\dag\Re\left[\sqrt{-1}\partial_{\mathrm{B}} u\wedge\bar\partial_{\mathrm{B}}(\omega_\dag^{n-2})\right]\wedge\omega_\dag^{n-1}\wedge\chi_{\mathcal{F}}}
{\omega_\dag^{n}\wedge\chi_{\mathcal{F}}}\nonumber\\
=&-\tr_{\omega_\dag}\omega_h-\frac{\Re\left[\sqrt{-1}\partial_{\mathrm{B}} u\wedge\bar\partial_{\mathrm{B}}(\omega_\dag^{n-2})\right]\wedge\frac{n}{(n-1)!}\ast_\dag\omega_\dag^{n-1}\wedge\chi_{\mathcal{F}}}
{\omega_\dag^{n}\wedge\chi_{\mathcal{F}}}\nonumber\\
=&-\tr_{\omega_\dag}\omega_h-\frac{n\Re\left[\sqrt{-1}\partial_{\mathrm{B}} u\wedge\bar\partial_{\mathrm{B}}(\omega_\dag^{n-2})\right]\wedge \omega_\dag \wedge\chi_{\mathcal{F}}}
{\omega_\dag^{n}\wedge\chi_{\mathcal{F}}}\nonumber\\
=&-\tr_{\omega_\dag}\omega_h-\frac{n(n-2)\Re\left[\sqrt{-1}\partial_{\mathrm{B}} u\wedge\bar\partial_{\mathrm{B}}(\omega_\dag^{n-1})\right] \wedge\chi_{\mathcal{F}}}
{(n-1)\omega_\dag^{n}\wedge\chi_{\mathcal{F}}}\nonumber\\
\geq&-C-\frac{n(n-2)\Re\left[\sqrt{-1}\partial_{\mathrm{B}} u\wedge\bar\partial_{\mathrm{B}}(\omega_\dag^{n-1})\right] \wedge\chi_{\mathcal{F}}}
{(n-1)\omega_\dag^{n}\wedge\chi_{\mathcal{F}}}.\nonumber
\end{align}
From \eqref{bst} and \eqref{deltabv}, we get
\begin{align}
\label{iterationuse}
&\int_M\left|\partial_{\mathrm{B}}v^{\frac{p+1}{2}}\right|^2\omega_\dag^n\wedge\chi_{\mathcal{F}}\\
=&\frac{n(p+1)^2}{4}\int_M\sqrt{-1}v^{p-1}\partial_{\mathrm{B}}v\wedge\bar \partial_{\mathrm{B}} v
\wedge\omega_\dag^n\wedge\chi_{\mathcal{F}}\nonumber\\
=&\frac{n(p+1)^2}{4p}\int_M\sqrt{-1} \partial_{\mathrm{B}}v^p\wedge\bar \partial_{\mathrm{B}} v
\wedge\omega_\dag^n\wedge\chi_{\mathcal{F}}\nonumber\\
=&\frac{n(p+1)^2}{4p}\int_M v^p(\Delta_{\mathrm{B}}v)
\wedge\omega_\dag^n\wedge\chi_{\mathcal{F}}\nonumber\\
&+\frac{n(p+1) }{4p}\int_M \sqrt{-1}(\bar\partial_{\mathrm{B}}v^{p+1})
\wedge\left(\partial_{\mathrm{B}}\omega_\dag^n\right)\wedge\chi_{\mathcal{F}}\nonumber\\
=&\frac{(p+1)^2}{4p}\int_M v^p(\Delta_{\mathrm{B}}v)
\wedge\omega_\dag^n\wedge\chi_{\mathcal{F}}\nonumber\\
\leq&C\frac{(p+1)^2}{4p}\int_M v^p
\wedge\omega_\dag^n\wedge\chi_{\mathcal{F}}\nonumber\\
&+\frac{n(n-2)(p+1)^2}{4p(n-1)}\Re\int_Mv^p\sqrt{-1}\partial_{\mathrm{B}} v\wedge\bar\partial_{\mathrm{B}}(\omega_\dag^{n-1})\wedge\chi_{\mathcal{F}}\nonumber\\
=&C\frac{(p+1)^2}{4p}\int_M v^p
\wedge\omega_\dag^n\wedge\chi_{\mathcal{F}}
+\frac{n(n-2)(p+1) }{4p(n-1)}\Re\int_M \sqrt{-1}\partial_{\mathrm{B}} v^{p+1}\wedge\bar\partial_{\mathrm{B}}(\omega_\dag^{n-1})\wedge\chi_{\mathcal{F}}\nonumber\\
=&C\frac{(p+1)^2}{4p}\int_M v^p
\wedge\omega_\dag^n\wedge\chi_{\mathcal{F}},\nonumber
\end{align}
where we twice use the transverse Gauduchon condition $\partial_{\mathrm{B}}\bar\partial_{\mathrm{B}}\omega_\dag^{n-1}=0.$
Then the $L^\infty$ estimate \eqref{foliation0} follows from \eqref{iterationuse} and the Morser iteration.

 There is another method to obtain the $L^\infty$ estimate by the weak Harnack inequality \cite[Theorem 9.22]{gt1998} and the modified Alexandroff-Bakelman-Pucci maximum principle (see \cite[Proposition 11]{gaborjdg}).
Indeed, we assume that $M$ is covered by finite the foliated charts $U_i$'s diffeomorphism to $V_i\times B_{2}(\mathbf{0})\subset \mathbb{R}^r\times \mathbb{C}^n$ such that
$\{\frac{1}{2}U_i\}$ each of which is diffeomorphism to  $\frac{1}{2}V_i\times B_1(\mathbf{0})$ still covers $M.$ We work in the quantities of the complex variables in $B_{2}(\mathbf{0})$ and hence the upper bound for $\||u|^p\|_{L^1}$ follows from \eqref{deltabv} and the argument in the proof of \cite[Proposition 10]{gaborjdg}. Then we use the  he modified Alexandroff-Bakelman-Pucci maximum principle to prove the $L^\infty$ estimate (see the argument in the proof of Theorem \ref{thm0order}).

\textbf{Step 1.3}: $C^2$ estimate
\begin{equation}
\label{foliation2nd}
\sup_M|\partial_{\mathrm{B}}\bar\partial_{\mathrm{B}}u|\leq C\left(1+\sup_M|\partial_{\mathrm{B}}u|^2\right).
\end{equation}
Using maximum principle, the perturbation argument in the proof of Theorem \ref{thm2order}, and the foliated local coordinate patch $(U;x_1,\cdots,x_r,z_1,\cdots,z_n)$, we work with the qualities of $(z_1,\cdots,z_n).$ Hence the estimate in \eqref{foliation2nd} follows from the argument in the proof of \cite[Theorem 1.2]{stw1503}.

\textbf{Step 1.4}: $C^1$ estimate
\begin{equation}
\label{foliation1st}
\sup_M|\partial_{\mathrm{B}} u|\leq C.
\end{equation}
The  argument in the proof of Theorem \ref{thm1order} yields that using the foliated local coordinate patch $(U;x_1,\cdots,x_r,z_1,\cdots,z_n)$, we work with the qualities of $(z_1,\cdots,z_n).$ Hence \eqref{foliation1st} follows from the argument in the proof of \cite[Theorem 1.7]{twcrelle} (see also \cite[Theorem 5.1]{twjams} and \cite[Theorem 5.1]{zhengimrn}).

\textbf{Step 1.5}: $C^{2,\beta}$ estimate
\begin{equation}
\label{foliation2alpha}
\sup_M|u|_{C^{2,\beta}}\leq C,
\end{equation}
with some $\beta\in(0,1)$
and high order estimate
\begin{equation}
\label{foliationhigh}
\sup_M|u|_{C^k}\leq C_k.
\end{equation}
In the foliated local coordinate patch $(U;x_1,\cdots,x_r,z_1,\cdots,z_n)$, we work with the qualities of $(z_1,\cdots,z_n).$ Therefore, given \eqref{foliation0}, \eqref{foliation2nd} and \eqref{foliation1st}, the $C^{2,\alpha}$ estimate \eqref{foliation2alpha} follows from the Evans-Krylov theory (see \cite{twwycvpde,chucvpde}) and high order estimate \eqref{foliationhigh} follows from the bootstrapping argument.

\textbf{Step 2}: Existence and uniqueness of the solution to \eqref{tstw1503equ}. We use the continuity method.
Fix a basic function $F\in C_{\mathrm{B}}^\infty(M,\mathbb{R})$  to find $(u,b)\in C_{\mathrm{B}}^\infty(M,\mathbb{R})\times \mathbb{R}$
such that
\begin{equation}
\label{foliationwhesstcmamodify}
\left(\omega_h+\frac{1}{n-1}\left[\left(\Delta_{\mathrm{B}}u\right)\omega_\dag
-\sqrt{-1}\partial_{\mathrm{B}}\overline{\partial}_{\mathrm{B}}u \right] +Z(u)\right)^n\wedge\chi_{\mathcal{F}}=e^{F+b}
\omega_\dag^n\wedge\chi_{\mathcal{F}},
\end{equation}
with
\begin{equation}
\label{foliationtildeomegath}
\omega_h+
\frac{1}{n-1}\left[\left(\Delta_{\mathrm{B}}u \right)\omega_\dag
-\sqrt{-1}\partial_{\mathrm{B}}\overline{\partial}_{\mathrm{B}}u  \right]+Z(u)>_{\mathrm{b}}0,
\quad \sup_Mu =0
\end{equation}
for transverse positive basic real $(1,1)$ forms $\omega_h$ and $\omega_\dag$ with $\partial_{\mathrm{B}}\overline{\partial}_{\mathrm{B}}\omega_\dag^{n-1}=0.$
Note the \eqref{foliationtildeomegath} is slightly different from \eqref{tstw1503equ} with $e^F=e^G\frac{\omega_\dag^n\wedge \chi_{\mathcal{F}}}{\omega_h^n\wedge \chi_{\mathcal{F}}}.$
The conclusion, except the closeness given by \emph{Step 1},  follows from  the argument in Corollary \ref{corjia1} with $k=n$ by replacing
$$\omega_h+\frac{1}{n-1}\left[\left(\Delta_{\mathrm{B}}u\right)\omega_\dag
-\sqrt{-1}\partial_{\mathrm{B}}\overline{\partial}_{\mathrm{B}}u \right] $$ with $$\omega_h+\frac{1}{n-1}\left[(\Delta_{\mathrm{B}} u)\omega_\dag-\sqrt{-1}\partial_{\mathrm{B}}\bar\partial_{\mathrm{B}} u\right]+Z(u),$$
replacing \eqref{deflv} with
\begin{equation*}
L(\varphi)=\frac{n\left[\left(\Delta_{\mathrm{B}}\varphi\right)\omega_\dag
-\sqrt{-1}\partial_{\mathrm{B}}\overline{\partial}_{\mathrm{B}}\varphi+Z(\varphi)\right]\wedge
\left(\ast_{\dag}\frac{\tilde\omega_{\dag,u}^{n-1}}{(n-1)!}\right)^{n-1} \wedge\chi_{\mathcal{F}}}{(n-1)\left(\ast_{\dag}\frac{\tilde\omega_{\dag,u}^{n-1}}{(n-1)!}\right)^n\wedge \chi_{\mathcal{F}}}=:g^{\bar j i}\varphi_{i\bar j},
\end{equation*}
 and replacing $\eta$ with $\chi_{\mathcal{F}}$ wherever it occurs.
\end{proof}
The same argument as in the proof of Theorem \ref{thmfoliation} (for $L^\infty$ estimate we use the weak Harnack inequality) also allows us to find a transverse Gauduchon metric $\tilde\omega_{\dag,u}$ defined in \eqref{tomegan-1} which solves the equation
\begin{equation*}
\tilde\omega_{\dag,u}^k\wedge\omega_\dag^{n-k}\wedge\chi_{\mathcal{F}}
=e^{G+b}\omega_\dag^{n}\wedge\chi_{\mathcal{F}},\quad 1\leq k\leq n,\quad G\in C^\infty_{\mathrm{B}}(M,\mathbb{R}).
\end{equation*}
\begin{thm}
\label{thmfoliation2}
Let $(M^{2n+r},\mathcal{F},I)$ be a compact oriented, taut, transverse Hermitian foliated manifold, where $\mathcal{F}$ is the foliation with complex codimension $n$ and  $I$ is the integrable transverse almost complex structure on $TM/(T\mathcal{F}),$  and let  $\omega_h$ be a  strictly  transverse $k$-positive basic real $(1,1)$ form $($denoted by $\omega_h>_{\mathrm{b},k}0$$)$, i.e., the $n$-tuple $\lambda(\omega_h)$ of eigenvalues  of $\omega_h$ with respect to $\omega_\dag$ satisfies $\lambda\left( \omega_h  \right)\in \Gamma_k.$  Then Theorem \ref{mainthm} holds. In particular,
 \begin{enumerate}
 \item  given a basic function $G\in C_{\mathrm{B}}^\infty(M, \mathbb{R})$, there exists a unique pair $(u,b) \in C_{\mathrm{B}}^\infty(M,\mathbb{R})\times\mathbb{R}$  solving the equation
\begin{equation}
\label{foliationtwhesstcma}
\left(\omega_h+\frac{1}{n-1}\left[\left(\Delta_{\mathrm{B}}u\right)\omega_\dag-\sqrt{-1}\partial_{\mathrm{B}}\overline{\partial}_{\mathrm{B}}u \right] \right)^k\wedge\omega_\dag^{n-k}\wedge\chi_{\mathcal{F}}=e^{G+b}
\omega_\dag^n\wedge\chi_{\mathcal{F}},
\end{equation}
where  $\sup_Mu=0$ and
\begin{equation*}
\omega_h+\frac{1}{n-1}\left[\left(\Delta_{\mathrm{B}}u\right)\omega_\dag-\sqrt{-1}\partial_{\mathrm{B}}\overline{\partial}_{\mathrm{B}}u \right] >_{\mathrm{b},k}0.
\end{equation*}
\item given a basic function $G\in C_{\mathrm{B}}^\infty(M, \mathbb{R})$,  there exists a unique pair $(u,b) \in C_{\mathrm{B}}^\infty(M,\mathbb{R})\times\mathbb{R}$  solving the equation the transverse Hessian equation
\begin{equation}
\label{foliationhesscma}
\left(\omega_h+ \sqrt{-1}\partial_{\mathrm{B}}\overline{\partial}_{\mathrm{B}}u \right)^k\wedge\omega_\dag^{n-k}\wedge\chi_{\mathcal{F}}=e^{G+b}
\omega_\dag^n\wedge\chi_{\mathcal{F}},
\end{equation}
where  $\sup_Mu=0$ and
\begin{equation*}
\omega_h+ \sqrt{-1}\partial_{\mathrm{B}}\overline{\partial}_{\mathrm{B}}u>_{\mathrm{b},k}0.
\end{equation*}
Furthermore, the solution of \eqref{foliationhesscma} for $k=n$ yields that $c_1^{\mathrm{BC,b}}(Q)=0$ holds if and only if there exist basic Chern-Ricci flat transverse Hermitian metrics on $M$.
\item if $ \mathrm{d}_\mathrm{B}\omega_h=\mathrm{d}_\mathrm{B}\omega_\dag=0,$ then for any $\mathrm{d}_{\mathrm{B}}$-closed strictly transverse positive real $(1,1)$ form $\omega_0,$ there exists a unique  $u\in C^\infty_{\mathrm{B}}(M,\mathbb{R})$ solving the general transverse Hessian quotient equation
\begin{equation}
\label{foliationhessquotient}
\left(\omega_h+ \sqrt{-1}\partial_{\mathrm{B}}\overline{\partial}_{\mathrm{B}}u \right)^\ell\wedge\omega_\dag^{n-\ell}\wedge\chi_{\mathcal{F}}
=c \left(\omega_h+ \sqrt{-1}\partial_{\mathrm{B}}\overline{\partial}_{\mathrm{B}}u \right)^k\wedge\omega_\dag^{n-k}\wedge\chi_{\mathcal{F}},\quad 1\leq \ell<k\leq n,
\end{equation}
if
\begin{equation}
\label{foliationhessquotientcon}
kc\omega_h^{k-1}\wedge\omega_\dag^{n-k} -\ell\omega_h^{\ell-1}\wedge\omega_\dag^{n-\ell}>_{\mathrm{b}}0,
\end{equation}
where
$$
c=\frac{\int_M\omega_h^\ell\wedge\omega_\dag^{n-\ell}\wedge\chi_{\mathcal{F}}}{\int_M\omega_h^k
\wedge\omega_\dag^{n-k}\wedge\chi_{\mathcal{F}}}.
$$
 \end{enumerate}
 \end{thm}
The conclusion for Theorem \ref{mainthm} is a transverse version of \cite{gaborjdg}. Equation \eqref{foliationtwhesstcma}, \eqref{foliationhesscma} and \eqref{foliationhessquotient} are the transverse version of \cite{twcrelle}, \cite{sun2017,zhangpjm} and \cite{gaborjdg} respectively.
\begin{proof}[Proof of Theorem \ref{thmfoliation2}]
The $L^\infty$ estimate follows from the weak Harnack inequality \cite[Theorem 9.22]{gt1998} and the modified Alexandroff-Bakelman-Pucci maximum principle (see \cite[Proposition 11]{gaborjdg}).

For the $C^2$ estimate, using maximum principle, the perturbation argument in the proof of Theorem \ref{thm2order}, and the foliated local coordinate patch $(U;x_1,\cdots,x_r,z_1,\cdots,z_n)$, we work with the qualities of $(z_1,\cdots,z_n).$ Hence the $C^2$ estimate follows from the argument in the proof of \cite[Theorem 1.2]{stw1503} with $Z(u)\equiv0$.

All other arguments are the same as the ones in Sasakian case by replacing $\eta$ with $\chi_{\mathcal{F}}$ wherever it occurs.
\end{proof}

\end{document}